\documentclass[11pt]{amsart} \textwidth=14.5cm \oddsidemargin=1cm
\usepackage{amsmath, amstext, amsbsy, amssymb, amscd, mathrsfs}
\usepackage{amsmath}
\usepackage{amsxtra}
\usepackage{amscd}
\usepackage{amsthm}
\usepackage{amsfonts}
\usepackage{amssymb}
\usepackage{eucal}
\usepackage{color}
\usepackage[all]{xy}
\usepackage{bbding}
\usepackage{pmgraph}
\usepackage{stmaryrd}
\usepackage{graphicx}
\usepackage{bm}
\usepackage[enableskew,vcentermath]{youngtab}

\newtheorem{theorem}{Theorem}[section]

\newtheorem{lemma}[theorem]{Lemma}
\newtheorem{proposition}[theorem]{Proposition}
\newtheorem{corollary}[theorem]{Corollary}
\newtheorem{example}[theorem]{Example}
\theoremstyle{definition}
\newtheorem{definition}[theorem]{Definition}

\theoremstyle{remark}
\newtheorem{remark}[theorem]{Remark}

\definecolor{A}{rgb}{.75,1,.75}

\numberwithin{equation}{section}

\newcommand{\ch}{\text{ch} }
\newcommand{\C}{ \mathbb C}
\newcommand{\Cl}{{\mathcal Cl}}

\newcommand{\End}{{\rm End}}
\newcommand{\ga}{\mathfrak{g}}
\newcommand{\gl}{\mathfrak{gl} }

\newcommand{\h}{\mathfrak{h}}
\newcommand{\ind}{\text{Ind}}
\newcommand{\HC}{\mc{Hc}}
\newcommand{\Hom}{\text{Hom} }
\newcommand{\hf}{\frac12}
\newcommand{\la}{\lambda}
\newcommand{\La}{\Lambda}
\newcommand{\Ga}{\Gamma}

\newcommand{\mc}{\mathcal}
\newcommand{\mf}{\mathfrak}
\newcommand{\N}{\mathbb N}
\newcommand{\ov}{\overline}
\newcommand{\ev}[1]{{#1}_{\bar{0}}}
\newcommand{\od}[1]{{#1}_{\bar{1}}}
\newcommand{\Pdn}{\mathcal{P}_d(n)}

\newcommand{\Q}{\mathbb Q}
\newcommand{\qn}{\mf q (n)}

\newcommand{\sgn}{\text{sgn}}
\newcommand{\td}{\widetilde}
\newcommand{\tr}{\text{tr}}

\newcommand{\Z}{ \mathbb Z }

\newcommand{\al}{\alpha}
\newcommand{\ds}{\displaystyle}
\newcommand{\Haff}{\HC_n^{\rm aff}}
\newcommand{\uli}{\underline{i}}

{\vskip-\lastskip\medskip
  \noindent
  {\em #1.}\enspace
  }%
{\qed\par\medskip
  }

\title[Lectures on spin representation theory]
{Lectures on spin representation theory of symmetric groups}
\author[Wan and Wang]{Jinkui Wan and Weiqiang Wang}

\address{
Department of Mathematics,
Beijing Institute of Technology,
Beijing, 100081, P.R. China. }
\email{wjk302@gmail.com}

\address{Department of Mathematics, University of Virginia,
Charlottesville,VA 22904, USA.}
\email{ww9c@virginia.edu}

\begin{document}

\maketitle

\begin{abstract}
The representation theory of the symmetric groups is intimately
related to geometry, algebraic combinatorics, and Lie theory. The
spin representation theory of the symmetric groups was originally
developed by Schur. In these lecture notes, we present a coherent
account of the spin counterparts of several classical constructions
such as the Frobenius characteristic map, Schur duality, the
coinvariant algebra, Kostka polynomials, and Young's seminormal
form.
\end{abstract}

\setcounter{tocdepth}{1}
 \tableofcontents

\section{Introduction}

\subsection{}

The representation theory of symmetric groups has many connections
and applications in geometry, combinatorics and Lie theory. The
following classical constructions in representation theory of
symmetric groups over the complex field $\C$ are well known:
\begin{enumerate}
\item
The characteristic map and symmetric functions

\item
Schur duality

\item
The coinvariant algebra

\item
Kostka numbers and Kostka polynomials

\item
Seminormal form representations and Jucys-Murphy elements
%
%
%
\end{enumerate}
(1) and (2) originated in the work of Frobenius and  Schur, (3) was
developed by Chevalley (see also Steinberg \cite{S}, Lusztig
\cite{Lu1}, and Kirillov \cite{Ki}). The Kostka polynomials in (4)
have striking combinatorial, geometric and representation theoretic
interpretations, due to Lascoux, Sch\"utzenberger, Lusztig,
Brylinski, Garsia and Procesi \cite{LS, Lu2, Br, GP}. Young's
seminormal form construction of irreducible modules of symmetric
groups has been redone by Okounkov and Vershik \cite{OV} using
Jucys-Murphy elements.

Motivated by projective (i.e., spin) representation theory of finite
groups and in particular of symmetric groups $\mf S_n$, Schur
\cite{Sch} introduced a double cover $\td{\mf S}_n$ of $\mf S_n$:
$$
1 \longrightarrow \Z_2 \longrightarrow \widetilde{\mf S}_n
\longrightarrow \mf S_n \longrightarrow 1.
$$
Let us write $\Z_2 =\{1,z\}$. The spin representation theory of $\mf
S_n$, or equivalently, the representation theory of the spin group
algebra $\C \mf S_n^- =\C \widetilde{\mf S}_n/\langle z+1\rangle$,
has been systematically developed by Schur (see J\'ozefiak
\cite{Jo1} for an excellent modern exposition via a superalgebra
approach; also see Stembridge \cite{St}).

The goal of these lecture notes is to provide a systematic account
of the spin counterparts of the classical constructions (1)-(5)
above over $\C$. Somewhat surprisingly, several of these spin
analogues have been developed only very recently (see for example
\cite{WW2}). It is our hope that these notes will be accessible to
people working in algebraic combinatorics who are interested in
representation theory and to people in super representation theory
who are interested in applications.

In addition to the topics (1)-(5), there are spin counterparts of
several classical basic topics which are not covered in these
lecture notes for lack of time and space: the
Robinson-Schensted-Knuth correspondence (due to Sagan and Worley
\cite{Sag, Wor}; also see \cite{GJK} for connections to crystal
basis); the plactic monoid (Serrano \cite{Ser}); Young symmetrizers
\cite{Naz, Se2}; Hecke algebras \cite{Ol, JN, Wa1, Wa2}. We refer an
interested reader to these papers and the references therein for
details.

Let us explain the contents of the lecture notes section by section.

\subsection{}
In Section \ref{sec:spinHC}, we explain how Schur's original
motivation of studying the projective representations of the
symmetric groups leads one to  study the  representations of the
spin symmetric group algebras. It has become increasingly well known
(cf. \cite{Jo2, Se2, St, Ya} and \cite[Chap.~13]{Kle}) that the
representation theory of spin symmetric group (super)algebra $\C\mf
S_n^-$ is super-equivalent to its counterpart for Hecke-Clifford
(super)algebra $\HC_n = \Cl_n \rtimes \C\mf S_n$. We shall explain
such a super-equivalence in detail, and then we mainly work with the
algebra $\HC_n$, keeping in mind that the results can be transferred
to the setting for $\C\mf S_n^-$. We review the basics on
superalgebras as needed.

The Hecke-Clifford superalgebra $\HC_n$ is identified as a quotient of
the group algebra of a double cover $\widetilde{B}_n$ of the
hyperoctahedral group $B_n$, and this allows us to apply various
standard finite group constructions to the study of representation
theory of $\HC_n$. In particular, the split conjugacy classes for
$\widetilde{B}_n$ (due to Read \cite{Re}) are classified.

\subsection{}
It is well known that the Frobenius characteristic map serves as a
bridge to relate the representation theory of symmetric groups to
the theory of symmetric functions.

In Section~\ref{sec:spinch}, the direct sum $R^-$ of the
Grothendieck groups of $\HC_n$-$\mf{mod}$ for all $n$ is shown to
carry a graded algebra structure and a bilinear form. Following
J\'ozefiak \cite{Jo2}, we formulate a spin version of the Frobenius
characteristic map
$$
\ch^-: R^- \longrightarrow \Gamma_\Q
$$
and establish its main properties, where $\Gamma_\Q$ is the ring of
symmetric functions generated by the odd power-sums. It turns out
that the Schur $Q$-functions $Q_\xi$ associated to strict partitions
$\xi$ play the role of Schur functions, and up to some $2$-powers,
they correspond to the irreducible $\HC_n$-modules $D^\xi$.

\subsection{}
The classical Schur duality relates the representation theory of the
general linear Lie algebras and that of the symmetric groups.

In Section~\ref{sec:duality}, we explain in detail the Schur-Sergeev
duality as formulated concisely in \cite{Se1}. A double centralizer
theorem for the actions of $\qn$ and the Hecke-Clifford algebra
$\HC_d$ on the tensor superspace $(\C^{n|n})^{\otimes d}$ is
established, and this leads to an explicit multiplicity-free
decomposition of the tensor superspace as a $U(\qn) \otimes
\HC_d$-module. As a consequence, a character formula for the simple
$\qn$-modules appearing in the tensor superspace is derived in terms
of Schur $Q$-functions. A more detailed exposition on materials covered
in Sections \ref{sec:spinch} and \ref{sec:duality} can be found in
\cite[Chapter 3]{CW}.

\subsection{}
The symmetric group $\mf S_n$ acts on $V=\C^n$ and then on the
symmetric algebra $S^*V$ naturally. A closed formula for the graded
multiplicity of a Specht module $S^\la$ for a partition $\la$ of $n$
in the graded algebra $S^*V$ in different forms has been well known
(see Steinberg \cite{S}, Lusztig \cite{Lu1} and Kirillov \cite{Ki}). More generally,
Kirillov and Pak \cite{KP} obtained the bi-graded multiplicity of
the Specht module $S^\la$ for any $\la$ in $S^*V \otimes \wedge^* V$
(see Theorem \ref{thm:KPak}), where $\wedge^* V$ denotes the
exterior algebra. We give a new proof here by relating this
bi-graded multiplicity to a $2$-parameter specialization of the
super Schur functions.

In Section \ref{sec:coinvariant}, we formulate a spin analogue of the
above graded multiplicity formulas. We present formulas with new
proofs for the (bi)-graded multiplicity of a simple $\HC_n$-module
$D^{\xi}$ in $\Cl_n\otimes S^*V, \Cl_n\otimes S^*V\otimes\wedge^*V$
and $\Cl_n\otimes S^*V\otimes S^*V$ in terms of various
specializations of the Schur $Q$-function $Q_{\xi}(z)$. The case of
$\Cl_n\otimes S^*V\otimes S^*V$ is new in this paper, while the
other two cases were due to the authors \cite{WW1}. The shifted hook
formula for the principal specialization $Q_{\xi}(1,t,t^2,\ldots)$
of $Q_{\xi}(z)$ was established by the authors \cite{WW1}
with a bijection proof and  in a different form  by Rosengren
\cite{Ro} based on formal Schur function identities. Here we present yet a third proof.

\subsection{}
The Kostka numbers and Kostka(-Foulkes) polynomials are ubiquitous
in combinatorics, geometry, and representation theory. Kostka
polynomials  have positive integer coefficients (see \cite{LS} for a
combinatorial proof, and see \cite{GP} for a geometric proof).
Kostka polynomials also coincide with Lusztig's $q$-weight
multiplicity in finite-dimensional irreducible representations of
the general linear Lie algebra \cite{Lu2, Ka}, and these are
explained by a Brylinski-Kostant filtration on the weight spaces
\cite{Br}. More details can be found in  the book of Macdonald
\cite{Mac} and the survey paper \cite{DLT}.

In Section \ref{sec:spinKostka}, following a very recent work of the
authors \cite{WW2}, we formulate a notion of spin Kostka
polynomials, and establish their main properties including the
integrality and positivity as well as interpretations in terms of
representations of the Hecke-Clifford algebras and the queer Lie
superalgebras. The graded multiplicities in the spin coinvariant
algebra described in Section \ref{sec:coinvariant} are shown to be
special cases of spin Kostka polynomials. Our constructions
naturally give rise to formulations of the notions of spin
Hall-Littlewood functions and spin Macdonald polynomials.

\subsection{}

By studying the action of the Jucys-Murphy elements  on the
irreducible $\mf S_n$-modules, Okounkov and Vershik \cite{OV}
developed a new approach to the representation theory of symmetric
groups. In their approach, one can see the natural appearance of
Young diagrams and standard tableaux, and obtain in the end Young's
seminormal form. A similar construction for the degenerate affine
Hecke algebra associated to $\mf S_n$ has been obtained by
Cherednik, Ram and Ruff  \cite{Ch, Ram, Ru}.

In Section \ref{sec:seminormal}, we explain a recent approach to
Young's seminormal form construction for the (affine) Hecke-Clifford
algebra. The affine Hecke-Clifford algebra $\Haff$ introduced by
Nazarov \cite{Naz} provides a natural general framework for $\HC_n$.
Following the independent works of Hill, Kujawa and Sussan
\cite{HKS} and the first author \cite{Wan}, we classify and
construct the irreducible $\Haff$-modules on which the polynomial
generators in $\Haff$ act semisimply. A surjective homomorphism from
$\Haff$ to $\HC_n$ allows one to pass the results for $\Haff$ to
$\HC_n$, and in this way we obtain Young's seminormal form for
irreducible $\HC_n$-modules. This recovers a construction of Nazarov
\cite{Naz} and the main result of Vershik-Sergeev \cite{VS} who
followed more closely Okounkov-Vershik's approach.

\medskip

{\bf Acknowledgments.} This paper is a modified and expanded written
account of the 8 lectures given by the authors at the Winter School
on Representation Theory, held at Academia Sinica, Taipei, December
2010. We thank Shun-Jen Cheng for his hospitality and a very
enjoyable winter school. The paper was partially written up during
our visit to Academia Sinica in Taipei and NCTS (South) in Tainan,
from which we gratefully acknowledge the support and excellent
working environment. Wan's research is partially supported
by Excellent young scholars Research Fund of Beijing Institute of Technology.
Wang's research has been partially supported by
NSF. We thank the referee for his careful reading and helpful suggestions.

\section{Spin symmetric groups and Hecke-Clifford algebra}\label{sec:spinHC}

In this section, we formulate an equivalence between the spin
representation theory of the symmetric group $\mf S_n$ and the
representation theory of the Hecke-Clifford algebra $\HC_n$. The
algebra $\HC_n$ is then identified as a twisted group algebra for a
distinguished double cover $\widetilde{B}_n$ of the hyperoctahedral
group $B_n$. We classify the split conjugacy classes of
$\widetilde{B}_n$ and show that the number of simple $\HC_n$-modules
is equal to the number of strict partitions of $n$.

\subsection{From spin symmetric groups to $\HC_n$}\label{sec:HC}

The symmetric group $\mf S_n$ is generated by the simple reflections
$s_i =(i,i+1), 1\leq i,j\leq n-1,$ subject to the Coxeter relations:
\begin{align}
s_i^2=1,\quad s_is_j =s_js_i, \quad s_is_{i+1}s_i=s_{i+1}s_is_{i+1},
\quad|i-j|>1. \label{braid}
\end{align}

One of Schur's original motivations is the study of projective
representations $V$ of $\mf S_n$, which are homomorphisms $\mf S_n
\rightarrow PGL(V):=GL(V)/\C^*$ (see \cite{Sch}). By a sequence
of analysis and deduction, Schur showed the study of projective
representation theory (RT for short) of $\mf S_n$ is equivalent to
the study of (linear) representation theory of a double cover $\td
{\mf S}_n$:
$$
\text{Projective RT of }\mf S_n \Leftrightarrow \text{ (Linear) RT
of } \td {\mf S}_n
$$

A double cover $\td {\mf S}_n$ means the following short exact
sequence of groups (nonsplit for $n \ge 4$):
$$
1 \longrightarrow \{1,z\} \longrightarrow \td {\mf S}_n
\stackrel{\pi_n}{\longrightarrow} \mf S_n \longrightarrow 1.
$$
The quotient algebra $\C \mf S_n^- = \C \td {\mf S}_n/ \langle
z+1\rangle$ by the ideal generated by $(z+1)$ is call the {\em spin
symmetric group algebra}. The algebra $\C {\mf S}_n^-$ is an algebra
generated by
 $t_1,t_2,\ldots,t_{n-1}$ subject to the relations:
\begin{align*}
t_i^2=1,\quad t_it_{i+1}t_i=t_{i+1}t_it_{i+1},\quad
t_it_j&=-t_jt_i,\quad |i-j|> 1.
\end{align*}
(A presentation for the group $\td {\mf S}_n$ can be obtained from
the above formulas by keeping the first two relations and replacing
the third one by $t_it_j =zt_jt_i$.) $\mathbb{C}{\mf S}_n^-$ is naturally
a super (i.e., $\Z_2$-graded) algebra with each $t_i$ being odd, for
$1\leq i\leq n-1$.

By Schur's lemma, the central element $z$ acts as $\pm 1$ on a
simple $\td {\mf S}_n$-module. Hence we see that
$$
\text{RT of } \td {\mf S}_n \Leftrightarrow \text{ RT of } \mf S_n
\bigoplus \text{ RT of } \C \mf S_n^-
$$
Schur then developed systematically the spin representation theory
of $\mf S_n$ (i.e., the representation theory of $\C \mf S_n^-$). We
refer to J\'ozefiak \cite{Jo1} for an excellent modern exposition
based on the superalgebra approach.

The development since late 1980's by several authors shows that the
representation theory of $\C\mf S_n^-$ is ``super-equivalent" to
the representation theory of a so-called Hecke-Clifford algebra
$\HC_n$:
\begin{equation}  \label{eq:spinS=HC}
\text{RT of } \C \mf S_n^- \Leftrightarrow \text{ RT of } \HC_n
\end{equation}
We will formulate this super-equivalence precisely in the next
subsections.

\subsection{A digression on superalgebras}  \label{sec:superalg}
By a vector superspace we mean a $\Z_2$-graded space $V =\ev V
\oplus \od V$. A superalgebra $\mc A =\ev{\mc{A}} \oplus
\od{\mc{A}}$ satisfies $\mc{A}_i \cdot \mc{A}_j \subseteq
\mc{A}_{i+j}$ for $i,j \in\Z_2$. By an ideal $I$ and a module $M$
of a superalgebra $\mc{A}$ in these lecture notes, we always mean
that $I$ and $M$ are $\Z_2$-graded, i.e., $I=(I\cap \ev{\mc{A}})\oplus
(I\cap \od{\mc{A}})$, and $M=\ev M\oplus \od M$ such that
$\mc{A}_iM_j\subseteq M_{i+j}$ for $i, j\in\Z_2$.
For a superalgebra $\mc A$, we let $\mc A$-$\mf{mod}$ denote the
category of $\mc A$-modules (with morphisms of degree one allowed). This superalgebra
approach handles ``self-associated and associates of simple modules"
simultaneously in a conceptual way. There is a parity reversing
functor $\Pi$ on the category of vector superspaces (or module
category of a superalgebra): for a vector superspace $V=\ev
V\oplus\od V$, we let
$$
\Pi(V)=\ev{\Pi(V)}\oplus\od{\Pi(V)}, \quad \Pi(V)_i=V_{i+\bar{1}},
\forall i\in\Z_2.
$$
Clearly, $\Pi^2 =\text{I}.$

Given a vector superspace $V$ with both even and odd subspaces of
equal dimension  and given an odd automorphism $P$ of $V$ of order
$2$, we define the following subalgebra of the endomorphism
superalgebra $\End (V)$:
$$
Q(V) =\{ x \in \End (V) \mid x \text{ and } P \text{
super-commute}\}.
$$
In case when $V =\C^{n|n}$ and $P$ is the linear transformation in
the block matrix form
\begin{equation*}
\sqrt{-1}\begin{pmatrix}
0&I_n \\
-I_n &0\\
\end{pmatrix},
\end{equation*}
we write $Q(V)$ as $Q(n)$, which consists of $2n\times 2n$
matrices of the form:
\begin{equation*}
\begin{pmatrix}
a&b\\
b&a\\
\end{pmatrix},
\end{equation*}
where $a$ and $b$ are arbitrary $n\times n$ matrices, for $n \ge 0$.
Note that we have a superalgebra isomorphism $Q(V) \cong Q(n)$ by
properly choosing coordinates in $V$, whenever $\dim V =n|n$. A
proof of the following theorem can be found in J\'ozefiak  \cite{Jo}
or \cite[Chapter~3]{CW}.

\begin{theorem}[Wall]
There are exactly two types of finite-dimensional simple associative
superalgebras over $\C$: (1) the matrix superalgebra $M(m|n)$, which
is naturally isomorphic to the endomorphism superalgebra of
$\C^{m|n}$; (2) the superalgebra $Q(n)$.
\end{theorem}

The basic results of finite-dimensional semisimple (unital
associative) algebras over $\C$ have natural super generalizations
(cf.~\cite{Jo}). The proof is standard.

\begin{theorem}[Super Wedderburn's Theorem]   \label{th:Wedderburn}
A finite-dimensional semisimple superalgebra $\mc A$ is isomorphic
to a direct sum of simple superalgebras:
\begin{equation*}
\mc A \cong \bigoplus_{i=1}^m M(r_i|s_i) \oplus \bigoplus_{j=1}^q
Q(n_j).
\end{equation*}
\end{theorem}
A simple $\mc A$-module $V$ is annihilated by all but one such
summand. We say $V$ is {\em of type $\texttt M$} if this summand is
of the form $M(r_i|s_i)$ and {\em of type $\texttt Q$} if this
summand is of the form $Q(n_j)$. In particular, $\C^{r|s}$ is a
simple module of the superalgebra $M(r|s)$ of type $\texttt M$, and
$\C^{n|n}$ is a simple module of the superalgebra $Q(n)$.  These two
types of simple modules are distinguished by the following super
analogue of Schur's Lemma (see \cite{Jo}, \cite[Chapter~3]{CW} for a
proof).

\begin{lemma}  (Super Schur's Lemma) \label{lem:superSchur}
If $M$ and $L$ are simple modules over a finite-dimensional
superalgebra $\mc A$, then
\begin{equation*}
 \dim \Hom_{\mc A} (M, L) =\left \{
 \begin{array}{ll}
 1 & \text{ if } M \cong L \text{ is of type } \texttt M, \\
 2 & \text{ if } M \cong L \text{ is of type } \texttt Q,  \\
 0 & \text{ if } M \not\cong L.
 \end{array}
 \right.
\end{equation*}
\end{lemma}

\begin{remark}
It can be shown (cf. \cite{Jo}) that a simple module of type
$\texttt M$ as an ungraded module remains to be simple (which is
sometimes referred to as ``self-associated" in literature), and a
simple module of type $\texttt Q$ as an ungraded module is a direct
sum of a pair of nonisomorphic simples (such pairs are referred to
as ``associates" in literature).
\end{remark}

Given two associative superalgebras $\mathcal{A}$ and $\mathcal{B}$, the
tensor product  $\mathcal{A}\otimes\mathcal{B}$ is naturally a
superalgebra, with multiplication defined by
$$
(a\otimes b)(a'\otimes b')=(-1)^{|b| \cdot |a'|}(aa')\otimes (bb')
\qquad (a,a'\in\mathcal{A}, b,b'\in\mathcal{B}).
$$

If $V$ is an irreducible $\mathcal{A}$-module and $W$ is an
irreducible $\mathcal{B}$-module, $V\otimes W$ may not be
irreducible (cf. \cite{Jo}, \cite{BK}, \cite[Lemma~ 12.2.13]{Kle}).

\begin{lemma}\label{tensorsmod}
Let $V$ be an irreducible $\mathcal{A}$-module and $W$ be an
irreducible $\mathcal{B}$-module.
\begin{enumerate}
\item If both $V$ and $W$ are of type $\texttt{M}$, then
$V\otimes W$ is an irreducible
$\mathcal{A}\otimes\mathcal{B}$-module of type $\texttt{M}$.

\item If one of $V$ or $W$ is of type $\texttt{M}$ and the other
is of type $\texttt{Q}$, then $V\otimes W$ is an irreducible
$\mathcal{A}\otimes\mathcal{B}$-module of type $\texttt{Q}$.

\item If both $V$ and $W$ are of type $\texttt{Q}$, then
$V\otimes W\cong X\oplus \Pi X$ for a type $\texttt{M}$ irreducible
$\mathcal{A}\otimes\mathcal{B}$-module $X$.
\end{enumerate}
Moreover, all irreducible $\mathcal{A}\otimes\mathcal{B}$-modules
arise as components of $V\otimes W$ for some choice of irreducibles
$V,W$.
\end{lemma}
If $V$ is an irreducible $\mathcal{A}$-module and $W$ is an
irreducible $\mathcal{B}$-module, denote by $V\circledast W$ an
irreducible component of $V\otimes W$. Thus,
$$
V\otimes W=\left\{
\begin{array}{ll}
V\circledast W\oplus \Pi (V\circledast W), & \text{ if both } V \text{ and } W
 \text{ are of type }\texttt{Q}, \\
V\circledast W, &\text{ otherwise }.
\end{array}
\right.
$$

\begin{example}  \label{ex:Clifford}
The Clifford algebra $\Cl_n$ is the $\C$-algebra generated by $c_i
(1\le i \le n)$, subject to relations
\begin{equation}
c_i^2 =1, \quad c_i c_j = -c_j c_i  \; (i \neq j) . \label{clifford}
\end{equation}
Note that $\Cl_n$ is a superalgebra with each generator $c_i$ being
odd, and $\dim \Cl_n =2^n$.

For $n=2k$ even, $\Cl_n$ is isomorphic to a simple matrix
superalgebra $M(2^{k-1}|2^{k-1})$. This can be seen by constructing
an isomorphism $\Cl_2\cong M(1|1)$ directly via Pauli matrices, and
then using the superalgebra isomorphism
$$
\Cl_{2k} = \underbrace{\Cl_2 \otimes \ldots \otimes \Cl_2}_k.
$$
Note that $\Cl_1 \cong Q(1)$. For $n=2k+1$ odd, we have superalgebra
isomorphisms:
$$
\Cl_n \cong \Cl_1 \otimes \Cl_{2k} \cong Q(1) \otimes
M(2^{k-1}|2^{k-1}) \cong Q(2^k).
$$

So $\Cl_n$ is always a simple superalgebra, of type $\texttt M$ for
$n$ even and of type $\texttt Q$ for $n$ odd. The fundamental fact
that there are two types of complex Clifford algebras is a key to
Bott's reciprocity.
\end{example}

\subsection{A Morita super-equivalence}

The symmetric group $\mf S_n$ acts as automorphisms on the Clifford
algebra $\Cl_n$ naturally by permuting the generators $c_i$. We will
refer to the semi-direct product $\HC_n := \Cl_n \rtimes \C \mf S_n$
as the {\em Hecke-Clifford algebra}, where
\begin{equation}  \label{pc}
s_ic_i=c_{i+1}s_i, \; s_ic_{i+1} =c_is_i, \; s_ic_j=c_js_i,\quad
j\neq i, i+1.
\end{equation}
Equivalently, $\sigma c_i =c_{\sigma (i)} \sigma,$ for all $1\le i
\le n$ and $\sigma \in \mf S_n.$  The algebra $\HC_n$ is naturally a
superalgebra by letting each $\sigma \in \mf S_n$ be even and each
$c_i$ be odd.

Now let us make precise the super-equivalence \eqref{eq:spinS=HC}.

By a direct computation, there is a superalgebra isomorphism  (cf.
\cite{Se1, Ya}):
\begin{align}
 \label{map:isorm.HC}
\begin{split}
\C\mf S_n^-\otimes\Cl_n&\longrightarrow\HC_n \\
c_i&\mapsto c_i, \quad 1\leq i\leq n,    \\
t_j&\mapsto \frac{1}{\sqrt{-2}}s_j(c_j-c_{j+1}),\quad 1\leq j\leq n-1.
\end{split}
\end{align}

By Example~\ref{ex:Clifford}, $\Cl_n$ is a simple superalgebra.
Hence, there is a unique (up to isomorphism) irreducible
$\Cl_n$-module $U_n$, of type $\texttt M$ for $n$ even and of type
$\texttt Q$ for $n$ odd. We have $\dim U_n=2^{k}$ for $n=2k$ or
$n=2k-1$.
Then the two exact functors
\begin{eqnarray*}
\mathfrak{F}_n :=-\otimes U_n: & \C
\mf S_n^-\text{-}\mf{mod} \rightarrow\HC_n\text{-}\mf{mod},\notag\\
\mathfrak{G}_n :={\rm Hom}_{\Cl_n}(U_n,-): & \HC_n \text{-}\mf{mod}
\rightarrow\C\mf S_n^-\text{-}\mf{mod}  \notag
\end{eqnarray*}
define a Morita super-equivalence between the superalgebras $\HC_n$
and $\C\mf S_n^-$ in the following sense.

\begin{lemma}\cite[Lemma~9.9]{BK} \cite[Proposition~13.2.2]{Kle}\label{lem:two functors}
\begin{enumerate}
\item Suppose that $n$ is even. Then the two functors
$\mathfrak{F}_n$ and $\mathfrak{G}_n$ are equivalences of
categories with
\begin{align*}
\mathfrak{F}_n\circ\mathfrak{G}_n\cong{\rm id},\quad
\mathfrak{G}_n\circ\mathfrak{F}_n\cong{\rm id}.
\end{align*}
\item Suppose that $n$ is odd. Then
\begin{align*}
\mathfrak{F}_n\circ\mathfrak{G}_n\cong{\rm id}\oplus\Pi,\quad
\mathfrak{G}_n\circ\mathfrak{F}_n\cong{\rm id}\oplus\Pi.
\end{align*}
\end{enumerate}
\end{lemma}

\begin{remark}
The superalgebra isomorphism \eqref{map:isorm.HC}
and the Morita super-equivalence in Lemma \ref{lem:two functors}
have a natural generalization to any
finite Weyl group; see Khongsap-Wang \cite{KW} (and the symmetric
group case here is regarded as a type $A$ case).
\end{remark}

\subsection{The group $\td B_n$ and the algebra $\HC_n$}\label{sec:tdBn}

Let $\Pi_n$ be the finite group generated by $a_i$ ($i=1, \ldots,
n$) and the central element $z$ subject to the relations
\begin{equation}\label{E:rel1}
a_i^2=1, \quad z^2=1, \quad a_ia_j=za_ja_i \quad (i\neq j).
\end{equation}
The symmetric group $\mf S_n$ acts on $\Pi_n$ by $\sigma
(a_i)=a_{\sigma (i)}$, $\sigma \in \mf S_n$. The semidirect
product $\td B_n:=\Pi_n\rtimes \mf S_n $ admits a natural finite
group structure and will be called the {\it twisted
hyperoctahedral group}. Explicitly the multiplication in $\td B_n$
is given by
\begin{equation*}
(a, \sigma)(a', \sigma')=(a \sigma(a'),
 \sigma \sigma'), \qquad a, a'\in \Pi_n, \sigma, \sigma'\in \mf S_n.
\end{equation*}

Since $\Pi_n/\{1, z \} \simeq \Z_2^n$, the group $\td B_n$ is a
double cover of the hyperoctahedral group $B_n:=\Z_2^n\rtimes \mf
S_n$, and the order $|\td B_n|$ is $2^{n+1}n!$. That is, we have a
short exact sequence of groups
\begin{equation}   \label{eq:Bcover}
1 \longrightarrow \{1,z\} \longrightarrow \td B_n
\stackrel{\theta_n}{\longrightarrow} B_n \longrightarrow 1,
\end{equation}
with $\theta_n(a_i)=b_i$, where $b_i$ is the generator  of the $i$th
copy of $\Z_2$ in $B_n$. We define a $\mathbb Z_2$-grading on the
group $\td B_n$ by setting the degree of each $a_i$ to be $1$ and
the degree of elements in $\mf S_n$ to be $ 0$. The group $B_n$
inherits a $\Z_2$-grading from $\td B_n$ via the homomorphism
$\theta_n$.

The quotient algebra $\C \Pi_n/\langle z+1\rangle$ is isomorphic to
the  Clifford algebra $\Cl_n$ with the identification $\bar{a}_i
=c_i, 1\le i \le n$. Hence we have a superalgebra isomorphism:
\begin{equation}   \label{eq:HC=Bn-}
\C \td B_n /\langle z+1\rangle
\cong \HC_n.
\end{equation}
A  $\td B_n$-module on which  $z$ acts as $-1$ is called a {\em spin
 $\td B_n$-module}.
As a consequence of the isomorphism \eqref{eq:HC=Bn-}, we have the
following equivalence:
$$
\text{RT of } \HC_n \Leftrightarrow \text{ Spin RT of  } \td B_n
$$

\subsection{The split conjugacy classes for $B_n$}

Recall for a finite group $G$,  the number of simple $G$-modules
coincides with the number of conjugacy classes of $G$. The finite
group $B_n$ and its double cover $\td B_n$ defined in
\eqref{eq:Bcover} are naturally $\Z_2$-graded. Since elements in a
given conjugacy class of $B_n$ share the same parity
($\Z_2$-grading), it makes sense to talk about even and odd
conjugacy classes of $B_n$ (and $\td B_n$). One can show by using
the Super Wedderburn's Theorem~\ref{th:Wedderburn} that the number
of simple $\td B_n$-modules coincides with the number of {\em even}
conjugacy classes of $\td B_n$.

For a conjugacy class $\mc C$ of $B_n$, $\theta_n^{-1} (\mc{C})$ is
either a single conjugacy class of $\td B_n$ or it splits into two
conjugacy classes of $\td B_n$; in the latter case, $\mc C$  is
called a {\em split} conjugacy class, and either conjugacy class in
$\theta_n^{-1} (\mc{C})$ will also be called {\em split}. An element
$x \in B_n$ is called {\em split} if the conjugacy class of $x$ is
split. If we denote $\theta_n^{-1}(x) =\{\tilde x, z \tilde x\}$,
then $x$ is split if and only if $\tilde x$ is not conjugate to $z
\tilde x$. By analyzing the structure of the even center of $\C \td
B_n$ using the Super Wedderburn's Theorem~\ref{th:Wedderburn} and
noting that $\C \td B_n \cong \C B_n \oplus \HC_n$, one can show the
following \cite{Jo} (also see \cite[Chapter~3]{CW}).

\begin{proposition}  \label{squaretable}
\begin{enumerate}
\item The number of simple $\HC_n$-modules equals the
number of even split conjugacy classes of $B_n$.

\item The number of simple $\HC_n$-modules of type
$\texttt Q$ equals the number of odd split conjugacy classes of
$B_n$.
\end{enumerate}
\end{proposition}

Denote by $\mc P$ the set of all partitions and by $\mc P_n$ the set
of partitions of $n$. We denote by $\mathcal{SP}_n$ the set of all
strict partitions of $n$, and by $\mathcal{OP}_n$ the set of all odd
partitions of $n$. Moreover, we denote
$$
\mathcal{SP}  = \bigcup_{n \ge 0} \mathcal{SP}_n, \qquad
\mathcal{OP}  = \bigcup_{n \ge 0} \mathcal{OP}_n,
$$
and denote
\begin{eqnarray*}
\mc{SP}_n^+ &=& \{\la\in \mc{SP}_n \mid
\ell(\la) \text{ is even} \},
 \\
\mc{SP}_n^-  &=& \{\la\in \mc{SP}_n \mid  \ell(\la) \text{ is odd} \}.
 \end{eqnarray*}

The conjugacy classes of the group $B_n$ (a special case of a wreath
product) can be described as follows, cf. Macdonald \cite[I,
Appendix~B]{Mac}. Given a cycle $t= (i_1, \ldots, i_m)$, we call the
set $\{i_1, \ldots, i_m \}$ the {\em support} of $t$, denoted by
$\text{supp}(t)$. The subgroup $\Z_2^n$ of $B_n$ consists of
elements $b_I: =\prod_{i \in I} b_i$ for $I \subset \{1,\ldots,
n\}$. Each element $b_I\sigma\in B_n$ can be written as a product
(unique up to reordering) $b_I\sigma = (b_{I_1}\sigma_1)
(b_{I_2}\sigma_2)\ldots (b_{I_k} \sigma_k),$ where $\sigma \in \mf
S_n$ is a product of disjoint cycles $\sigma =\sigma_1 \ldots
\sigma_k$, and $I_a \subset \text{supp}(\sigma_a)$ for each $1\le a
\le k$. The {\it cycle-product} of each $b_{I_a}\sigma_a$ is defined
to be the element $\prod_{i\in I_a} b_i \in \Z_2$ (which can be
conveniently thought as a sign  $ \pm$). Let $m_i^+$ (respectively,
$m_i^-$) be the number of $i$-cycles of $b_I\sigma$ with associated
cycle-product being the identity (respectively, the non-identity).
Then  $\rho^+ =(i^{m_i^+})_{i \geq 1}$ and $\rho^- =(i^{m_i^-})_{i
\geq 1}$ are partitions such that $|\rho^+|+|\rho^-|=n$. The pair of
partitions $(\rho^+, \rho^-)$ will be called the {\it type} of the
element $b_I\sigma$.

The basic fact on the conjugacy classes of $B_n$ is that two
elements of $B_n$ are conjugate if and only if their types are the
same.

\begin{example}
Let $\tau =(1,2,3,4)(5,6,7)(8,9), \sigma =(1,3, 8,6)(2, 7,9)(4,5)
\in \mf S_{10}$. Both $x =( (+,+,+,-,+,+,+,-,+,-), \tau)$  and $y =(
(+,-,-,-,+,-,-,-,+,-), \sigma)$ in $B_{10}$ have the same type
$(\rho^+, \rho^-) =((3),  (4,2,1))$. Then $x$ is conjugate to $y$ in
$B_{10}$.
\end{example}

The even and odd split conjugacy classes of $B_n$ are classified by
Read \cite{Re} as follows. The proof relies on an elementary yet
lengthy case-by-case analysis on conjugation, and it will be skipped
(see \cite[Chapter~3]{CW} for detail).
\begin{theorem} \cite{Re}  \label{splitclass}
The conjugacy class $C_{\rho^+, \rho^-}$ in $B_n$ splits if and
only if

(1) For even $C_{\rho^+, \rho^-}$, we have
$\rho^+\in\mathcal{OP}_n$ and $\rho^-=\emptyset$;

(2) For odd $C_{\rho^+, \rho^-}$, we have $\rho^+=\emptyset$ and
$\rho^-\in \mc{SP}^-_n$.
\end{theorem}
For $\alpha\in\mathcal{OP}_n$ we let $\mc{C}_{\alpha}^+$ be the
split conjugacy class in $\td B_n$ which lies in
$\theta_n^{-1}(C_{\alpha, \emptyset})$ and contains a permutation in
$\mf S_n$ of cycle type $\alpha$. Then $z\mc{C}_{\alpha}^+$ is the
other conjugacy class in $\theta_n^{-1}(C_{\alpha, \emptyset})$,
which will be denoted by $\mc{C}_{\alpha}^-$. By \eqref{eq:HC=Bn-}
and Proposition~\ref{squaretable}, we can construct a (square)
character table $(\varphi_\alpha)_{\varphi,\alpha}$ for $\HC_n$
whose rows are simple $\HC_n$-characters $\varphi$ or equivalently,
simple spin $\td B_n$-characters (with $\Z_2$-grading implicitly
assumed), and whose columns are even split conjugacy classes
$\mc{C}_{\alpha}^+$ for $\alpha \in \mc{OP}_n$.

Recall the Euler identity that $|\mc{SP}_n| =|\mc{OP}_n|$. By
Proposition~\ref{squaretable} and Theorem~\ref{splitclass}, we have
the following.

\begin{corollary}  \label{cor:number simple}
The number of simple $\HC_n$-modules equals $|\mc{SP}_n|$. More
precisely, the number of simple $\HC_n$-modules of type $\texttt M$
equals $|\mc{SP}^+_n|$ and the number of simple $\HC_n$-modules of
type $\texttt Q$ equals $|\mc{SP}^-_n|$.
\end{corollary}

\section{The (spin) characteristic map}\label{sec:spinch}

In this section, we develop systematically the representation theory
of $\HC_n$ after a quick review of the Frobenius characteristic map
for $\mf S_n$. Following \cite{Jo2}, we define a (spin)
characteristic map using the character table for the simple
$\HC_n$-modules, and establish its main properties. We review the
relevant aspects of symmetric functions. The image of the
irreducible characters of $\HC_n$ under the characteristic map are
shown to be Schur $Q$-functions up to some 2-powers.

\subsection{The Frobenius characteristic map}

The conjugacy classes of $\mf S_n$ are parameterized by partitions
$\la$ of $n$. Let
$$
z_\la =\prod_{i \ge 1} i^{m_i} m_i!
$$
denote the order of the centralizer of an element in a conjugacy
class of cycle type $\la$.

Let $R_n :=R(\mf S_n)$ be the Grothendieck group of $\mf
S_n\text{-}\mf{mod}$, which can be identified with the $\Z$-span of
irreducible characters $\chi^\la$ of the Specht modules $S^{\la}$
for $\la \in\mc P_n$. There is a bilinear form on $R_n$ so that
$(\chi^\la, \chi^\mu) =\delta_{\la\mu}.$ This induces a bilinear
form on the direct sum
$$
R =\bigoplus_{n=0}^\infty R_n,
$$
so that the $R_n$'s are orthogonal for different $n$. Here $R_0=\Z$.
In addition, $R$ is a graded ring with multiplication given by $fg
=\ind_{\mf S_m \times \mf S_n}^{\mf S_{m+n}} (f\otimes g)$ for $f
\in R_m$ and $g \in R_n$.

Denote by $\La$ the ring of symmetric functions in infinitely many
variables, which is the $\Z$-span of the monomial symmetric
functions $m_\la$ for $\la \in \mc P$. There is a standard bilinear
form $(\cdot, \cdot)$ on $\La$ such that the Schur functions $s_\la$
form an orthonormal basis for $\La$. The ring $\La$ admits several
distinguished bases: the complete homogeneous symmetric functions
$\{h_\la\}$, the elementary symmetric functions $\{e_\la\}$, and the
power-sum symmetric functions $\{p_\la\}$. See \cite{Mac}.

The (Frobenius) characteristic map $\ch: R \rightarrow \La$ is
defined by
\begin{equation}
\ch (\chi) =\sum_{\mu\in \mc{P}_n} z_\mu^{-1} \chi_\mu p_\mu,
\label{charmap}
\end{equation}
where $\chi_\mu$ denotes the character value of $\chi$ at a
permutation of cycle type $\mu$. Denote by $\textbf{1}_n$ and
$\sgn_n$ the trivial and the sign module/character of $\mf S_n$,
respectively. It is well known that
\begin{itemize}
\item
$\ch$ is an isomorphism of graded rings.

\item
$\ch$ is an isometry.

\item
$\ch(\textbf{1}_n) =h_n$, \quad $\ch(\sgn_n) =e_n$, \quad
$\ch(\chi^\la) =s_\la$.
\end{itemize}
Moreover,  the following holds for any composition $\mu$ of $n$:
\begin{equation}\label{eq:chpermutation}
\ch \big({\rm ind}^{\C \mf S_n}_{\C \mf
S_{\mu}}\textbf{1}_n\big)=h_{\mu},
\end{equation}
where $\mf S_{\mu} =\mf S_{\mu_1} \times \mf S_{\mu_2} \times\cdots$
denotes the associated Young subgroup.

We record the Cauchy identity for later use (cf. \cite[I, \S
4]{Mac})
\begin{equation}\label{eq:Cauchy}
\sum_{\mu\in\mc P}m_{\mu}(y)h_{\mu}(z) = \prod_{i,
j}\frac1{1-y_iz_j} =\sum_{\la\in\mc P}s_{\la}(y)s_{\la}(z).
\end{equation}

\subsection{The basic spin module}  \label{subsec:basicSpin}

The exterior algebra $\Cl_n$ is naturally an $\HC_n$-module (called
the {\em basic spin module}) where the action is given by
$$
c_i.(c_{i_1}c_{i_2}\ldots) =c_ic_{i_1}c_{i_2}\ldots, \quad
\sigma.(c_{i_1}c_{i_2}\ldots)=c_{\sigma(i_1)}c_{\sigma(i_2)} \ldots,
$$
for $\sigma \in \mf S_n$.  Let $\sigma=\sigma_1\ldots \sigma_\ell\in
\mf S_n$ be a cycle decomposition with cycle length of $\sigma_i$
being $\mu_i$. If  $I$ is a union of some of the
$\text{supp}(\sigma_i)$'s, say $I =\text{supp}(\sigma_{i_1}) \cup
\ldots \cup \text{supp}(\sigma_{i_s})$, then
$\sigma(c_I)=(-1)^{\mu_{i_1}+\ldots +\mu_{i_s} -s}c_I$.  Otherwise,
$\sigma(c_I)$ is not a scalar multiple of $c_I$. This observation
quickly leads to the following.

\begin{lemma}   \label{lem:basicSpinChar}
The value of the character $\xi^n$ of the  basic spin
$\HC_n$-module at the conjugacy class $\mc{C}_{\alpha}^+$ is
given by
\begin{equation}\label{E:spinchar}
\xi^n_\alpha =2^{\ell(\alpha)}, \qquad
\mbox{$\alpha\in\mathcal{OP}_n$}.
\end{equation}
\end{lemma}

The basic spin module of $\HC_n$ should be regarded as the spin
analogue of the trivial/sign modules of $\mf S_n$.

\subsection{The ring $R^-$}
\label{subsec:ringR-}

Thanks to the superalgebra isomorphism \eqref{eq:HC=Bn-},
$\HC_n$-$\mf{mod}$ is equivalent to the category of spin $\td
B_n$-modules. We shall not distinguish these two isomorphic
categories below, and the latter one has the advantage that one can
apply the standard arguments from the theory of finite groups
directly as we have seen in Section~\ref{sec:spinHC}. Denote by
$R_n^-$ the Grothendieck group of $\HC_n$-$\mf{mod}$. As in the
usual (ungraded) case, we may replace the isoclasses of modules by
their characters, and then regard $R_n^-$ as the free abelian group
with a basis consisting of the characters of the simple
$\HC_n$-modules. It follows by Corollary~\ref{cor:number simple}
that the rank of $R_n^-$ is $|\mc{SP}_n|$. Let
$$
 R^- := \bigoplus_{n=0}^\infty  R^-_n,  \qquad
 R^-_\Q := \bigoplus_{n=0}^\infty \Q \otimes_\Z R^-_n,
$$ where
it is understood that $R^-_0 =\Z$.

We shall define a ring structure on $R^-$ as follows. Let
$\HC_{m,n}$ be the subalgebra of $\HC_{m+n}$ generated by
$\Cl_{m+n}$ and $\mf S_m \times \mf S_n$. For $M \in
\HC_m$-$\mf{mod}$ and $N \in \HC_n$-$\mf{mod}$, $M\otimes N$ is
naturally an $\HC_{m,n}$-module, and we define the product
$$
[M] \cdot [N] = [\HC_{m+n} \otimes_{\HC_{m,n}} (M \otimes N)],
$$
and then extend by $\Z$-bilinearity. It follows from the
properties of the induced characters that the multiplication on
$R^-$ is commutative and associative.

Given spin $\td B_n$-modules $M,N$, we define a bilinear form on
$R$ and so on $R_\Q$ by letting
\begin{equation}  \label{eq:bilformR-}
\langle M,N \rangle  =\dim \Hom_{\td B_n} (M,N).
\end{equation}

\subsection{The Schur $Q$-functions}
The materials in this subsection are pretty standard (cf. \cite{Mac,
Jo1} and \cite[Appendix~A]{CW}). Recall $p_r$ is the $r$th power sum
symmetric function, and for a partition $\mu =(\mu_1, \mu_2,
\ldots)$ we define $p_\mu =p_{\mu_1} p_{\mu_2} \cdots$. Let $x
=\{x_1, x_2, \ldots\}$ be a set of indeterminates. Define  a family
of symmetric functions $q_r =q_r(x)$, $r \ge 0$, via a generating
function
\begin{align}
Q(t) := \sum_{r\geq 0}q_r(x)t^r=\prod_i\frac{1+t x_i}{1-t
x_i}.\label{gen.fun.qr}
\end{align}
It follows from \eqref{gen.fun.qr} that
$$
Q(t) =\exp \Big( 2 \sum_{r \ge 1, r\,\text{odd}} \frac{p_{r}
t^{r}}{r}  \Big).
$$
Componentwise, we have
\begin{equation}  \label{eq:qnp}
q_n =\sum_{\alpha \in \mc{OP}_n} 2^{\ell(\alpha)} z_\alpha^{-1}
p_\alpha.
\end{equation}

Note that $q_0=1$, and that $Q(t)$ satisfies the relation
\begin{equation}  \label{eq:QQ}
Q(t) Q(-t) =1,
\end{equation}
which is equivalent to the identities:
$$
\sum_{r+s =n} (-1)^r q_r q_s =0, \quad n \ge 1.
$$
These identities are  vacuous  for $n$ odd. When $n=2m$ is even, we obtain that
\begin{equation}  \label{eq:q quardr}
q_{2m} = \sum_{r=1}^{m-1} (-1)^{r-1} q_r q_{2m-r} - \hf (-1)^m
q_m^2.
\end{equation}
Let $\Gamma$ be the $\Z$-subring of $\Lambda$ generated by the $q_r$'s:
$$
\Gamma =\Z [q_1, q_2, q_3, \ldots].
$$
The ring $\Gamma$ is graded by the degree of functions:
$
\Gamma =\bigoplus_{n \ge 0} \Gamma^n$.
We set
$
\Gamma_{\Q} =\Q \otimes_{\Z} \Gamma.
$
For any partition $\mu =(\mu_1, \mu_2, \ldots)$, we define
$$
q_\mu =q_{\mu_1} q_{\mu_2} \ldots.
$$

\begin{theorem}  \label{th:Gammafree}
The following holds for $\Gamma$ and $\Gamma_{\Q}$:
\begin{enumerate}
\item
$\Gamma_{\Q}$ is a polynomial algebra with polynomial generators
$p_{2r-1}$ for $r \ge 1$.


\item
$\{p_\mu \mid \mu \in \mc{OP}\}$ forms a linear basis for
    $\Gamma_{\Q}$.

\item
$\{q_\mu \mid \mu \in \mc{OP}\}$ forms a linear basis for
    $\Gamma_{\Q}$.

\item
$\{q_\mu \mid \mu \in \mc{SP}\}$ forms a $\Z$-basis for
    $\Gamma$.
\end{enumerate}
\end{theorem}

\begin{proof}
By clearing the denominator of the identity
$$
\frac{Q'(t)}{Q(t)}  =2\sum_{r \ge 0} p_{2r+1} t^{2r},
$$
we deduce that
$$
r q_r = 2(p_1q_{r-1} + p_3 q_{r-3} + \ldots ).
$$

By using induction on $r$, we conclude that (i) each $q_r$ is expressible as a
polynomial in terms of $p_s$'s with odd $s$; (ii) each $p_r$ with odd $r$ is expressible
as a polynomial in terms of $q_s$'s, which can be further restricted to the odd $s$.
So,
$
\Gamma_\Q = \Q [p_1, p_3, \ldots ] = \Q [q_1, q_3, \ldots ],
$
and from this (1), (2) and (3) follow.

To prove (4), it suffices to show that, for
any partition $\la$,
$$
q_\la =\sum_{\mu \in \mc{SP}, \mu \ge \la} a_{\mu\la} q_\mu,
$$
for some $a_{\mu\la} \in \Z$.  This can be seen by induction downward
on the dominance order on $\la$ with the help of \eqref{eq:q quardr}.
\end{proof}

We shall define the Schur $Q$-functions $Q_\la$, for $\la \in \mc{SP}$. Let
$$
Q_{(n)} =q_n, \quad n \ge 0.
$$
Consider the generating function
$$
Q(t_1, t_2) := (Q(t_1) Q(t_2) -1) \frac{t_1 -t_2}{t_1 +t_2}.
$$
By \eqref{eq:QQ}, $Q(t_1, t_2)$ is a power series in $t_1$ and $t_2$, and we write
$$
Q(t_1, t_2) = \sum_{r, s \ge 0} Q_{(r,s)} t_1^r t_2^s.
$$
Noting $Q(t_1, t_2) =- Q(t_2,
t_1)$, we have $Q_{(r,s)} =- Q_{(s,r)}, Q_{(r,0)} =q_r$. In addition,
$$
Q_{(r,s)} =q_r q_s + 2 \sum_{i=1}^s (-1)^i q_{r+i} q_{s-i}, \quad
r>s.
$$

For a strict partition $\la =(\la_1,
\ldots, \la_m)$, we define the Schur $Q$-function $Q_\la$ recursively as follows:
\begin{eqnarray*} \label{qschur:aux7}
Q_\la &=&\sum_{j=2}^m (-1)^j Q_{(\la_1,\la_j)} Q_{(\la_2,\ldots,
\hat{\la}_j,\ldots, \la_m)}, \quad \text{for } m \text{ even},
 \\
Q_\la &=&\sum_{j=1}^m (-1)^{j-1} Q_{\la_j} Q_{(\la_1,\ldots,
\hat{\la}_j,\ldots, \la_m)}, \quad \text{for } m \text{ odd}.
\end{eqnarray*}
Note that the $Q_\la$ above is simply the Laplacian expansion of the
pfaffian of the skew-symmetric matrix $(Q_{(\la_i,\la_j)})$ when $m$
is even (possibly $\la_m=0$).

It follows from the recursive definition of $Q_\la$ and \eqref{eq:q quardr} that,
for $\la \in \mc{SP}_n$,
$$
Q_\la = q_\la + \sum_{\mu \in \mc{SP}_n, \mu > \la} d_{\la\mu}
q_\mu,
$$
for some $d_{\la\mu} \in \Z$.
From this and Theorem~\ref{th:Gammafree} we further deduce the following.
\begin{theorem}\label{qschur:aux9}
The $Q_\la$ for all strict partitions $\la$ form a $\Z$-basis for
$\Gamma$. Moreover, for any composition $\mu$ of $n$, we have
 $$
 q_\mu = \sum_{\la \in \mc{SP}_n, \la \ge \mu}  \widehat{K}_{\la \mu}
Q_{\la},
$$
where $\widehat{K}_{\la \mu} \in \Z$ and $\widehat{K}_{\la \la} =1$.
\end{theorem}

Let $x =\{x_1, x_2, \ldots\}$ and $y =\{y_1, y_2, \ldots\}$ be two independent sets of variables.
We have by \eqref{gen.fun.qr} that
\begin{equation}  \label{eq:pp}
\prod_{i,j} \frac{ 1+ x_i y_j}{ 1-x_i y_j}
 = \sum_{\alpha \in \mc{OP}} 2^{\ell(\alpha)} z_\alpha^{-1} p_\alpha (x) p_\alpha (y).
\end{equation}

We define an inner product $\langle \cdot, \cdot \rangle$ on $\Gamma_\Q$ by letting
\begin{equation}  \label{eq:bilformGamma}
\langle p_\alpha, p_\beta \rangle = 2^{- \ell (\alpha)} z_\alpha
\delta_{\alpha \beta}.
\end{equation}

\begin{theorem}\label{qschur:aux12}
We have
$$
\langle Q_\la, Q_\mu \rangle =2^{\ell (\la)} \delta_{\la\mu},
\quad \la, \mu \in \mc{SP}.
$$
Moreover, the following Cauchy identity holds:
$$
\prod_{i,j} \frac{ 1+ x_i y_j}{ 1-x_i y_j} = \sum_{\la \in
\mc{SP}} 2^{-\ell(\la)}  Q_\la (x) Q_\la (y).
$$
\end{theorem}

We will skip the proof of Theorem~\ref{qschur:aux12} and make some
comments only. The two statements therein can be seen to be
equivalent in light of \eqref{eq:pp} and \eqref{eq:bilformGamma}.
One possible proof of the first statement following from the theory
of Hall-Littlewood functions \cite{Mac}, and another direct proof is
also available \cite{Jo1}. The second statement would follow easily
once the shifted Robinson-Schensted-Knuth correspondence is
developed (cf. \cite[Corollary~8.3]{Sag}).

\subsection{The characteristic map}
 \label{subsec:charmap}

We define the (spin) {\em characteristic map}
$$
\ch^-: R^-_\Q \longrightarrow \Gamma_\Q
$$
to be the linear map given by
\begin{equation}
\ch^- (\varphi) = \sum_{\alpha \in \mc{OP}_n}
z_\alpha^{-1}\varphi_\alpha p_\alpha, \qquad \varphi \in R^-_n.
\label{spincharmap}
\end{equation}
The following theorem is due to J\'ozefiak \cite{Jo2} (see
\cite[Chapter~3]{CW} for an exposition).

\begin{theorem}  \cite{Jo2} \label{th:2iso}
(1) The characteristic map $\ch^-: R^-_\Q \rightarrow \Gamma_\Q$ is
an isometry.

(2) The characteristic map $\ch^-: R^-_\Q \rightarrow \Gamma_\Q$ is
an isomorphism of graded algebras.
\end{theorem}

\begin{proof}[Sketch of a proof]
We first show that $\ch^-$ is an isometry. Take $\varphi, \psi \in
R_n^-$. Since $\varphi$ is a character of a $\Z_2$-graded module, we
have the character value $\varphi_\alpha =0$ for $\alpha \not \in
\mc{OP}_n$. We can reformulate the bilinear form
\eqref{eq:bilformR-} using the standard bilinear form formula on
characters of the finite group $\td B_n$ as
$$
\langle \varphi, \psi\rangle  = \sum_{\alpha \in \mc{OP}_n}
2^{-\ell(\alpha)} z_\alpha^{-1}\varphi_\alpha \psi_\alpha,
$$
which can be seen using \eqref{eq:bilformGamma} to be equal to
$\langle \ch^-(\varphi), \ch^-(\psi)\rangle$.

Next, we show that $\ch^-$ is a homomorphism of graded algebras. For
$\phi \in R^-_m$, $\psi \in R^-_n$ and $\gamma \in \mc{OP}_{m+n}$,
we obtain a standard induced character formula for $(\phi\cdot
\psi)_\gamma$ evaluated at a conjugacy class $\mc C_\gamma^+$. This
together with the definition of $\ch^-$ imply that
\begin{eqnarray*}
\ch^- (\phi \cdot \psi)
 &=& \sum_{\gamma \in \mc{OP}}z_\gamma^{-1} (\phi \cdot \psi)_\gamma p_\gamma   \\
 &=& \sum_\gamma \sum_{\alpha,\beta \in \mc{OP}, \alpha \cup \beta
 =\gamma} z_\gamma^{-1} \frac{z_\gamma}{z_\alpha z_\beta} \phi_\alpha
 \psi_\beta p_\gamma
 = \ch^-(\phi)\ch^- (\psi).
\end{eqnarray*}

Recalling the definition of $\ch^-$ and the basic spin character
$\xi^n$, it follows by \eqref{eq:qnp} and
Lemma~\ref{lem:basicSpinChar} that $\ch^-(\xi^n) =q_n.$ Since $q_n$
for $n \ge 1$ generate the algebra $\Gamma_\Q$ by
Theorem~\ref{th:Gammafree}, $\ch^-$ is surjective. Then $\ch^-$ is
an isomorphism of graded vector spaces by the following comparison
of the graded dimensions (cf. Corollary~\ref{cor:number simple} and
Theorem~\ref{th:Gammafree}):
$$
\dim_q R_\Q^- = \prod_{r \ge 1} (1+q^r) = \dim_q \Gamma_\Q.
$$
This completes the proof of the theorem.
\end{proof}

Recall from the proof above that $\ch^-( \xi^{n}) =q_n.$ Regarding
$\xi^{(n)} =\xi^{n}$, we define $\xi^\la$ for $\la \in \mc{SP}$
using the same recurrence relations for the Schur $Q$-functions
$Q_\la$. Then by Theorem~\ref{th:2iso},
$\ch^-(\xi^\la) =Q_\la$, and
$\langle \xi^\la, \xi^\mu \rangle =2^{\ell(\la)} \delta_{\la\mu}$,
for $\la, \mu \in \mc{SP}$.

For a partition $\la$  with length $\ell(\la)$, we set
\begin{eqnarray}\label{delta:lambda}
\delta(\lambda)= \left \{
 \begin{array}{ll}
 0,
 & \text{ if } \ell(\lambda) \text{ is even}, \\
 1, & \text{ if } \ell(\lambda) \text{ is odd}.
 \end{array}
 \right.
\end{eqnarray}
By chasing the recurrence relation more closely, we can show by
induction on $\ell(\la)$ that the element
$$
\zeta^\la := 2^{-\frac{\ell(\la) -\delta(\la)}2} \xi^\la
$$
lies in $R^-$, for $\la \in \mc{SP}_n$. Note that
\begin{equation}   \label{eq:irredch}
\ch^-(\zeta^\la) =2^{-\frac{\ell(\la) -\delta(\la)}2} Q_\la.
\end{equation}
It follows that, for each $\la \in \mc{SP}_n$,
\begin{eqnarray*}
Q_\la =  2^{\frac{\ell(\la) -\delta(\la)}2} \sum_{\alpha \in
\mc{OP}_n} z_\alpha^{-1} \zeta^\la_\alpha p_\alpha.
\end{eqnarray*}

Given $\mu \in \mc{P}_n$, let us denote $\HC_\mu :=\HC_{\mu_1}
\otimes \HC_{\mu_2} \otimes \cdots$, and recall the Young subgroup
$\mf S_\mu$ of $\mf S_n$. The induced $\HC_n$-module
$$
M^\mu :=\HC_n \otimes_{\C \mf S_\mu} {\bf 1}_n
$$
will be called a {\em permutation module} of $\HC_n$. By the
transitivity of the tensor product, it can be rewritten as
$$
M^\mu =\HC_n \otimes_{\HC_\mu} (\Cl_{\mu_1} \otimes \Cl_{\mu_2} \otimes \cdots).
$$
Since $\ch^-( \xi^{n}) =q_n$ and $\ch^-$ is an algebra homomorphism, we obtain that
\begin{equation}   \label{eq:spinchpermutation}
\ch^- (M^\mu) = q_\mu.
\end{equation}

\begin{theorem}  \cite{Jo2}  \label{th:irredch}
The set of characters $\zeta^\la$ for $\la \in \mc{SP}_n$ is a
complete list of pairwise non-isomorphic simple (super) characters
of $\HC_n$. Moreover, the degree of $\zeta^\la$
is equal to $
2^{n-\frac{\ell(\la) -\delta(\la)}2} g_\la$, where
$$
g_\la = \frac{n!}{\la_1!\ldots \la_\ell !}
\prod_{i<j} \frac{\la_i -\la_j}{\la_i +\la_j}.
$$
\end{theorem}

\begin{proof}[Sketch of a proof]
For strict partitions $\la, \mu$, we have
\begin{eqnarray}
 \langle \zeta^\la, \zeta^\la \rangle
 &=& \left\{
 \begin{array}{ll}
 1 & \text{ for } \ell(\la) \text{ even},  \\
 2 & \text{ for } \ell(\la) \text{ odd},
 \end{array}
 \right. \label{eqn:spincharinner}\\
  \langle \zeta^\la, \zeta^\mu \rangle
 &=& 0, \quad \text{ for } \la \neq \mu.\notag
\end{eqnarray}
From this and Corollary~\ref{cor:number simple}, it is not difficult
to see that either $\zeta^\la$ or $-\zeta^\la$ is a simple (super)
character, first for $\la$ with $\ell(\la)$ even and then for $\la$
with $\ell(\la)$ odd.

To show that $\zeta^\la$ instead of $-\zeta^\la$ is a character of a
simple module, it suffices to know that the degree of $\zeta^\la$ is
positive. The degree formula can be established by induction on
$\ell(\la)$ (see the proof of \cite[Proposition~4.13]{Jo1} for
detail).
\end{proof}

We shall denote by $D^\la$  the irreducible $\HC_n$-module whose
character is $\zeta^\la$,  for $\la \in \mc{SP}_n$. The following is
an immediate consequence of Theorem~\ref{qschur:aux9},
\eqref{eq:irredch}, and \eqref{eq:spinchpermutation}.

\begin{proposition}  \label{prop:triang}
Let $\mu$ be a composition of $d$. We have the following
decomposition of $M^{\mu}$ as an $\HC_d$-module:
$$
M^{\mu} \cong \bigoplus_{\la \in \mc{SP}, \la \ge \mu}
2^{\frac{\ell(\la)-\delta(\la)}{2}}  \widehat{K}_{\la \mu} D^\la,
$$
where $\widehat{K}_{\la \mu} \in \Z_+$.
\end{proposition}

\section{The Schur-Sergeev duality}\label{sec:duality}

In this section, we formulate a double centralizer property for the
actions of the Lie superalgebra $\mf{q}(n)$ and of the algebra
$\HC_d$ on the tensor superspace $(\C^{n|n})^{\otimes d}$. We obtain
a multiplicity-free decomposition of $(\C^{n|n})^{\otimes d}$ as a
$U (\mf q(n))\otimes \HC_d$-module. The characters of the simple
$\mf{q}(n)$-modules arising this way are shown to be Schur
$Q$-functions (up to some 2-powers).

\subsection{The classical Schur duality}
Let us first recall a general double centralizer property. We
reproduce a proof below which can be easily adapted to the
superalgebra setting later on.

\begin{proposition}  \label{prop:Centralizer}
Suppose that $W$ is a finite-dimensional vector space, and
$\mc B$ is a semisimple subalgebra of $\End (W)$. Let $\mc A =
\End_{\mc B} (W)$. Then, $\End_{\mc A} (W) =\mc B$.

As an $\mc A \otimes \mc B$-module, $W$ is multiplicity-free,
i.e.,
$$
W \cong \bigoplus_{i} U_i \otimes V_i,
$$
where $\{U_i\}$  are pairwise
non-isomorphic simple $\mc A$-modules and $\{V_i\}$ are pairwise
non-isomorphic simple $\mc B$-modules.
\end{proposition}

\begin{proof}
Assume that $V_a$ are all the pairwise non-isomorphic simple $\mc
B$-modules. Then the Hom-spaces $U_a:=\Hom_{\mc B}(V_a, W)$ are
naturally $\mc A$-modules. By the semisimplicity assumption on $\mc
B$, we have a $\mc B$-module isomorphism:
$$
W\cong\bigoplus_aU_a\otimes V_a.
$$
By applying Schur's Lemma, we obtain
$$
\mc A = \End_{\mc B} (W)\cong\bigoplus_{a}\End_{\mc B}(U_a\otimes
V_a) \cong \bigoplus_a\End(U_a)\otimes{\rm id}_{V_a}.
$$
Hence $\mc A$ is semisimple and $U_a$ are all the pairwise
non-isomorphic simple $\mc A$-modules.

Since $\mc A$ is now semisimple, we can reverse the roles of $\mc A$
and $\mc B$ in the above computation of $\End_{\mc B}(W)$, and
obtain the following isomorphism:
$$
\End_{\mc A}(W)\cong\bigoplus_a{\rm id}_{U_a}\otimes\End(V_a)\cong
\mc B.
$$
The proposition is proved.
\end{proof}

The natural action of $\mf{gl}(n)$ on $\C^n$ induces a
representation $(\omega_d, (\C^n)^{\otimes d})$ of the general linear
Lie algebra $\mf{gl}(n)$, and we have a representation $(\psi_d,
(\C^n)^{\otimes d})$ of the symmetric group $\mf S_d$ by
permutations of the tensor factors.

\begin{theorem}
[Schur duality] The  images $\omega_d(U(\mf{gl}(n)))$ and $\psi_d(\C
\mf S_d)$ satisfy the double centralizer property, i.e.,
\begin{align*}
\omega_{d}(U(\mf{gl}(n)))=&\End_{\C\mf S_d}((\C^n)^{\otimes d}),\\
&\End_{\mf{gl}(n)}((\C^n)^{\otimes d}) = \psi_d(\C\mf S_d).
\end{align*}
Moreover, as a $\mf{gl}(n)\times\mf S_d$-module,
\begin{equation}\label{eq:Schurduality}
(\C^n)^{\otimes d}\cong\bigoplus_{\la\in\mc{P}_d, \ell(\la)\leq n}
L(\la)\otimes S^{\la},
\end{equation}
where $L(\la)$ denotes the irreducible $\mf{gl}(n)$-module of
highest weight $\la$.
\end{theorem}
We will skip the proof of the Schur duality here, as it is similar
to a detailed proof below for its super analogue
(Theorems~\ref{Sergeev q} and \ref{th:QnDecomp}).

As an application of the Schur duality, let us derive the character
formula for $\ch L(\la) =\tr x_1^{E_{11}}x_2^{E_{22}}\cdots
x_n^{E_{nn}}|_{L(\la)}$, where as usual $E_{ii}$ denotes the matrix
whose $(i,i)$th entry is 1 and zero else.

Denote by $\mc{CP}_d(n)$ the set of compositions of $d$ of length
$\leq n$. Set $W=(\C^n)^{\otimes d}$.  Given  $\mu\in\mc{CP}_d(n)$,
let $W_{\mu}$ indicate the $\mu$-weight space of $W$. Observe that
$W_{\mu}$ has a linear basis
\begin{equation}\label{eq:basis}
e_{i_1} \otimes \ldots \otimes
e_{i_d}, \text{ with }
\{i_1, \ldots, i_d\} =\{\underbrace{1,\ldots, 1}_{\mu_1},
\ldots, \underbrace{n,\ldots, n}_{\mu_n} \}.
\end{equation}
On the other hand, $\mf S_n$ acts on the basis \eqref{eq:basis} of
$W_{\mu}$ transitively, and the stablizer of the basis element
$e_1^{\mu_1}\otimes\cdots\otimes e_n^{\mu_n}$ is the Young subgroup
$\mf S_{\mu}$. Therefore we have $W_{\mu}\cong\ind^{\mf S_d}_{\mf
S_{\mu}}\textbf{1}_d$ and hence
\begin{equation}
W\cong\bigoplus_{\mu\in\mc{CP}_d(n)}W_{\mu}\cong\bigoplus_{\mu\in\mc{CP}_d(n)}\ind^{\mf
S_d}_{\mf S_{\mu}}\textbf{1}_d.
\end{equation}
This and \eqref{eq:Schurduality} imply that
$$
\bigoplus_{\mu\in\mc{CP}_d(n)}\ind^{\mf S_d}_{\mf
S_{\mu}}\textbf{1}_d\cong\bigoplus_{\la\in\mc{P}_d,\ell(\la)\leq
n}L(\la)\otimes S^{\la}.
$$
Applying the trace operator $\tr x_1^{E_{11}}x_2^{E_{22}}\cdots
x_n^{E_{nn}}$ and the Frobenius characteristc map $\ch$ to both
sides of the above isomorphism and summing over $d$, we obtain
$$
\sum_{\mu\in\mc{P},\ell(\mu)\leq n}m_{\mu}(x_1,\ldots,
x_n)h_{\mu}(z) =\sum_{\la\in\mc{P},\ell(\la)\leq n}\ch
L(\la)s_{\la}(z),
$$
where $z=\{z_1,z_2,\ldots\}$ is infinite. Then using the Cauchy
identity \eqref{eq:Cauchy} and noting the linear independence of the
$s_{\la}(z)$'s, we recover the following well-known character
formula:
\begin{equation}
\ch L(\la)=s_{\la}(x_1,x_2,\ldots,x_n).
\end{equation}

\subsection{The queer Lie superalgebras}

The associative superalgebra $Q(n)$ (defined in
Section~\ref{sec:superalg}) equipped with the super-commutator is
called the {\em queer Lie superalgebra} and  denoted by $\mf{q}(n)$. Let
$$
I(n|n) =\{\bar{1}, \ldots, \bar{n}, 1,\ldots, n\}.
$$
The $\mf{q}(n)$  can be explicitly realized as matrices in the $n|n$ block form,
indexed by $I(n|n)$:
\begin{align}   \label{qmatrix}
\begin{pmatrix}
a&b\\
b&a
\end{pmatrix},
\end{align}
where $a$ and $b$ are arbitrary $n\times n$ matrices. The even
(respectively, odd) part $\ev\ga$ (respectively, $\od\ga$) of
$\ga=\mf{q}(n)$ consists of those matrices of the form
\eqref{qmatrix} with $b=0$ (respectively, $a=0$). Denote by $E_{ij}$
for $i,j \in I(n|n)$ the standard elementary matrix with the
$(i,j)$th entry being $1$ and zero elsewhere.

The standard Cartan subalgebra $\h={\ev \h}\oplus {\od \h}$ of $\ga$
consists of matrices of the form \eqref{qmatrix} with $a,b$ being
arbitrary diagonal matrices. Noting that $[\ev\h,\h]=0$ and
$[\od\h,\od\h]=\ev\h$, the Lie superalgebra $\h$ is not abelian. The
vectors
\begin{align*}
H_i:=E_{\ov{i},\ov{i}}+E_{ii},\quad i=1,\ldots,n,
\end{align*}
is a basis for the $\ev\h$.  We let $\{\epsilon_i|i=1,\ldots,n\}$
denote the corresponding dual basis in $\h_{\bar{0}}^*$. With
respect to $\ev\h$ we have the root space decomposition
$\ga=\h\oplus\bigoplus_{\alpha\in\Phi}\ga_\alpha$ with roots
$\{\epsilon_i-\epsilon_j|1\le i\not=j\le n\}.$ For each root
$\alpha$ we have $\dim_\C(\ga_\alpha)_i=1$, for $i\in\Z_2$. The
system of positive roots corresponding to the Borel subalgebra $\mf
b$ consisting of matrices of the form \eqref{qmatrix} with $a,b$
upper triangular is given by $\{\epsilon_i-\epsilon_j|1\le i<j\le
n\}$.

 The Cartan
subalgebra $\h=\ev\h\oplus\od\h$ is a solvable Lie superalgebra, and
its irreducible representations are described as follows. Let
$\la\in\ev\h^*$ and consider the symmetric bilinear form
$\langle\cdot,\cdot\rangle_\la$ on $\od\h$ defined by
\begin{align*}
\langle v,w\rangle_\la:=\la([v,w]).
\end{align*}
Denote by $\text{Rad} \langle\cdot,\cdot\rangle_\la$ the radical of
the form $\langle\cdot,\cdot\rangle_\la$. Then the form
$\langle\cdot,\cdot\rangle_\la$ descends to a nondegenerate
symmetric bilinear form on $\od\h/\text{Rad}
\langle\cdot,\cdot\rangle_\la$, and it gives rise to a Clifford
superalgebra $\Cl_\la:= \Cl(\od\h/\text{Rad}
\langle\cdot,\cdot\rangle_\la)$. By definition we have an
isomorphism of superalgebras
\begin{equation*}  
\Cl_\la \cong U(\h) /I_\la,
\end{equation*}
where $I_\la$ denotes the ideal of $U(\h)$ generated by $\text{Rad}
\langle\cdot,\cdot\rangle_\la$ and $a -\la(a)$ for $a \in\ev\h$.

Let $\od\h'\subseteq\od\h$ be a maximal isotropic subspace and
consider the subalgebra $\h'=\ev\h\oplus\od\h'$.  Clearly the
one-dimensional $\ev\h$-module $\C v_\la$, defined by
$hv_\la=\la(h)v_\la$, extends trivially to $\h'$.  Set
$$
W_\la:=\text{Ind}_{\h'}^{\h}\C v_\la.
$$
We see that the action of
$\h$ factors through $\Cl_\la$ so that $W_\la$ becomes the unique
irreducible $\Cl_\la$-module and hence is independent of the choice
of $\od\h'$. The following can now be easily verified.

\begin{lemma}\label{cliff:irred:mod}
For $\la\in\ev\h^*$,
$W_\la$ is an irreducible $\h$-module. Furthermore, every
finite-dimensional irreducible $\h$-module is isomorphic to some
$W_\la$.
\end{lemma}

Let $V$ be a finite-dimensional irreducible $\ga$-module and
let $W_\mu$ be an irreducible $\h$-submodule of $V$.
For every $v\in W_\mu$
we have $hv=\mu(h)v$, for all $h\in\ev\h$.  Let $\alpha$ be a
positive root with associated root vectors $e_\alpha$ and
$\ov{e}_\alpha$ in $\mf{n}^+$ satisfying
$\text{deg}e_\alpha=\bar{0}$ and $\text{deg}\ov{e}_\alpha=\bar{1}$.
Then the space $\C e_\alpha W_\mu+\C \ov{e}_\alpha W_\mu$ is an
$\h$-module on which $\ev\h$ transforms by the character
$\mu+\alpha$. Thus by the finite dimensionality of $V$ there exists
$\la\in\ev\h^*$ and an irreducible $\h$-module $W_\la\subseteq V$
such that $\mf{n}^+W_\la=0$. By the irreducibility of $V$ we must
have $U(\mf{n}^-)W_\la=V$, which gives rise to a weight space
decomposition of $V=\bigoplus_{\mu\in\ev\h^*}V_\mu$. The space
$W_\la=V_\la$ is the {\em highest weight space} of $V$, and it
completely determines the irreducible module $V$.  We denote $V$ by
$V(\la)$.

Let $\ell(\la)$ be the dimension of space
$\od\h/\text{Rad}
\langle\cdot,\cdot\rangle_\la$, which equals the
number of $i$ such that $\la(H_i)\neq 0$. Then the highest weight
space $W_\la$ of $V(\la)$ has dimension
$2^{(\ell(\la)+\delta(\la))/2}$. It is easy to see that the
$\h$-module $W_\la$ has an odd automorphism if and only if
$\ell(\la)$ is an odd integer. An automorphism of the irreducible
$\ga$-module $V(\la)$ clearly induces an $\h$-module automorphism of
its highest weight space. Conversely, any $\h$-module automorphism
on $W_\la$ induces an automorphism of the $\ga$-module
$\text{Ind}_\mf{b}^\ga W_\la$. Since an automorphism preserves the
maximal submodule, it induces an automorphism of the unique
irreducible quotient $\ga$-module. Summarizing, we have established
the following.

\begin{lemma}
Let $\ga =\qn$, and $\h$ be a Cartan subalgebra of $\ga$. Let
$\la\in\ev\h^*$ and $V(\la)$ be an irreducible $\ga$-module of
highest weight $\la$. We have
\begin{eqnarray*}
\dim \text{End}_\ga(V(\la))= \left \{
 \begin{array}{ll}
 1,
 & \text{ if } \ell(\lambda) \text{ is even}, \\
 2
 , & \text{ if } \ell(\lambda) \text{ is odd}.
 \end{array}
 \right.
\end{eqnarray*}
\end{lemma}

\subsection{The Sergeev duality}
In this subsection, we give a detailed exposition (also see
\cite[Chapter~3]{CW}) on the results of Sergeev \cite{Se1}.

Set $V =\C^{n|n}$. We have a representation $(\Omega_d, V^{\otimes
d})$ of $\gl(n|n)$, hence of its subalgebra $\qn$, and we also have
a representation $(\Psi_d, V^{\otimes d})$ of the symmetric group
${\mathfrak S}_d$ defined by
\begin{align*}
\Psi_d (s_i).(v_1\otimes  \ldots\otimes v_i\otimes v_{i+1}\otimes\ldots \otimes v_d) &=
 (-1)^{|v_i|\cdot|v_{i+1}|}v_1\otimes  \ldots\otimes v_{i+1}\otimes v_i\otimes\ldots \otimes v_d,
 \end{align*}
where $s_i=(i,i+1)$ is the simple reflection and $v_i, v_{i+1} \in V$ are $\Z_2$-homogeneous.
Moreover, the actions of $\gl (n|n)$ and the
symmetric group ${\mathfrak S}_d$ on $V^{\otimes d}$ commute with
each other. Note in addition that the Clifford algebra $\Cl_d$ acts
on $V^{\otimes d}$, also denoted by $\Psi_d$:
\begin{align*}
\Psi_d (c_i).(v_1\otimes  \ldots \otimes v_d) =
 (-1)^{(|v_1| +\ldots +|v_{i-1}|)} v_1\otimes
   \ldots \otimes v_{i-1} \otimes P v_i \otimes\ldots \otimes v_d,
\end{align*}
where $v_i \in V$ is assumed to be $\Z_2$-homogeneous and
$1\le i \le n$.

\begin{lemma}   \label{lem:commQ}
Let $V =\C^{n|n}$. The actions of $\mf S_d$ and $\Cl_d$ above give
rise to a representation $(\Psi_d, V^{\otimes d})$ of $\HC_d$.
Moreover, the actions of $\qn$ and $\HC_d$ on $V^{\otimes d}$
super-commute with each other.
\end{lemma}
Symbolically, we write
$$
 \qn \; \stackrel{\Omega_{d}}{\curvearrowright} \; V^{\otimes d} \;
\stackrel{\Psi_{d}}{\curvearrowleft} \; {\HC_d}.
$$

\begin{proof}
It is straightforward to check that the actions of $\mf S_d$ and
$\Cl_d$ on $V^{\otimes d}$ are compatible and they give rise to an
action of $\HC_d$. By the definition of $\qn$ and the definition of
$\Psi_d(c_i)$ via $P$, the action of $\qn$ (super)commutes with the
action of $c_i$ for $1\le i \le d$. Since $\gl(n|n)$ (super)commutes
with $\mf S_d$, so does the subalgebra  $\qn$ of $\gl(n|n)$. Hence,
the action of $\qn$ commutes with the action of $\HC_d$ on
$V^{\otimes d}$.
\end{proof}

Let us digress on the double centralizer property for superalgebras
in general. Note the superalgebra isomorphism
$$
Q(m) \otimes Q(n) \cong M(mn|mn).
$$
Hence, as a $Q(m) \otimes Q(n)$-module, the tensor product $\C^{m|m}
\otimes \C^{n|n}$ is a direct sum of two isomorphic copies of a
simple module (which is $\cong \C^{mn|mn}$), and we have
$\Hom_{Q(n)} (\C^{n|n},\C^{mn|mn}) \cong \C^{m|m}$ as a
$Q(m)$-module. Let $\mc A$ and $\mc B$ be two semisimple
superalgebras. Let $M$ be a simple $\mc A$-module  of type $\texttt
Q$ and let $N$ be a simple $\mc B$-module of type $\texttt Q$. Then,
by Lemma~\ref{tensorsmod}, the $\mc A \otimes \mc B$-supermodule $M
\otimes N$ is a direct sum of two isomorphic copies of a simple
module $M\circledast  N$ of type $\texttt M$, and we shall write
$M\circledast  N =2^{-1} M \otimes N$; Moreover, $\Hom_{\mc B} (N,
2^{-1} M \otimes N)$ is naturally an $\mc A$-module, which is
isomorphic to the $\mc A$-module $M$. The usual double centralizer
property Proposition~\ref{prop:Centralizer} affords the following
superalgebra generalization (with essentially the same proof once we
keep in mind the Super Schur's Lemma~\ref{lem:superSchur}).

\begin{proposition}  \label{prop:ssCentralizer}
Suppose that $W$ is a finite-dimensional vector superspace, and
$\mc B$ is a semisimple subalgebra of $\End (W)$. Let $\mc A =
\End_{\mc B} (W)$. Then, $\End_{\mc A} (W) =\mc B$.

As an $\mc A \otimes \mc B$-module, $W$ is multiplicity-free,
i.e.,
$$
W \cong \bigoplus_{i} 2^{-\delta_i} U_i \otimes V_i,
$$
where $\delta_i \in \{0,1\}$, $\{U_i\}$ are pairwise
non-isomorphic simple $\mc A$-modules, $\{V_i\}$ are pairwise
non-isomorphic simple $\mc B$-modules. Moreover, $U_i$ and
$V_i$ are of same type, and they are of type $\texttt M$ if and
only if $\delta_i =0$.
\end{proposition}

\begin{theorem} [Sergeev duality I] \label{Sergeev q}
The images $\Omega_d(U(\qn))$ and $\Psi_d(\HC_d)$ satisfy the double
centralizer property, i.e.,
\begin{align*}
\Omega_{d}(U(\qn))=&\End_{\HC_d}(V^{\otimes d}),\\
&\End_{\qn}(V^{\otimes d}) = \Psi_d(\HC_d).
\end{align*}
\end{theorem}

\begin{proof}
Write $\ga = \qn$. We will denote by $Q(V)$ the associative
subalgebra of endomorphisms on $V$ which super-commute with the
linear operator $P$. By Lemma~\ref{lem:commQ}, we have
$\Omega_{d}(U(\ga)) \subseteq \End_{{\HC}_d}(V^{\otimes d})$.

We shall proceed to prove that $\Omega_{d}(U(\ga)) \supseteq
\End_{{\HC}_d}(V^{\otimes d})$. By examining the actions of $\Cl_d$
on $V^{\otimes d}$, we see that the natural isomorphism $\End
(V)^{\otimes d} \cong \End (V^{\otimes d})$ allows us to identify
$\End_{\Cl_d} (V^{\otimes d}) \equiv Q(V)^{\otimes d}$. As we recall
$\HC_d =\Cl_d \rtimes \mf S_d$, this further leads to the
identification $\End_{\HC_d} (V^{\otimes d}) \equiv \text{Sym}^d
(Q(V))$, the space of $\mf S_d$-invariants in $Q(V)^{\otimes d}$.

Denote by $Y_k$, $1\le k \le d$, the $\C$-span of the
supersymmetrization
$$
\omega (x_1,\ldots, x_k) := \sum_{\sigma \in \mf S_d} \sigma .
(x_1 \otimes \ldots \otimes x_k \otimes 1^{d-k}),
$$
for all $x_i \in Q(V)$. Note that $Y_d = \text{Sym}^d (Q(V))
\equiv \End_{\HC_d} (V^{\otimes d})$.

Let $\tilde{x} = \Omega (x) =\sum_{i=1}^d 1^{i-1} \otimes x \otimes
1^{d-i}$, for $x\in \ga =Q(V)$, and denote by $X_k$, $1\le k \le d$,
the $\C$-span of $\tilde{x}_1  \ldots \tilde{x}_k$ for all $x_i \in
\qn$.

{\bf Claim}.  We have $Y_k \subseteq X_k$ for
$1 \le k \le d$.

Assuming the claim, we have $\Omega_{d}(U(\ga))=
\End_{{\HC}_d}(V^{\otimes d})= \End_{\mc B}(V^{\otimes d})$, for
$\mc B := \Psi_d(\HC_d)$. Note that the algebra $\HC_d$, and hence
also $\mc B$, are semisimple superalgebras, and so the assumption of
Proposition~\ref{prop:ssCentralizer} is satisfied. Therefore, we
have $\End_{U(\ga)}(V^{\otimes d}) = \Psi_d(\HC_d).$

It remains to prove the Claim by induction on $k$. The case $k=1$ holds,
thanks to $\omega (x) = (d-1)! \tilde{x}$.

Assume that $Y_{k-1} \subseteq X_{k-1}$. Note that $\omega
(x_1,\ldots, x_{k-1}) \cdot \tilde{x}_k \in X_k$. On the other
hand, we have
\begin{align*}
\omega & (x_1,\ldots,   x_{k-1}) \cdot \tilde{x}_k
 \\
 &= \sum_{\sigma \in \mf S_d} \sigma . (x_1\otimes \ldots \otimes
x_{k-1} \otimes 1^{d-k+1}) \cdot   \tilde{x}_k
 \\
&= \sum_{j=1}^d \sum_{\sigma \in \mf S_d} \sigma .
\left((x_1\otimes \ldots \otimes x_{k-1} \otimes 1^{d-k+1}) \cdot
(1^{j-1} \otimes x_k \otimes 1^{d-j}) \right),
\end{align*}
which can be written as a sum $A_1 + A_2$, where
$$A_1 =\sum_{j=1}^{k-1} \omega (x_1, \ldots, x_j x_k, \ldots,
x_{k-1}) \in Y_{k-1},
$$
and
\begin{align*}
A_2 &=\sum_{j=k}^d \sum_{\sigma \in \mf S_d} \sigma .(x_1 \otimes
\ldots \otimes x_{k-1} \otimes 1^{j-k} \otimes x_k \otimes
1^{d-j}) \\
 &= (d-k+1) \omega (x_1, \ldots, x_{k-1}, x_k).
\end{align*}

Note that $A_1 \in X_k$, since $Y_{k-1} \subseteq X_{k-1} \subseteq
X_k$. Hence, $A_2 \in X_k$, and so, $Y_k \subseteq X_k$. This proves
the claim and hence the theorem.
\end{proof}

\begin{theorem}
[Sergeev duality II]  \label{th:QnDecomp}
Let $V =\C^{n|n}$. As a $U(\qn) \otimes \HC_d$-module, we have
\begin{equation}  \label{eq:decomp}
V^{\otimes d} \cong \bigoplus_{\la \in \mc{SP}_d, \ell(\la) \le n}
2^{-\delta(\la)} V(\la) \otimes D^\la.
\end{equation}
\end{theorem}

\begin{proof}
Let $W =V^{\otimes d}$. It follows from the double centralizer
property and the semisimplicity of the superalgebra $\HC_d$ that we have a
multiplicity-free decomposition of  the $(\qn, \HC_d)$-module $W$:
\begin{equation*}
W \cong \bigoplus_{\la \in \mathfrak Q_d(n)} 2^{-\delta (\la)}
V^{[\la]} \otimes D^\la,
\end{equation*}
where $V^{[\la]}$ is some simple $\qn$-module associated to $\la$,
whose highest weight (with respect to the standard Borel) is to be
determined. Also to be determined is the index set $\mathfrak Q_d(n)
=\{ \la \in \mc{SP}_d \mid V^{[\la]} \neq 0\}$.

We shall identify as usual a weight $\mu =\sum_{i=1}^n \mu_i
\varepsilon_i$ occuring in $W$ with a composition $\mu
=(\mu_1,\ldots,\mu_n) \in \mc{CP}_d(n)$. We have the following
weight space decomposition:
\begin{equation}  \label{eq:wtqn}
W =\bigoplus_{\mu \in \mc{CP}_d(n)} W_{\mu},
\end{equation}
where $W_{\mu}$ has a linear basis $e_{i_1} \otimes \ldots \otimes
e_{i_d}$, with the indices satisfying the following equality of
sets:
\begin{equation*}
\{|i_1|, \ldots, |i_d|\} =\{\underbrace{1,\ldots, 1}_{\mu_1},
\ldots, \underbrace{n,\ldots, n}_{\mu_n} \}.
\end{equation*}
We have an $\HC_d$-module isomorphism:
\begin{equation}  \label{eq:ind}
W_{\mu} \cong  M^\mu,
\end{equation}
where we recall $M^\mu$ denotes the permutation $\HC_d$-module
$M^\mu =\HC_d \otimes_{\C \mathfrak S_{\mu}} {\bf 1}_d$.

It follows by Proposition~\ref{prop:triang} and \eqref{eq:ind} that
$V^{[\la]} =\bigoplus_{\mu \in \mc{CP}_d(n), \mu \le \la}
V^{[\la]}_{\mu}$, and hence, $\la \in \Pdn$ if $V^{[\la]} \neq 0$.
Among all such $\mu$, clearly $\la$ corresponds to a highest weight.
Hence, we conclude that $V^{[\la]} =V(\la)$, the simple $\ga$-module
of highest weight $\la$, and that $\mf Q_d(n) =\{\la \in \mc{SP}_d
\mid \ell(\la) \le n\}$. This completes the proof of
Theorem~\ref{th:QnDecomp}.
\end{proof}

\subsection{The irreducible character formula for $\qn$}

A character of a $\qn$-module with weight space decomposition $M
=\oplus M_\mu$ is defined to be
$$
\tr \, x_1^{H_1} \ldots x_n^{H_n}|_M
 = \sum_{\mu =(\mu_1,\ldots, \mu_n)}
 \dim M_\mu \cdot  x_1^{\mu_1}\ldots x_n^{\mu_n}.
$$
\begin{theorem}\label{thm:charqueer}
Let $\la$ be a strict partition of length $\le n$. The character
of the simple $\qn$-module $V(\la)$ is given by
$$
\ch V(\la)
=2^{-\frac{\ell(\la)-\delta(\la)}{2}}Q_{\la}(x_1,\ldots, x_n).
$$
\end{theorem}

\begin{proof}
By \eqref{eq:wtqn} and \eqref{eq:ind}, we have
$$
V^{\otimes d} =\bigoplus_{\mu \in \mc{CP}_d(n)} \ind_{\mathfrak
S_{\mu}}^{\HC_d} {\bf 1}_d.
$$
Applying $\ch^-$ and the trace operator $\tr \, x_1^{H_1} \ldots
x_n^{H_n}$ to this decomposition of $V^{\otimes d}$ simultaneously,
which we will denote by $\text{ch}^2$, we obtain that
\begin{align*}
\sum_d\text{ch}^2 (V^{\otimes d})
 &=\sum_{\mu \in \mc{P}, \ell(\mu) \le n} q_\mu(z) m_\mu(x)
   \\
  &=\prod_{1 \le i \le n, 1\le j} \frac{1+x_iz_j}{1-x_iz_j}
   \\
  &= \sum_{\lambda\in \mc{SP}}2^{-\ell(\lambda)}
  Q_{\lambda}(x_1,\ldots,x_n)Q_{\lambda}(z),
\end{align*}
where the last equation is the Cauchy identity in
Theorem~\ref{qschur:aux12} and the middle equation can be verified
directly.

On the other hand, by applying $\text{ch}^2$ to  \eqref{eq:decomp}
and using \eqref{eq:irredch}, we obtain that
$$
\sum_d\text{ch}^2 (V^{\otimes d}) = \sum_{\la \in \mc{SP}, \ell(\la)
\le n} 2^{-\delta(\la)} \ch V(\la) \cdot
2^{-\frac{\ell(\la)-\delta(\la)}{2}}Q_{\la}(z).
$$
Now the theorem follows by comparing the above two identities
and noting the linear independence of the $Q_{\lambda}(z)$'s.
\end{proof}

\section{The coinvariant algebra and generalizations} \label{sec:coinvariant}

In this section, we formulate a graded regular representation for
$\HC_n$, which is a spin analogue of the coinvariant algebra for
$\mf S_n$. We also study its generalizations which involve the
symmetric algebra $S^*\C^n$ and the exterior algebra $\wedge^*\C^n$.
We solve the corresponding graded multiplicity problems in terms of
specializations of Schur $Q$-functions. In addition, a closed
formula  for the principal specialization $Q_{\xi}(1,t,t^2,\ldots)$
of the Schur $Q$-function is given.

\subsection{A commutative diagram}

Recall a homomorphism $\varphi$
( cf. \cite[III, \S8, Example 10]{Mac})
defined by
\begin{align}
& \varphi: \La \longrightarrow\Ga,   \notag\\
\varphi (p_r) = & \left\{
 \begin{array}{cc}
 2p_r,  &\quad \text{ for $r$ odd}, \\
 0,  &\quad \quad  \text{ otherwise},
 \end{array}
 \right. \label{eqn:mapvarphi}
\end{align}
where $p_r$ denotes the $r$th power sum.  Denote
$$
H(t)=\sum_{n \ge 0} h_n t^n =\prod_i \frac1{1-x_it} =\exp
\Big(\sum_{r \ge 1} \frac{p_r t^r}r \Big).
$$
Note that $Q(t)$ from \eqref{gen.fun.qr} can be rewritten as
$$
Q(t) =\exp \Big(2 \sum_{r \ge 1, r \text{ odd}} \frac{p_r t^r}r \Big),
$$
and so we see that
\begin{equation}   \label{eq:HQ}
\varphi \big(H(t) \big) =Q(t).
\end{equation}
Hence, we have $\varphi (h_n) =q_n$ for all $n\ge 0$, and
 \begin{equation}\label{eq:hq}
 \varphi(h_{\mu})=q_{\mu}, \quad \forall \mu \in \mc P.
 \end{equation}

Given an $\mf S_n$-module $M$, the algebra $\HC_n$ acts naturally on
$\Cl_n \otimes M$, where $\Cl_n$ acts by left multiplication on the
first factor and $\mf S_n$ acts diagonally. We have an isomorphism
of $\HC_n$-modules:
\begin{equation} \label{eq:ind=tensor}
\Cl_n \otimes M \cong \ind^{\HC_n}_{\C \mf S_n} M.
\end{equation}

Following \cite{WW2}, we define a functor for $n \ge 0$
\begin{align*}
\Phi_n:& \mf S_n\text{-}\mf{mod}  \longrightarrow \HC_n\text{-}\mf{mod} \\
\Phi_n & (M) ={\rm ind}^{\HC_n}_{\C \mf S_n} M.
\end{align*}
Such a sequence  $\{\Phi_n\}$ induces a $\Z$-linear map on the
Grothendieck group level:
$$
\Phi: R \longrightarrow  R^-,
$$
by letting $\Phi([M]) = [\Phi_n(M)]$ for $M \in \mf
S_n\text{-}\mf{mod}$.

Recall that $R$ carries a natural Hopf algebra structure with
multiplication given by induction and comultiplication given by
restriction \cite{Ze}. In the same fashion, we can define a Hopf
algebra structure on $R^-$ by induction and restriction. On the
other hand, $\La_\Q \cong \Q[p_1, p_2, p_3, \ldots]$ is naturally a
Hopf algebra, where each $p_r$ is a primitive element, and $\Ga_\Q
\cong \Q[p_1, p_3, p_5, \ldots]$ is naturally a Hopf subalgebra of
$\La_\Q$. The characteristic map $\ch: R_\Q \rightarrow \La_\Q$ is
an isomorphism of Hopf algebras (cf. \cite{Ze}). A similar argument
easily shows that the map $\ch^-: R^-_\Q \rightarrow \Ga_\Q$
is an isomorphism of Hopf algebras.

\begin{proposition} \cite{WW2} \label{prop:commute}
The map $\Phi: R_\Q \rightarrow R^-_\Q$ is a homomorphism of Hopf
algebras. Moreover, we have the following  commutative diagram of
Hopf algebras:
\begin{eqnarray}  \label{eq:CD}
\begin{CD}
 R_\Q @>\Phi>> R^-_\Q  \\
 @V\text{ch}V{\cong}V @V\ch^-V{\cong}V \\
 \La_\Q  @>\varphi>> \Ga_\Q
  \end{CD}
\end{eqnarray}
\end{proposition}

\begin{proof}Using~(\ref{eq:chpermutation})~and~(\ref{eq:hq}) we have
$$
\varphi \big({\rm ch}({\rm ind}^{\C \mf S_n}_{\C \mf
S_{\mu}}\textbf{1}_n) \big)=q_{\mu}.
$$
On the other hand, it follows by~\eqref{eq:spinchpermutation} that
$$
{\rm ch}^-\big(\Phi({\rm ind}^{\C \mf S_n}_{\C \mf
S_{\mu}}\textbf{1}_n) \big) ={\rm ch}^-({\rm ind}^{\HC_n}_{\C \mf
S_{\mu}}{\bf 1}_n)=q_{\mu}.
$$
This establishes the commutative diagram on the level of linear
maps, since  $R_n$ has a basis given by the characters of the
permutation modules ${\rm ind}^{\C \mf S_n}_{\C \mf S_{\mu}}{\bf
1}_n$ for $\mu\in\mc P_n$.

It can be verified easily that $\varphi: \La_\Q \rightarrow \Ga_\Q$
is a homomorphism of Hopf algebras. Since both $\ch$ and $\ch^-$ are
isomorphisms of Hopf algebras, it follows from the commutativity of
\eqref{eq:CD} that $\Phi: R_\Q \rightarrow R^-_\Q$ is a homomorphism
of Hopf algebras.
\end{proof}
We shall use the commutation diagram \eqref{eq:CD} as a bridge to
discuss spin generalizations of some known constructions in the
representation theory of symmetric groups, such as the coinvariant
algebras, Kostka polynomials, etc.

\subsection{The coinvariant algebra for $\mf S_n$}

The symmetric group $\mf S_n$ acts on $V=\C^n$ and then on the
symmetric algebra $S^*V,$ which is identified with $\C[x_1,\ldots,
x_n]$ naturally. It is well known that the algebra of $\mf
S_n$-invariants on $S^*V$, or equivalently $\C[x_1,\ldots,x_n]^{\mf
S_n}$, is a polynomial algebra in $e_1,e_2, \ldots, e_n$, where $e_i
=e_i[x_1,\ldots,x_n]$ denotes the $i$th elementary symmetric
polynomial.

For a partition $\la=(\la_1, \la_2, \ldots)$ of $n$, denote
\begin{eqnarray}   \label{n lambda}
n(\lambda)=\sum_{i\geq 1}(i-1)\lambda_i.
\end{eqnarray}
We also denote by  $h_{ij} =\la_i +\la_j' -i-j+1$ the {\em hook
length} and $c_{ij}=j-i$ the {\em content} associated to a cell
$(i,j)$ in the Young diagram of $\la$.

\begin{example}
For $\la =(4,3,1)$, the hook lengths are listed in the corresponding
cells as follows:
$$
\young(6431,421,1)
$$
In this case, $n(\la) =5.$
\end{example}

Denote by $t^\bullet =(1,t,t^2,\ldots)$ for a formal variable $t$.
We have the following principal specialization of the $r$th
power-sum:
$$
p_r(t^\bullet) =\frac1{1-t^r}.
$$
The following well-known formula (cf. \cite[I, \S 3, 2]{Mac}) for
the principal specialization of $s_\la$ can be proved in a multiple of
ways:
\begin{equation}  \label{eq:Schurspec}
s_\la(t^\bullet)  =\frac{t^{n(\la)}} {\prod_{(i,j)\in
\la}(1-t^{h_{ij}})}.
\end{equation}

Write formally
$$
S_t V = \sum_{j \ge 0} t^j (S^jV).
$$
Consider the graded multiplicity of a given Specht module $S^\la$
for a partition $\la$ of $n$ in the graded algebra $S^*V$, which is
by definition
$$
f_\la (t):
= \dim \Hom_{\mf S_n} (S^\la, S_tV).
$$
The {\em coinvariant algebra} of $\mf S_n$ is defined to be
$$
(S^*V)_{\mf S_n} =S^*V/ I,
$$
where $I$ denotes the ideal generated by $e_1, \ldots, e_n$. By a
classical theorem of Chevalley (cf. \cite{Ka}), we have an
isomorphism of $\mf S_n$-modules:
\begin{eqnarray}  \label{eq:coinv=}
S^*V \cong (S^*V)_{\mf S_n} \otimes (S^*V)^{\mf S_n}.
\end{eqnarray}

Define the polynomial
$$
f^\la (t):
= \dim \Hom_{\mf S_n} (S^\la, (S_tV)_{\mf S_n}).
$$
Closed formulas for $f_\la (t)$ and $f^\la (t)$ in various forms
have been well known (cf. Steinberg~\cite{S}, Lusztig \cite{Lu1},
Kirillov \cite{Ki}). Following Lusztig, $f^\la(t)$ is called the
{\em fake degree} in connection with Hecke algebras and finite
groups of Lie type.  We will skip a proof of Theorem~\ref{th:grmult}
below, as it can be read off by specializing $s=0$ in the proof of
Theorem~\ref{thm:KPak}. Thanks to \eqref{eq:coinv=}, the formula
\eqref{eq:coinv} is equivalent to \eqref{eq:multSV}.

\begin{theorem}  \label{th:grmult}
The following formulas for the graded multiplicities hold:
\begin{align}
f_\la (t) &=\frac{t^{n(\la)}} {\prod_{(i,j)\in \la}(1-t^{h_{ij}})},
 \label{eq:multSV} \\
f^\la (t)&=\frac{t^{n(\la)}(1-t)(1-t^2)\ldots (1-t^n)}
{\prod_{(i,j)\in \la}(1-t^{h_{ij}})}.  \label{eq:coinv}
\end{align}
\end{theorem}

 Note that the dimension of the Specht module $S^\la$ is given by the
hook formula
$$f^\la(1) =\frac{n!}{ \prod_{(i,j)\in \la} h_{ij}}.$$
Setting $t \mapsto 1$ in \eqref{eq:coinv} confirms that the
coinvariant algebra $(S^*V)_{\mf S_n}$ is a regular representation
of $\mf S_n$.

\subsection{The graded multiplicity in $S^*V\otimes\wedge^*V$
and $S^*V\otimes S^*V$}

Recall that $x=\{x_1,x_2,\ldots\}$ and $y=\{y_1, y_2,\ldots\}$ are
two independent sets of variables. Recall a well-known formula
relating Schur and skew Schur functions:
$s_{\la}(x,y)=\sum_{\rho\subseteq\la}s_{\rho}(x)s_{\la/\rho}(y).$
For $\la\in\mc P$, the {\em super Schur function} (also known as
{\em hook Schur function}) $hs_{\la}(x;y)$ is defined as
\begin{equation}\label{eq:hook Schur}
hs_{\la}(x;y)=\sum_{\rho\subseteq\la}s_{\rho}(x)s_{\la'/\rho'}(y).
\end{equation}
In other words, $hs_{\la}(x;y) =\omega_y (s_\la(x,y))$, where
$\omega_y$ is the standard involution on the ring of symmetric
functions in $y$. We refer to \cite[Appendix~A]{CW} for more detail.

Since $\omega_y(p_r(y))=(-1)^{r-1}p_r(y)$, $ p_r(x;y):=
\omega_y(p_r(x,y))$ for $r\geq 1$ is given by
$$
p_r(x;y) =\sum_i x_i^r - \sum_j (-y_j)^r.
$$
Applying $\omega_y$ to the Cauchy identity \eqref{eq:Cauchy} gives
us
\begin{equation}\label{eq:hookCauchy}
\sum_{\la\in\mc P}s_{\la}(z)hs_{\la}(x;y)
=\frac{\prod_{j, k}(1+y_jz_k)}{\prod_{i, k}(1-x_iz_k)}.
\end{equation}
Let $a, b$ be variables. The formula in \cite[Chapter I, \S 3,
3]{Mac} can be interpreted as the specialization of $hs_{\la}(x;y)$
at $x=at^{\bullet}$ and $y=bt^{\bullet}$:
\begin{equation}\label{eq:MacEx3.3}
hs_{\la}(at^{\bullet};bt^{\bullet})
=t^{n(\la)}\prod_{(i,j)\in\la}\frac{a+bt^{c_{ij}}}{1-t^{h_{ij}}}.
\end{equation}

The $\mf S_n$-action on $V=\C^n$ induces a natural $\mf S_n$-action on the exterior algebra
$$
\wedge^{*}V
=\bigoplus_{j=0}^n \wedge^{j}V.
$$
This gives rise to a $\Z_+ \times \Z_+$ bi-graded $\C \mf S_n$-module structure on
\begin{align*}
S^*V\otimes\wedge^*V=\bigoplus_{i\geq 0, 0\leq
j\leq n}S^iV\otimes\wedge^jV.
\end{align*}
Let $s$ be a variable and write formally
$$
\wedge_sV=\sum^n_{j= 0 }s^j (\wedge^jV).
$$
Let $\widehat{f}_{\la}(t,s)$ be the bi-graded multiplicity of the
Specht module $S^{\la}$ for $\la\in \mc{P}_n$ in $S^*V\otimes
\wedge^*V$, which is by definition
$$
\widehat{f}_{\la}(t,s)=\dim\Hom_{\mf S_n}(S^{\la}, S_tV\otimes\wedge_sV).
$$

\begin{theorem} \label{thm:KPak}
Suppose $\la\in\mc P_n$. Then
\begin{enumerate}
\item
$\widehat{f}_{\la}(t,s) =hs_\la(t^\bullet;st^{\bullet})$.

\item
$\widehat{f}_{\la}(t,s)
=\frac{t^{n(\la)}\prod_{(i,j)\in\la}(1+st^{c_{ij}})}{\prod_{(i,j)\in\la}(1-t^{h_{ij}})}
=\frac{\prod_{(i,j)\in\la}(t^{i-1}+st^{j-1})}{{\prod_{(i,j)\in\la}(1-t^{h_{ij}})}}.$
\end{enumerate}
\end{theorem}

Formula (2) above in the second expression for
$\widehat{f}_{\la}(t,s)$ was originally established with a bijective
proof by Kirillov-Pak \cite{KP}, with $(t^i+st^j)$ being corrected
as $(t^{i-1}+st^{j-1})$ above. Our proof below is different, making
clear the connection with the specialization of super Schur
functions.

\begin{proof}
It suffices to prove (1), as (2) follows  from \eqref{eq:MacEx3.3}
and (1).

By the definition of $\widehat{f}_{\la}(t,s)$ and the characteristic
map, we have
\begin{equation}\label{eq:grSVwedgeV}
\ch (S_tV\otimes \wedge_sV)=\widehat{f}_{\la}(t,s)s_{\la}(z).
\end{equation}

Take $\sigma =(1,2, \ldots, \mu_1)(\mu_1+1, \ldots, \mu_1+\mu_2)
\cdots$ in $\mf S_n$ of type $\mu =(\mu_1, \mu_2, \ldots, \mu_\ell)$
with $\ell=\ell(\mu)$. Note that $\sigma$ permutes the monomial
basis for $S^*V$, and the monomials fixed by $\sigma$ are of the
form
$$
(x_1x_2\ldots x_{\mu_1})^{a_1} (x_{\mu_1+1} \ldots
x_{\mu_1+\mu_2})^{a_2} \ldots  (x_{\mu_1+\ldots +\mu_{\ell-1}+1}
\ldots x_n)^{a_\ell},
$$
where $a_1\ldots, a_\ell \in \Z_+$. Let us denote by $dx_1 \ldots,
dx_n$ the generators for $\wedge^*V$. Similarly, the exterior
monomials fixed by $\sigma$ up to a sign are of the form
$$
(dx_1dx_2\ldots dx_{\mu_1})^{b_1} (dx_{\mu_1+1} \ldots
dx_{\mu_1+\mu_2})^{b_2} \ldots  (dx_{\mu_1+\ldots +\mu_{\ell-1}+1}
\ldots dx_n)^{b_\ell},
$$
where $b_1\ldots, b_\ell \in \{0,1\}$. The sign here is
$(-1)^{\sum_i b_i(\mu_i-1)}$.

From these we deduce that
\begin{align*}
\tr \sigma |_{S_tV\otimes \wedge_sV}
&= \sum_{a_1,\ldots, a_\ell\geq0,(b_1,\ldots, b_\ell)\in\Z_2^n }
t^{\sum_{i=1}^\ell a_i\mu_i}s^{\sum_{i=1}^\ell b_i\mu_i}(-1)^{\sum_i b_i(\mu_i-1)}\\
&=\frac{(1-(-s)^{\mu_1})(1-(-s)^{\mu_2})\ldots
(1-(-s)^{\mu_\ell})}{(1-t^{\mu_1})(1-t^{\mu_2})\ldots
(1-t^{\mu_\ell})}.
\end{align*}
We shall denote $[u^n]g(u)$ the coefficient of $u^n$ in a power
series expansion of $g(u)$. Applying the characteristic map $\ch$,
we obtain that
\begin{align}  \label{eq:chSVwedgeV}
&\ch (S_tV\otimes \wedge_sV) \\
&= \sum_{\mu \in\mc P_n} z_\mu^{-1}
\frac{(1-(-s)^{\mu_1})(1-(-s)^{\mu_2})\ldots
(1-(-s)^{\mu_\ell})}{(1-t^{\mu_1})(1-t^{\mu_2})\ldots
(1-t^{\mu_\ell})} p_\mu
  \notag\\
&= [u^n]  \sum_{\mu \in \mc P} z_\mu^{-1} u^{|\mu|}
p_\mu(t^\bullet;st^{\bullet}) p_\mu
  \notag \\
&= [u^n]  \prod_{j \ge 0} \prod_i \frac{1+ust^jz_i}{1- ut^j z_i}
  \notag \\
&= \sum_{\la \in\mc{P}_n} hs_\la(t^\bullet;st^{\bullet}) s_\la(z),
\notag
\end{align}
where the last equation used the Cauchy
identity~\eqref{eq:hookCauchy}. By comparing \eqref{eq:grSVwedgeV}
and \eqref{eq:chSVwedgeV}, we have proved (1).
\end{proof}

We can also consider the bi-graded multiplicity of Specht modules
$S^{\la}$ for $\la\in\mc P_n$ in the $\C\mf S_n$-module $S^*V\otimes S^*V$,
which by definition is
$$
\widetilde{f}_{\la}(t,s)=\dim\Hom_{\mf S_n}(S^{\la}, S_tV\otimes S_sV).
$$
\begin{theorem}  \cite{BL} \label{thm:grSV2}
We have $\widetilde{f}_{\la}(t,s) =s_{\la}(t^{\bullet}s^{\bullet}),$
for $\la\in\mc P$, where $t^{\bullet}s^{\bullet}$ indicates the
variables $\{t^js^k \mid j, k\geq 0\}.$
\end{theorem}
\begin{proof}
By the definition of $\widetilde{f}_{\la}(t,s)$, we have
\begin{equation}\label{eq:grSV2}
\ch (S_tV\otimes S_sV)=\widetilde{f}_{\la}(t,s)s_{\la}(z).
\end{equation}
Arguing similarly  as in the proof of Theorem~\ref{thm:KPak}, one
deduces that
\begin{align}  \label{eq:chSVSV2}
&\ch (S_tV\otimes S_sV) \\
&= \sum_{\mu \in \mc{P}_n} z_\mu^{-1} \frac1{(1-t^{\mu_1})
(1-t^{\mu_2})\ldots
(1-t^{\mu_\ell})}\cdot\frac1{(1-s^{\mu_1})(1-s^{\mu_2})\ldots
(1-s^{\mu_\ell})} p_\mu
  \notag\\
&= [u^n]  \sum_{\mu \in \mc P} z_\mu^{-1} u^{|\mu|} p_\mu(t^\bullet s^\bullet) p_\mu
  \notag \\
&= [u^n]  \prod_{j,k \ge 0} \prod_i \frac1{1- ut^js^k z_i}
  \notag \\
&= \sum_{\la \in\mc{P}_n} s_\la(t^\bullet s^{\bullet}) s_\la(z),
\notag
\end{align}
where the last equation used the Cauchy identity \eqref{eq:Cauchy}.
The theorem is proved by comparing \eqref{eq:grSV2} and
\eqref{eq:chSVSV2}.
\end{proof}

\begin{remark}
By ~\eqref{eq:coinv=}
and Theorem~\ref{thm:grSV2}, the graded multiplicity of $S^{\la}$
for $\la\in\mc P_n$ in the  $\mf S_n$-module $(S^*V)_{\mf
S_n}\otimes (S^*V)_{\mf S_n}$ is $\prod_{r=1}^n(1-t^r)(1-s^r)
s_{\la}(t^{\bullet}s^{\bullet}).$ This recovers Bergeron-Lamontagne
\cite [Theorem 6.1 or (6.4)]{BL}.
\end{remark}
\subsection{The spin coinvariant algebra for $\HC_n$}\label{sub:spininv}

Suppose that the main diagonal of the Young diagram $\la$ contains
$r$ cells. Let $\alpha_i=\lambda_i-i$ be the number of cells in
the $i$th row of $\lambda$ strictly to the right of $(i,i)$, and
let $\beta_i=\lambda_i'-i$ be the number of cells in the $i$th
column of $\lambda$ strictly below $(i,i)$, for $1\leq i\leq r$.
We have $\alpha_1>\alpha_2>\cdots>\alpha_r\geq0$ and
$\beta_1>\beta_2>\cdots>\beta_r\geq0$. Then the Frobenius notation
for a partition is
$\lambda=(\alpha_1,\ldots,\alpha_r|\beta_1,\ldots,\beta_r)$. For
example, if $\lambda=(5,4,3,1)$ whose corresponding Young diagram is
$$
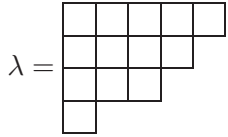
\la =\young(\,\,\,\,\,,\,\,\,\,,\,\,\,,\,)
$$
then $\alpha =(4,2,0),
\beta=(3,1,0)$ and hence $\lambda=(4,2,0|3,1,0)$ in Frobenius
notation.

Suppose that $\xi$ is a strict partition of $n$. Let $\xi^*$ be the
associated shifted diagram, that is,
$$
\xi^*=\{(i,j)~|~1\leq i\leq l(\lambda), i\leq j\leq\lambda_i+i-1
\}
$$
which is obtained from the ordinary Young diagram by shifting the
$k$th row to the right by $k-1$ squares, for each $k$. Denoting
$\ell(\xi)=\ell$, we define the {\em double partition}
$\widetilde{\xi}$ to be $\widetilde{\xi}=(\xi_1,\ldots,\xi_\ell|
\xi_1-1,\xi_2-1,\ldots,\xi_\ell-1)$ in Frobenius notation. Clearly,
the shifted diagram $\xi^*$ coincides with the part of
$\widetilde{\xi}$ that lies above the main diagonal. For each cell
$(i,j)\in \xi^*$, denote by $h^*_{ij}$ the associated hook length in
the Young diagram $\widetilde{\xi}$, and set the content
$c_{ij}=j-i$.

\begin{example}
Let $\xi = (4, 3, 1)$. The corresponding shifted diagram $\xi^*$ and
double diagram $\widetilde\xi$ are
$$
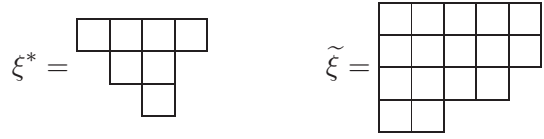
\xi^*=\young(\,\,\,\,,:\,\,,::\,)
\qquad \qquad
\widetilde{\xi}=\young(\,\,\,\,\,,\,\,\,\,\,,\,\,\,\,,\,\,)
$$
The contents of $\xi$ are listed in the corresponding cell of $\xi^*$ as follows:
$$
\young(0123,:012,::0)
$$
The shifted hook lengths for each cell in $\xi^*$ are  the
usual hook lengths for the corresponding cell in $\xi^*$, as
part of the double diagram $\widetilde \xi$, as follows:
$$
\young(\,7542,\,\,431,\,\,\,1,\,\,)
\qquad \qquad \young(7542,:431,::1)
$$
\end{example}

Since $(S^*V)_{\mf S_n}$ is a regular representation of $\mf S_n$,
$\Cl_n \otimes (S^*V)_{\mf S_n}$ is a regular representation of
$\HC_n$ by \eqref{eq:ind=tensor}. Denote by
\begin{align*}
d_\xi (t) &= \dim \Hom_{\HC_n} (D^\xi, \Cl_n \otimes S_tV), \\
d^\xi (t) &= \dim \Hom_{\HC_n} (D^\xi, \Cl_n \otimes (S_tV)_{\mf
S_n}).
\end{align*}
The polynomial $d^\xi(t)$ will be referred to as the {\em spin fake
degree} of the simple $\HC_n$-module $D^\xi$, and it specializes to
the degree of $D^\xi$ as $t$ goes to $1$. Note
$d^\xi(t)=d_\xi(t)\prod_{r=1}^n(1-t^r).$

\begin{theorem} \cite{WW1}  \label{th:S(V)}
Let $\xi$ be a strict partition of $n$. Then
\begin{enumerate}
\item
$d_\xi (t) = 2^{-\frac{\ell(\xi)-\delta(\xi)}{2}} Q_\xi(t^\bullet)$.

\item
$d^\xi (t) = 2^{-\frac{\ell(\xi)-\delta(\xi)}{2}} t^{n(\xi)}
\frac{\prod_{r=1}^n(1-t^r)\prod_{(i,j)\in \xi^*}(1+t^{c_{ij}})}
{\prod_{(i,j)\in \xi^*}(1-t^{h^*_{ij}})}.$
\end{enumerate}
\end{theorem}

\begin{proof}
Let us first prove (1). By Lemma~\ref{lem:basicSpinChar}, the value
of the character $\xi^n$ of the  basic spin $\HC_n$-module at an
element $\sigma \in \mf S_n$ of cycle type $\mu \in \mc{OP}_n$ is
$\xi^n_\mu =2^{\ell(\mu)}.$ When combining with the computation in
the proof of Theorem~\ref{thm:KPak}, we have
$$
\tr \sigma |_{\Cl_n \otimes S_tV}
=\frac{2^{\ell(\mu)}}{(1-t^{\mu_1})(1-t^{\mu_2})\ldots
(1-t^{\mu_\ell})}.
$$

Applying the characteristic map $\ch^-: R^- \rightarrow \Gamma_\Q$,
we obtain that
\begin{align}  \label{eq:chClSV2}
\ch^- (\Cl_n \otimes S_tV) &= \sum_{\mu \in \mc{OP}_n} z_\mu^{-1}
\frac{2^{\ell(\mu)}}{(1-t^{\mu_1})(1-t^{\mu_2})\ldots
(1-t^{\mu_\ell})} p_\mu
  \\
&= [u^n]  \sum_{\mu \in \mc{OP}} {2^{\ell(\mu)}} z_\mu^{-1}
u^{|\mu|} p_\mu(t^\bullet) p_\mu
  \notag \\
&= [u^n]  \prod_{m \ge 0} \prod_i \frac{1+ ut^m z_i}{1- ut^m z_i}
  \notag \\
&= \sum_{\la \in \mc{SP}_n} 2^{-\ell(\xi)} Q_\xi(t^\bullet)
Q_\xi(z), \notag
\end{align}
where the last two equations used \eqref{eq:pp} and the Cauchy
identity from Theorem~\ref{qschur:aux12}. It also follows by
\eqref{eq:irredch} and the definition of $d_\xi(t)$ that
$$
\ch^- (\Cl_n \otimes S_tV) =\sum_{\xi \in \mc{SP}_n}
2^{-\frac{\ell(\xi)+\delta(\xi)}{2}} d_\xi(t) Q_\xi(z).
$$

Comparing these two different expressions for $\ch^- (\Cl_n \otimes
S_tV)$ and noting the linear independence of $Q_\xi(z)$, we have
proved (1). Part (2) follows by (1) and applying
Theorem~\ref{th:Qspec} below.
\end{proof}

\begin{theorem} \label{th:Qspec}
The following holds for any $\xi\in\mc{SP}$:
$$
Q_\xi(t^\bullet)  =t^{n(\xi)}\prod_{(i,j)\in\xi^*}\frac{1+t^{c_{ij}}}{1-t^{h^*_{ij}}}
=\prod_{(i,j)\in\xi^*}\frac{t^{i-1}+t^{j-1}}{1-t^{h^*_{ij}}}.
$$
\end{theorem}

Theorem~\ref{th:Qspec} in a different form was proved by Rosengren
\cite{Ro} using formal Schur function identities, and in the current form was proved in
\cite[Section 2]{WW1} by setting up a bijection between marked
shifted tableaux and new combinatorial objects called colored
shifted tableaux. The following new proof follows an approach
suggested by a referee of \cite{WW1}.

\begin{proof}
Recall the homomorphism $\varphi:\La\rightarrow\Ga$  from
\eqref{eqn:mapvarphi}. For $\la\in\mc P$, let $S_{\la} \in\Gamma$ be
the determinant (cf. \cite[III, \S8, 7(a)]{Mac})
$$
S_{\la}
={\rm det}(q_{\la_i-i+j}).
$$
It follows by  the Jacobi-Trudi identity for $s_\la$ and \eqref{eq:hq} that
\begin{equation}\label{eq:imgSchur}
\varphi(s_{\la})=S_{\la}.
\end{equation}
Applying $\varphi$ to the Cauchy identity \eqref{eq:Cauchy}
 and using \eqref{eq:HQ} with $t =z_i$, we obtain that
\begin{equation}\label{eqn:Cauchy2}
\prod_{i, j\geq 1}\frac{1+x_iz_j}{1-x_iz_j}=\sum_{\la\in\mc P}s_{\la}(z)S_{\la}(x).
\end{equation}
This together with \eqref{eq:hookCauchy} implies that
\begin{equation}\label{eq:Shook}
S_{\la}(x)=hs_{\la}(x;x).
\end{equation}

Recall the definition of the double diagram $\widetilde{\xi}$ from
Section \ref{sub:spininv}. It follows from \cite[Theorem 3]{You}
(cf. \cite[III, \S 8, 10]{Mac}) that
$$
\varphi(s_{ \widetilde{\xi}})=2^{-\ell(\xi)}Q^2_{\xi}, \quad \forall
\xi\in\mc{SP}_n,
$$
and hence by \eqref{eq:imgSchur} we have
\begin{equation}\label{QsqS}
Q^2_{\xi}=2^{\ell(\xi)}S_{\widetilde{\xi}}, \quad \forall
\xi\in\mc{SP}_n.
\end{equation}
By \eqref{eq:MacEx3.3}  and \eqref{eq:Shook}, we have
\begin{equation}\label{spec.hstt}
S_{\widetilde{\xi}}(t^{\bullet})
=t^{n(\widetilde{\xi})}\prod_{(i,j)\in\widetilde{\xi}}\frac{1+t^{c_{ij}}}{1-t^{h_{ij}}}
=\prod_{(i,j)\in\widetilde{\xi}}\frac{t^{i-1}+t^{j-1}}{1-t^{h_{ij}}}.
\end{equation}

Let $\ell=\ell(\xi)$. Denote by $H_r$ the $r$th hook which consists
of the cells below or to the right of a given cell $(r,r)$ on the
diagonal of $\widetilde{\xi}$ (including $(r,r)$). For a fixed $r$,
we have
\begin{align*}
\prod_{(i,j)\in H_r}(t^{i-1}+t^{j-1})
&=\frac{(t^{r-1}+t^{\xi_r+r-1})}{t^{r-1}+t^{r-1}}\prod_{r\leq j\leq\xi_r+r-1}(t^{r-1}+t^{j-1})^2\\
&=\frac{1+t^{\xi_r}}{2}\prod_{(r,j)\in\xi^*}(t^{r-1}+t^{j-1})^2.
\end{align*}
Hence,
\begin{align}
\prod_{(i,j)\in\widetilde{\xi}}(t^{i-1}+t^{j-1})
&=\prod_{1\leq r\leq\ell}\prod_{(i,j)\in H_r}(t^{i-1}+t^{j-1}) \label{dcontent}  \\
&=2^{-\ell}\prod^{\ell}_{r=1}(1+t^{\xi_r})\prod_{(i,j)\in\xi^*}(t^{i-1}+t^{j-1})^2.
  \notag
\end{align}

On the other hand, for a fixed $i$,  the hook lengths $h_{ij}$ for
$(i,j)\in\widetilde{\xi}$  with $j>i$ are exactly the hook lengths
$h^*_{ij}$ for $(i,j)\in\xi^*$, which are
$1,2,\ldots,\xi_i,\xi_i+\xi_{i+1},\xi_i+\xi_{i+2},\ldots,\xi_{i}+\xi_{\ell}$
with exception
$\xi_i-\xi_{i+1},\xi_i-\xi_{i+2},\ldots,\xi_{i}-\xi_{\ell}$ (cf.
\cite[III, \S 8, 12]{Mac}). Meanwhile, one can deduce that the hook
lengths $h_{ki}$ for $(k,i)\in\widetilde{\xi}$ with $k\geq i$ for a
given $i$ are $1,2,\ldots,\xi_i-1, 2\xi_i,
\xi_i+\xi_{i+1},\xi_i+\xi_{i+2},\ldots,\xi_{i}+\xi_{\ell}$ with
exception
$\xi_i-\xi_{i+1},\xi_i-\xi_{i+2},\ldots,\xi_{i}-\xi_{\ell}$. This
means
\begin{equation}  \label{eq:thij}
\prod_{(i,j)\in\widetilde{\xi}}(1-t^{h_{ij}})
=\prod_{(i,j)\in\xi^*}(1-t^{h^*_{ij}})^2\prod^{\ell}_{i=1}\frac{1-t^{2\xi_i}}{1-t^{\xi_i}}
=\prod_{(i,j)\in\xi^*}(1-t^{h^*_{ij}})^2\prod^{\ell}_{i=1}(1+t^{\xi_i}).
\end{equation}

Now the theorem follows from \eqref{QsqS}, \eqref{spec.hstt},
\eqref{dcontent}, and \eqref{eq:thij}.
\end{proof}

\begin{remark}
The formulas in Theorem \ref{th:S(V)} appear to differ by a factor
$2^{\delta(\xi)}$ from \cite[Theorem~A]{WW1} because of a different
formulation due to the type $\texttt Q$ phenomenon.
\end{remark}

\subsection{The graded multiplicity in $\Cl_n\otimes S^*V\otimes\wedge^*V$
and $\Cl_n\otimes S^*V\otimes S^*V$} Similarly, we can consider the
multiplicity of $D^{\xi}$ for $\xi\in\mc{SP}_n$ in the bi-graded
$\HC_n$-modules $\Cl_n\otimes S^*V\otimes\wedge^*V$ and
$\Cl_n\otimes S^*V\otimes S^*V$, and let
\begin{align*}
\widehat{d}_{\xi}(t,s)&=\dim\Hom_{\HC_n}(D^{\xi}, \Cl_n\otimes S_tV\otimes\wedge_sV),\\
\widetilde{d}_{\xi}(t,s)&=\dim\Hom_{\HC_n}(D^{\xi}, \Cl_n\otimes S_tV\otimes S_sV).
\end{align*}

\begin{theorem}
Suppose $\xi\in\mc{SP}_n$. Then
\begin{enumerate}
\item $\widehat{d}_{\xi}(t,s)=2^{-\frac{\ell(\la)-\delta(\la)}{2}}Q_{\xi}(t^{\bullet},st^{\bullet})$.
\item $\widetilde{d}_{\xi}(t,s)=2^{-\frac{\ell(\la)-\delta(\la)}{2}}Q_{\xi}(t^{\bullet}s^{\bullet})$.
\end{enumerate}
\end{theorem}
Part~(1) here is \cite[Theorem C]{WW1} with a different proof, while
(2) is new.

\begin{proof}
By Lemma~\ref{lem:basicSpinChar} and the computation at the
beginning of the proof of Theorem~\ref{thm:KPak}, we have
$$
\tr \sigma |_{\Cl_n\otimes S_tV\otimes \wedge_sV}
=2^{\ell(\mu)}\cdot\frac{(1-(-s)^{\mu_1})(1-(-s)^{\mu_2})\ldots
(1-(-s)^{\mu_\ell})}{(1-t^{\mu_1})(1-t^{\mu_2})\ldots
(1-t^{\mu_\ell})},
$$
for any $\sigma\in\mf S_n$ of cycle type
$\mu=(\mu_1,\mu_2,\ldots)\in \mc{OP}_n$. Applying the characteristic
map $\ch^-: R^- \rightarrow \Gamma_\Q$, we obtain that
\begin{align}  \label{eq:chClSVwedgeV}
\ch^- &(\Cl_n\otimes S_tV\otimes \wedge_sV)
 \\
&= \sum_{\mu \in
\mc{OP}_n} z_\mu^{-1}
\frac{2^{\ell(\mu)}(1-(-s)^{\mu_1})(1-(-s)^{\mu_2})\ldots
(1-(-s)^{\mu_\ell})}{(1-t^{\mu_1})(1-t^{\mu_2})\ldots
(1-t^{\mu_\ell})} p_\mu
   \notag  \\
&= [u^n]  \sum_{\mu \in \mc{OP}} {2^{\ell(\mu)}} z_\mu^{-1}
u^{|\mu|} p_\mu(t^\bullet;st^{\bullet})p_\mu
  \notag \\
&= [u^n]  \prod_{j \ge 0} \prod_i \frac{1+ ut^j z_i}{1- ut^j z_i}\frac{1+ust^jz_i}{1-ust^jz_i}
  \notag \\
&= \sum_{\la \in \mc{SP}_n} 2^{-\ell(\xi)} Q_\xi(t^\bullet, st^{\bullet}) Q_\xi(z),
\notag
\end{align}
where the last two equalities used \eqref{eq:pp} and Cauchy identity
from Theorem~\ref{qschur:aux12}. It follows by \eqref{eq:irredch}
and the definition of $\widehat{d}_\xi(t,s)$ that
$$
\ch^- (\Cl_n\otimes S_tV\otimes \wedge_sV)
 =\sum_{\xi \in \mc{SP}_n} 2^{-\frac{\ell(\xi)+\delta(\xi)}{2}} \widehat{d}_\xi(t,s) Q_\xi(z).
$$
Comparing these two different expressions for $\ch^- (\Cl_n\otimes
S_tV\otimes \wedge_sV)$ and noting the linear independence of the
$Q_\xi(z)$'s, we prove (1).

Using a similar argument, one can verify  (2)  with the calculation
of the character values of $S^*V\otimes S^*V$ in the proof of
Theorem~\ref{thm:grSV2} at hand.
\end{proof}

\begin{remark}
It will be  interesting to find closed formulas for $s_\la(t^\bullet
s^\bullet)$ and $Q_\xi(t^\bullet, st^\bullet)$.
\end{remark}

\section{Spin Kostka polynomials}\label{sec:spinKostka}

In this section, following our very recent work \cite{WW2} we
introduce the spin Kostka polynomials, and show that the spin Kostka
polynomials enjoy favorable properties parallel to the  Kostka
polynomials. Two interpretations of the spin Kostka polynomials in
terms of graded multiplicities in the representation theory of
Hecke-Clifford algebras and $q$-weight multiplicity for the queer
Lie superalgebras are presented.

\subsection{The  ubiquitous Kostka polynomials}
 For $\la, \mu\in\mc P$, let $K_{\la\mu}$ be the Kostka number
which counts the number of semistandard tableaux of shape $\la$ and
weight $\mu$. We write $|\la|=n$ for $\la \in \mc P_n$. The
dominance order on $\mc P$ is defined by letting
$$
\la \geq \mu \Leftrightarrow |\la| =|\mu| \text{ and } \la_1+\ldots
+\la_i \geq \mu_1 +\ldots +\mu_i, \forall i \geq 1.
$$

Let $\la, \mu\in\mc P$.  The Kostka(-Foulkes) polynomial
$K_{\la\mu}(t)$ is defined  by
\begin{align}\label{eqn:Kostka}
s_{\la}(x)=\sum_{\mu}K_{\la\mu}(t)P_{\mu}(x;t),
\end{align}
where $P_{\mu}(x;t)$ are the Hall-Littlewood functions (cf.
\cite[III, \S2]{Mac}). The following is a summary of some main
properties of the Kostka polynomials.

\begin{theorem}(cf. \cite[III, \S6]{Mac}) \label{thm:Kostka}
Suppose  $\la,\mu\in\mc P_n$. Then the Kostka polynomials
$K_{\la\mu}(t)$ satisfy the following properties:
\begin{enumerate}
\item $K_{\la\mu}(t)=0$ unless $\la\geq\mu$;  $K_{\la\la}(t)=1$.

\item The degree of $K_{\la\mu}(t)$ is $n(\mu)-n(\la)$.

\item $K_{\la\mu}(t)$ is a polynomial with non-negative integer coefficients.

\item $K_{\la\mu}(1)=K_{\la\mu}$.

\item $K_{(n)\mu}(t)= t^{n(\mu)}$.

\item
$\ds K_{\la(1^n)}=\frac{t^{n(\la')}(1-t)(1-t^2)\cdots
(1-t^n)}{\prod_{(i,j)\in\la}(1-t^{h_{ij}})}$.
\end{enumerate}
\end{theorem}

Let $\mc{B}$ be the flag variety for the general linear group
$GL_n(\C)$. For a partition $\mu$ of $n$,  the Springer fiber
$\mc{B}_{\mu}$ is the subvariety of $\mc B$ consisting of flags
preserved by the Jordan canonical form $J_{\mu}$ of shape $\mu$.
According to the Springer theory, the cohomology group
$H^{\bullet}(\mc{B}_{\mu})$ of $\mc{B}_{\mu}$ with complex
coefficient affords a graded representation of  $\mf S_n$ (which is
the Weyl group of $GL_n(\C)$). Define $C_{\la\mu}(t)$ to be the
graded multiplicity
\begin{align}\label{eqn:mult.C}
C_{\la\mu}(t)=\sum_{i\geq0}t^i~\Hom_{\mf S_n}(S^{\la},
H^{2i}(\mc{B}_{\mu})).
\end{align}

\begin{theorem}(cf. \cite[III, 7, Ex.~8]{Mac}, \cite[(5.7)]{GP})
\label{thm:Kostka gr.mult}
The following holds for $\la,\mu\in\mc P$:
$$
K_{\la\mu}(t)=C_{\la\mu}(t^{-1})t^{n(\mu)}.
$$
\end{theorem}

Denote by $\{\epsilon_1,\ldots, \epsilon_n\}$ the basis dual to the
standard basis $\{E_{ii}~|~1\leq i\leq n\}$ in the standard Cartan
subalgebra of $\mf{gl}(n)$. For $\la, \mu\in\mc P$ with
$\ell(\la)\leq n$ and $\ell(\mu)\leq n$, define the {\em $q$-weight
multiplicity} of weight $\mu$ in an irreducible $\mf{gl}(n)$-module
$L(\la)$ to be
$$
m^{\la}_{\mu}(t)=[e^{\mu}]
\frac{\prod_{\al>0}(1-e^{-\al})}{\prod_{\al>0}(1-te^{-\al})}~{\rm
ch}L(\la),
$$
where the product $\prod_{\al>0}$ is over all positive roots $\{
\epsilon_i -\epsilon_j \mid 1 \le i < j \le n \}$ for $\mf{gl}(n)$
and $[e^{\mu}]f(e^{\epsilon_1},\ldots, e^{\epsilon_n})$ denotes the
coefficient of the monomial $e^{\mu}$ in a formal series
$f(e^{\epsilon_1},\ldots, e^{\epsilon_n})$. A  conjecture of Lusztig
proved by Sato \cite{Ka, Lu2} states that
\begin{equation}\label{eqn:q-weight}
K_{\la\mu}(t)=m^{\la}_{\mu}(t).
\end{equation}
Let $e$ be a regular nilpotent element in the Lie algebra
$\mf{gl}(n)$. For each $\mu\in\mc P$ with $\ell(\mu)\leq n$, define
the Brylinski-Kostant filtration $\{J_{e}^k(L(\la)_{\mu})\}_{k\ge
0}$ on the $\mu$-weight space $L(\la)_{\mu}$ with
\begin{align*}
J_{e}^k(L(\la)_{\mu})=\{v\in L(\la)_{\mu}~|~e^{k+1}v=0\}.
\end{align*}
Define a polynomial $\gamma_{\la\mu}(t)$  by letting
\begin{align*}
\gamma_{\la\mu}(t)=\sum_{k\geq 0}\Big(\dim
J_{e}^k(L(\la)_{\mu})/J_{e}^{k-1}(L(\la)_{\mu})\Big)t^k.
\end{align*}

The following theorem is due to R.~Brylinski (see
\cite[Theorem~3.4]{Br} and (\ref{eqn:q-weight})).

\begin{theorem}\label{thm:jump}
Suppose $\la, \mu\in\mc P$ with $\ell(\la)\leq n$ and $\ell(\mu)\leq
n$. Then we have
$$
K_{\la\mu}(t)=\gamma_{\la\mu}(t).
$$
\end{theorem}

\subsection{The spin Kostka polynomials}

 Denote by $\mathbf{P}'$ the ordered alphabet
$\{1'<1<2'<2<3'<3\cdots\}$. The symbols $1',2',3',\ldots$ are said
to be marked, and we shall denote by $|a|$ the unmarked version of
any $a\in\mathbf{P}'$; that is, $|k'| =|k| =k$ for each $k \in
\N$. For a strict partition $\xi$, a {\it marked shifted tableau}
$T$ of shape $\xi$, or a  {\it marked shifted $\xi$-tableau}
$T$, is an assignment
$T:\xi^*\rightarrow\mathbf{P}'$ satisfying:
\begin{itemize}
\item[(M1)] The letters are weakly increasing along each row and
column.

\item[(M2)] The letters $\{1,2,3,\ldots\}$ are strictly increasing
along each column.

\item[(M3)] The letters $\{1',2',3',\ldots\}$ are strictly
increasing along each row.   \label{M3}
\end{itemize}

For a marked shifted tableau $T$ of shape $\xi$, let $\alpha_k$ be
the number of cells $(i,j)\in \xi^*$ such that $|T(i,j)|=k$ for
$k\geq 1$. The sequence $(\alpha_1,\alpha_2,\alpha_3,\ldots)$ is
called the {\em weight} of $T$. The Schur $Q$-function  associated
to $\xi$ can be interpreted as (see \cite{Sag, St, Mac})
$$
Q_{\xi}(x)=\sum_{T}x^{T},
$$
where the summation is taken over all marked shifted tableaux of
shape $\xi$, and
$x^T=x_1^{\alpha_1}x_2^{\alpha_2}x_3^{\alpha_3}\cdots$ if $T$ has
weight $(\alpha_1,\alpha_2,\alpha_3,\ldots)$. Set
$$
K^-_{\xi\mu}=\# \{T~|~T \text{ is a marked shifted tableau of shape
}\xi\text{ and weight }\mu\}.
$$
Then we have
\begin{equation}\label{eqn:Qmonomial}
Q_{\xi}(x)=\sum_{\mu}K^-_{\xi\mu}m_{\mu}(x),
\end{equation}
where $K_{\xi\mu}^-$ is related to $\widehat{K}_{\xi\mu}$ appearing
in Theorem~\ref{qschur:aux9} by
$$
K_{\xi\mu}^-=2^{\ell(\xi)}\widehat{K}_{\xi\mu}.
$$

\begin{definition}   \cite{WW2}
The {\em spin Kostka polynomials} $K^-_{\xi\mu}(t)$ for $\xi\in\mc
{SP}$ and $\mu\in\mc P$ are given by
\begin{align}\label{eq:spin Kostka}
Q_{\xi}(x)=\sum_{\mu}K^-_{\xi\mu}(t)P_{\mu}(x;t).
\end{align}
\end{definition}

For $\xi\in\mc{SP}$, write
\begin{equation}\label{eqn:QSchur}
Q_{\xi}(x)=\sum_{\la\in\mc P}b_{\xi\la}s_{\la}(x),
\end{equation}
for some suitable structure constants $b_{\xi\la}$.

\begin{proposition} \label{prop:relate}
The following holds for $\xi\in\mc{SP}$ and $\mu\in\mc P$:
$$
K^-_{\xi\mu}(t)=\sum_{\la\in\mc P}b_{\xi\la}K_{\la\mu}(t).
$$
\end{proposition}
\begin{proof}
By (\ref{eqn:Kostka}) and (\ref{eqn:QSchur}), one can deduce that
$$
\sum_{\mu}K^-_{\xi\mu}(t)P_{\mu}(x;t) = \sum_{\la, \mu \in\mc
P_n}b_{\xi\la}K_{\la\mu}(t)P_{\mu}(x;t).
$$
The proposition now follows from the fact that the Hall-Littlewood
functions $P_{\mu}(x;t)$ are linearly independent in $\Z[t]\otimes_{\Z}\La$.
\end{proof}
The usual Kostka polynomial satisfies  that  $K_{\la\mu} (0)
=\delta_{\la\mu}$. It follows from Proposition~\ref{prop:relate}
that
$$
K^-_{\xi\mu} (0) =b_{\xi\mu}.
$$

For $\xi\in\mc{SP}, \la\in\mc P$, set
\begin{equation}   \label{eq:gb}
g_{\xi\la}=2^{-\ell(\xi)}b_{\xi\la}.
\end{equation}

Up to some $2$-power, $g_{\xi\la}$ has the following
interpretation of branching coefficient for the restriction of a
$\qn$-module $V(\la)$ to $\mf{gl}(n)$.

\begin{lemma}\label{lem:res.Vxi}
As  a $\mathfrak{gl}(n)$-module,  $V(\xi)$ can be decomposed as
$$
V(\xi) \cong\bigoplus_{\la\in\mc{P}, \ell(\la)\leq n}
2^{\frac{\ell(\xi) +\delta(\xi)}2} g_{\xi\la} L(\la).
$$
\end{lemma}

\begin{proof}
It suffices to verify on the character level. The corresponding
character identity indeed follows from   (\ref{eqn:QSchur}),
\eqref{eq:gb}  and Theorem~ \ref{thm:charqueer}, as the character of
$L(\la)$  is given by the Schur function $s_{\la}$.
\end{proof}

\begin{lemma}\cite[Theorem 9.3]{St} \cite[III, (8.17)]{Mac}
    \label{lem:gxila}
The following holds for $\xi\in\mc{SP}, \la\in\mc P$:
\begin{align}
g_{\xi\la}\in \Z_+;\quad
g_{\xi\la}=0\text{ unless }\xi\geq\la;\quad
g_{\xi\xi}=1.
\end{align}
\end{lemma}

Stembridge \cite{St} proved Lemma~\ref{lem:gxila} by providing a
combinatorial formula for $g_{\xi\la}$ in terms of marked shifted
tableaux. We give a representation theoretic proof below.

\begin{proof}
It follows by Lemma~\ref{lem:res.Vxi} that $g_{\xi\la} \ge 0$, and
moreover, $g_{\xi\la} =0$ unless $\xi \ge \la$ (the dominance order
for compositions coincide with the dominance order of weights for
$\qn$). The highest weight space for the $\qn$-module $V(\xi)$
is $W_\xi$, which has dimension $2^{\frac{\ell(\xi)
+\delta(\xi)}2}$. Hence, $g_{\xi\xi} =1$, by Lemma~\ref{lem:res.Vxi}
again.

By  Theorem~ \ref{thm:charqueer}, $2^{-\frac{\ell(\xi)
+\delta(\xi)}2} \ch V(\xi) =2^{ -\ell(\xi)} Q(x_1, \ldots, x_n)$,
which is known to lie in $\La$, cf. \cite{Mac} (this fact can also
be seen directly from representation theory of $\qn$). Hence, $2^{
-\ell(\xi)} Q(x_1, \ldots, x_n)$ is a $\Z$-linear combination of
Schur polynomials $s_\la$. Combining with Lemma~\ref{lem:res.Vxi},
this proves that $g_{\xi\la} \in \Z$.
\end{proof}

The following is a spin counterpart of the properties of Kostka
polynomials listed in Theorem~\ref{thm:Kostka}.

\begin{theorem}  \cite{WW2} \label{thm:spinKostka}
The spin Kostka polynomials $K_{\xi\mu}^-(t)$ for $\xi\in\mc {SP}_n,
\mu\in\mc P_n$ satisfy the following properties:
\begin{enumerate}
\item $K_{\xi\mu}^-(t)=0$ unless $\xi\geq\mu$; $K^-_{\xi\xi}(t)=2^{\ell(\xi)}$.

\item The degree of the polynomial $K_{\xi\mu}^-(t)$ is $n(\mu)-n(\xi)$.

\item
$2^{-\ell(\xi)}K_{\xi\mu}^-(t)$ is a polynomial with non-negative
integer coefficients.

\item
$K_{\xi\mu}^-(1)=K^-_{\xi\mu}$;\quad
$K^-_{\xi\mu}(-1)=2^{\ell(\xi)}\delta_{\xi\mu}$.

\item $K^-_{(n)\mu}(t)=t^{n(\mu)}\prod^{\ell(\mu)}_{i=1}(1+t^{1-i}).$

\item
$K^-_{\xi(1^n)}(t)=\ds\frac{t^{n(\xi)}(1-t)(1-t^{2})\cdots(1-t^{n})\prod_{(i,j)\in
\xi^*}(1+t^{c_{ij}})} {\prod_{(i,j)\in \xi^*}(1-t^{h^*_{ij}})}.$

\end{enumerate}

\end{theorem}
\begin{proof}
Combining Theorem~\ref{thm:Kostka}(1)-(3),  Lemma \ref{lem:gxila}
and Proposition~\ref{prop:relate}, we can easily verify that the
spin Kostka polynomial $K_{\xi\mu}^-(t) $ must satisfy the
properties (1)-(3) in the theorem.
It is known that $P_{\mu}(x;1)=m_{\mu}$ and hence
by~(\ref{eqn:Qmonomial}) we have $K^-_{\xi\mu}(1)=K^-_{\xi\mu}$.
Also, $ Q_{\xi}=2^{\ell(\xi)}P_{\xi}(x;-1)$, and
$\{P_{\mu}(x;-1)\mid\mu\in\mc P\}$ forms a basis for $\La$ (see
\cite[p.253]{Mac}). Hence (4) is proved.

By \cite[III, $\S$3, Example 1(3)]{Mac} we have
\begin{align}\label{eq:sumHallP}
\prod_{i\geq 1}\frac{1+x_i}{1-x_i}=
\sum_{\mu}t^{n(\mu)}\prod^{\ell(\mu)}_{j=1}(1+t^{1-j})P_{\mu}(x;t).
\end{align}
Comparing the degree $n$ terms of
 (\ref{eq:sumHallP})~and~(\ref{gen.fun.qr}), we obtain that
\begin{align*}
Q_{(n)}(x)=q_n(x)=\sum_{\mu\in\mc P_n}t^{n(\mu)}\prod^{\ell(\mu)}_{j=1}(1+t^{1-j})P_{\mu}(x;t).
\end{align*}
Hence (5) is proved.

Part (6) actually follows from Theorem~\ref{th:S(V)} and Theorem~\ref{thm:spingrmult} in
Section~\ref{subsec:spinKostkaRepresent} below,
and let us postpone its proof after completing  the
proof of Theorem~\ref{thm:spingrmult}.
\end{proof}

\subsection{Spin Kostka polynomials and graded multiplicity}\label{subsec:spinKostkaRepresent}
Recall the characteristic map $\ch$ and $\ch^-$
from~(\ref{charmap})~and~(\ref{spincharmap}). Note that   ${\rm
ch}^-$ is related to ${\rm ch}$ as follows:
\begin{equation}\label{eqn:spin.ch}
{\rm ch}^-(\zeta)={\rm ch} \big({\rm res}^{\HC_n}_{\C \mf S_n}\zeta
\big), \quad \text{ for } \zeta\in R^-_n.
\end{equation}
Recall that the $\mf S_n$-module $S^{\la}$ and $\HC_n$-module
$D^{\xi}$ have characters given by $\chi^{\la}$ and $\zeta^{\xi}$,
respectively. Up to some $2$-power as in Lemma~\ref{lem:res.Vxi},
$g_{\xi\la}$ has another representation theoretic interpretation.
\begin{lemma}\label{lem:ind.Specht}
Suppose $\xi\in\mc{SP}_n, \la\in\mc{P}_n$. The following holds:
$$
\dim{\rm Hom}_{\HC_n}(D^{\xi},{\rm ind}^{\HC_n}_{\C\mf S_n}S^{\la} )
= 2^{\frac{\ell(\xi) +\delta(\xi)}2} g_{\xi\la}.
$$

\end{lemma}
\begin{proof}
Since the $\HC_n$-module ${\rm ind}^{\HC_n}_{\C \mf S_n}S^{\la}$ is
semisimple, we have
\begin{align*}
\dim\Hom_{\HC_n}(D^{\xi},{\rm ind}^{\HC_n}_{\C\mf S_n}S^{\la} )
=&\dim\Hom_{\HC_n}({\rm ind}^{\HC_n}_{\C \mf S_n}S^{\la}, D^{\xi})\\
=&\dim\Hom_{\C\mf S_n}(S^{\la}, {\rm res}^{\HC_n}_{\C\mf S_n}D^{\xi})\\
=&(s_{\la},{\rm ch}({\rm res}^{\HC_n}_{\C\mf S_n}D^{\xi}))\\
=&(s_{\la},{\rm ch}^-(D^{\xi}))\\
=&( s_{\la},2^{-\frac{\ell(\xi)-\delta(\xi)}{2}}Q_{\xi}(x))\\
=&2^{\frac{\ell(\xi) +\delta(\xi)}2} g_{\xi\la},
\end{align*}
where the second equation uses the Frobenius reciprocity, the third
equation uses  the fact that ${\rm ch}$ is an isometry,  the fourth,
fifth and sixth  equations follow from~(\ref{eqn:spin.ch}),
\eqref{eq:irredch} and  \eqref{eqn:QSchur}, respectively.
\end{proof}

For $\mu\in\mc{P}_n$ and $\xi\in \mc{SP}_n$, recalling
\eqref{eqn:mult.C}, we define a polynomial $C^-_{\xi\mu}(t)$ as a
graded multiplicity of the graded $\HC_n$-module ${\rm
ind}^{\HC_n}_{\C\mf  S_n} H^{\bullet}(\mc{B}_{\mu}) \cong \Cl_n \otimes
H^{\bullet}(\mc{B}_{\mu})$:
\begin{align}\label{eqn:mult.C-}
C^-_{\xi\mu}(t) :=\sum_{i\geq 0}t^i \Big(\dim\Hom_{\HC_n}(D^{\xi},
\Cl_n\otimes H^{2i}(\mc{B}_{\mu}))\Big).
\end{align}

\begin{theorem} \cite{WW2}   \label{thm:spingrmult}
Suppose $\xi\in\mc{SP}_n, \mu\in\mc P_n$. Then we have
$$
 K^-_{\xi\mu}(t) =
2^{\frac{\ell(\xi)-\delta(\xi)}{2}}  C^-_{\xi\mu}(t^{-1})t^{n(\mu)}.
$$

\end{theorem}

\begin{proof}
By Proposition~\ref{prop:relate} and~ Theorem~\ref{thm:Kostka
gr.mult}, we obtain that
\begin{align*}
K^-_{\xi\mu}(t)
&=\sum_{\la\in\mc P_n}b_{\xi\la}K_{\la\mu}(t)
=\sum_{\la\in\mc P_n}b_{\xi\la}C_{\la\mu}(t^{-1})t^{n(\mu)}.
\end{align*}
On the other hand,  we have by Lemma \ref{lem:ind.Specht} that
\begin{align*}
C^-_{\xi\mu}(t) &=\sum_{i\geq 0}t^i \left(\dim\Hom_{\HC_n}
\big(D^{\xi}, {\rm ind}^{\HC_n}_{\C\mf S_n}
H^{2i}(\mc{B}_{\mu})\big) \right)
 \\
 &=\sum_{\la} C_{\la\mu}(t) \dim\Hom_{\HC_n}(D^{\xi},
{\rm ind}^{\HC_n}_{\C\mf S_n} S^\la)
 \\
&=2^{-\frac{\ell(\xi)-\delta(\xi)}{2}} \sum_{\la\in\mc  P_n}   b_{\xi\la}C_{\la\mu}(t).
\end{align*}
Now the theorem follows by comparing the above two identities.
\end{proof}

With Theorem~\ref{thm:spingrmult} at hand, we can complete the proof
of Theorem~\ref{thm:spinKostka}(6).
\begin{proof}[Proof of Theorem \ref{thm:spinKostka}(6)]
Suppose $\xi\in\mc{SP}_n$. Observe that $\mc B_{(1^n)}=\mc B$ and it
is well known that $ H^{\bullet}(\mc{B})$ is isomorphic to the
coinvariant algebra of the symmetric group $\mf S_n$. Hence by
Theorem~\ref{th:S(V)} we obtain that

$$
C^-_{\xi(1^n)}(t)
=d^{\xi}(t)
=2^{-\frac{\ell(\xi)-\delta(\xi)}{2}}
\frac{t^{n(\xi)}(1-t)(1-t^{2})\cdots(1-t^{n})\prod_{(i,j)\in
\xi^*}(1+t^{c_{ij}})} {\prod_{(i,j)\in \xi^*}(1-t^{h^*_{ij}})},
$$ where $\xi^*$ is the shifted diagram associated to $\xi$,
$c_{ij}, h^*_{ij}$ are contents and shifted hook lengths for a cell
$(i,j) \in \xi$. This together with Theorem~\ref{thm:spingrmult}
gives us
\begin{align*}
K_{\xi(1^n)}^-(t)
&=\frac{t^{\frac{n(n-1)}{2}-n(\xi)}(1-t^{-1})(1-t^{-2})\cdots(1-t^{-n})\prod_{(i,j)\in
\xi^*}(1+t^{-c_{ij}})} {\prod_{(i,j)\in \xi^*}(1-t^{-h^*_{ij}})}\\
&=\frac{t^{-n-n(\xi)+\sum_{(i,j)\in \xi^*}h^*_{ij}}(1-t)(1-t^2)\cdots(1-t^n)\prod_{(i,j)\in
\xi^*}(1+t^{c_{ij}})} {t^{\sum_{(i,j)\in \xi^*}c_{ij}}\prod_{(i,j)\in \xi^*}(1-t^{h^*_{ij}})}\\
&=\frac{t^{n(\xi)}(1-t)(1-t^2)\cdots(1-t^n)\prod_{(i,j)\in
\xi^*}(1+t^{c_{ij}})} {\prod_{(i,j)\in \xi^*}(1-t^{h^*_{ij}})},
\end{align*}
where the last equality can be derived by noting that the contents
$c_{ij}$ are $0,1,\ldots,\xi_i-1$ and the fact (cf. \cite[III, \S
8, Example 12]{Mac}) that in the $i$th row of $\xi^*$, the hook
lengths $h^*_{ij}$ for $i\leq j\leq\xi_i+i-1$ are
$1,2,\ldots,\xi_i,\xi_i+\xi_{i+1},\xi_i+\xi_{i+2},\ldots,\xi_{i}+\xi_{\ell}$
with exception
$\xi_i-\xi_{i+1},\xi_i-\xi_{i+2},\ldots,\xi_{i}-\xi_{\ell}$.
\end{proof}

\subsection{Spin Kostka polynomials and $q$-weight multiplicity}
Observe that there is a natural isomorphism $\ev{\qn} \cong
\mf{gl}(n)$. Regarding a regular nilpotent element $e$ in
$\mf{gl}(n)$ as an even element in $\qn$,  for $\xi\in\mc{SP},
\mu\in\mc{P}$ with $\ell(\xi)\leq n, \ell(\mu)\leq n$, we define a
{\em Brylinski-Kostant filtration} $\{J_{e}^k
\big(V(\xi)_{\mu}\big)\}_{k\ge 0}$ on the $\mu$-weight space
$V(\xi)_{\mu}$ of the irreducible $\qn$-module $V(\xi)$, where
\begin{align*}
J_{e}^k(V(\xi)_{\mu}) :=\{v\in V(\xi)_{\mu} \mid e^{k+1}v=0\}.
\end{align*}
Define a polynomial  $\gamma^-_{\xi\mu}(t)$ by letting
\begin{align*}
\gamma^-_{\xi\mu}(t)=\sum_{k\geq 0}\Big(\dim
J_{e}^k(V(\xi)_{\mu})/J_{e}^{k-1}(V(\xi)_{\mu})\Big)t^k.
\end{align*}

We are ready to establish the Lie theoretic interpretation of spin
Kostka polynomials.

\begin{theorem} \cite{WW2}
Suppose $\xi\in\mc{SP}, \mu\in\mc{P}$ with $\ell(\xi)\leq n, \ell(\mu)\leq n$.
Then we have
$$
 K^-_{\xi\mu}(t) =
2^{\frac{\ell(\xi)-\delta(\xi)}{2}} \gamma^-_{\xi\mu}(t).
$$
\end{theorem}
\begin{proof}
The Brylinski-Kostant filtration is defined via a regular nilpotent
element in $\mf{gl}(n) \cong \qn_{\bar{0}}$, and thus it is
compatible with the decomposition in Lemma~ \ref{lem:res.Vxi}.
Hence, we have $ J_{e}^k \big(V(\xi)_{\mu}\big) \cong
\oplus_{\la}2^{\frac{\ell(\xi) +\delta(\xi)}2} g_{\xi\la}J^k_{e}
\big(L(\la)_{\mu} \big). $ It follows by the definitions of the
polynomials $\gamma^-_{\xi\mu}(t)$ and $\gamma_{\la\mu}(t)$ that
$$
\gamma^-_{\xi\mu}(t)
=\sum_{\la}2^{\frac{\ell(\xi) +\delta(\xi)}2}g_{\xi\la}\gamma_{\la\mu}(t).
$$
Then by Theorem \ref{thm:jump} we obtain that
$$
\gamma^-_{\xi\mu}(t)
=\sum_{\la}2^{\frac{\ell(\xi) +\delta(\xi)}2}g_{\xi\la}K_{\la\mu}(t)
=\sum_{\la}2^{-\frac{\ell(\xi) -\delta(\xi)}2}b_{\xi\la}K_{\la\mu}(t).
$$
This together with Proposition~\ref{prop:relate} proves
the theorem.
\end{proof}

\begin{remark}
We can define spin Hall-Littlewood functions $H^-_{\mu}(x;t)$ via
the spin Kostka polynomials as well as spin Macdonald polynomials
$H^-_\mu (x;q,t)$ and the spin $q,t$-Kostka polynomials
$K^-_{\xi\mu} (q,t)$.  The use of $\Phi$ and $\varphi$ makes such a
$q,t$-generalization possible (see \cite{WW2} for details). There is
also a completely different vertex operator approach developed by
Tudose and Zabrocki \cite{TZ} toward a different version of  spin Kostka polynomials and
spin Hall-Littlewood functions, which did not seem to admit
representation theoretic interpretation.
\end{remark}

\section{The seminormal form construction}\label{sec:seminormal}

In this section, we formulate the seminormal form for the
irreducible $\HC_n$-modules, analogous to Young's seminormal form
for the irreducible $\C\mf S_n$-modules. Following the independent
works of \cite{HKS} and \cite{Wan} (which was built on the earlier
work of Nazarov \cite{Naz}), we first work on the generality of
affine Hecke-Clifford algebras, and then specialize to the (finite)
Hecke-Clifford algebras to give an explicit construction of Young's
seminormal form for the irreducible $\HC_n$-modules.

\subsection{Jucys-Murphy elements and Young's seminormal form for $\mf S_n$}

The Jucys-Murphy elements in the group algebra of the symmetric
group $\mf S_n$ are defined by
\begin{equation}\label{eq:JM}
L_k=\sum_{1\leq j<k}(j,k),
\end{equation}
where $(j,k)$ is the transposition between $j$ and $k$. Observe that
$L_k$ is the difference between the sum of all transpositions in
$\mf S_k$ and the sum of all transpositions in $\mf S_{k-1}$. Hence
the Jucys-Murphy elements $L_1,\ldots, L_n$ commute and act
semisimply on irreducible $\C\mf S_n$-modules.

The Gelfand-Zetlin subalgebra  $\mc A_n$ of $\C\mf S_n$ is defined
to be the subalgebra consisting of the diagonal matrices in the
Wedderburn decomposition of $\C\mf S_n$. It is not difficult to show
by induction on $n$ (see \cite[Corollary 4.1]{OV} and
\cite[Lemma~2.1.4]{Kle}) that $\mc A_n$ is  generated by the centers
of the subalgebras $\C\mf S_1,\C\mf S_2,\ldots,\C\mf S_n$, and that
it is also generated by the Jucys-Murphy elements $L_1,\ldots, L_n$.

The moral is that the subalgebra $\mc A_n$ of $\C\mf S_n$ plays a role of a Cartan
subalgebra of a semsimple Lie algebra. Every irreducible $\C\mf S_n$-module $V$
can be decomposed as
$$
V=\bigoplus_{\uli=(i_1,\ldots,i_n)\in\C^n}V_{\uli},
$$
where $V_{\uli}=\{v\in V\mid L_kv=i_kv, 1\leq k\leq n\}$ is the
simultaneous eigenspace of $L_1,\ldots, L_n$ with eigenvalues
$i_1,\ldots,i_n$.  By the description of $\mc A_n$ above, we have
either $V_{\uli}=0$ or $\dim V_{\uli}=1$. If $V_{\uli}\neq 0$, we say
that $\uli$ is a weight of $V$ and $V_{\uli}$  is the $\uli$-weight
space of $V$, and we fix a nonzero vector $v_{\uli}\in V_{\uli}$.

Suppose $\la\in\mc P_n$ and $T$ is a standard tableau of shape
$\la$. Define its content sequence
$c(T)=(c(T_1),\ldots,c(T_n))\in\C^n$ by letting $c(T_k)$ be the
content of the cell occupied by $k$ in $T$ for $1\leq k\leq n$. By
analyzing the structures of weights, we can show that the sequences
$c(T)$  for standard tableaux $T$ with $n$ cells are exactly all the
weights  for irreducible $\mf S_n$-modules. Now we are ready to
formulate the Young's seminormal form for irreducible $\C\mf
S_n$-modules. For $\la\in\mc P_n$, define $V^{\la}=\sum_T\C v_T$,
where the summation is taken over standard tableaux of shape $\la$.
For $1\leq k\leq n-1$,  define
\begin{equation}\label{seminormalsymm}
s_kv_T=\big(c(T_{k+1})-c(T_k)\big)^{-1}v_T+\sqrt{1-
\big(c(T_{k+1})-c(T_k)\big)^{-2}} v_{s_kT},
\end{equation}
where $s_kT$ indicates the standard tableau obtained by switching
$k$ and $k+1$ in $T$ and $v_{s_kT}=0$ if $s_kT$ is not standard. In
this way Okounkov and Vershik \cite{OV} established the following.

\begin{theorem}[Young's seminormal form]\label{thm:seminormalsymm}
For $\la\in \mc{P}_n$, $V^{\la}$ affords an irreducible $\mf
S_n$-module given by \eqref{seminormalsymm}. Moreover, $\{V^{\la}
\mid  \la\in\mc P_n\}$ forms a complete set of non-isomorphic
irreducible $\mf S_n$-modules.
\end{theorem}

\subsection{Jucys-Murphy elements for $\HC_n$ }
As in the group algebra of symmetric groups, there also exist
Jucys-Murphy elements $J_k(1\leq k\leq n)$ in
$\HC_n$ defined as (see \cite{Naz})
\begin{align}
J_k=\sum_{1\leq j< k}(1+c_jc_k)(j,k).\label{JM}
\end{align}

\begin{lemma} The following holds:
\begin{enumerate}
 \item $J_iJ_k=J_kJ_i$, \quad for $1\leq i\neq k\leq n$.
 \item $c_iJ_i=-J_ic_i, \quad c_iJ_k=J_kc_i$,\quad for $1\leq i\neq k\leq n$.
 \item $s_iJ_i=J_{i+1}s_i-(1+c_ic_{i+1})$,\quad for $1\leq i\leq n-1$.
 \item $s_iJ_k=J_ks_i$,\quad for $k\neq i, i+1$.
 \end{enumerate}
\end{lemma}
\begin{proof}
It follows by a direct computation that $c_iJ_n=J_nc_i$,
and $\sigma J_n=J_n\sigma$, for $1\le i \le n-1$ and $\sigma \in \mf S_{n-1}$.
Hence, $J_n$
commutes with $\HC_{n-1}$, and whence (1).
The remaining properties can be also verified by direct calculations.
\end{proof}

\subsection{Degenerate affine Hecke-Clifford algebras $\Haff$}\label{sec:Haff}
For $n\in\mathbb{Z}_+$, the affine Hecke-Clifford algebra $\Haff$ is
defined to be the superalgebra generated by even generators
$s_1,\ldots,s_{n-1},x_1,\ldots,x_n$ and odd generators
$c_1,\ldots,c_n$ subject to the following relations (besides the
relations \eqref{braid}, \eqref{clifford} and \eqref{pc}):
\begin{align}
x_ix_j&=x_jx_i, \quad 1\leq i,j\leq n, \label{poly}\\
s_ix_i&=x_{i+1}s_i-(1+c_ic_{i+1}), \quad 1\leq i\leq n-1,\label{px1}\\
s_ix_j&=x_js_i, \quad j\neq i, i+1, \quad 1\leq i,j\leq n, \label{px2}\\
x_ic_i=-c_ix_i,\; x_ic_j&=c_jx_i,\quad 1\leq i\neq j\leq n.
\label{xc}
\end{align}
\begin{remark}
The affine Hecke-Clifford algebra  $\Haff$ was introduced by
Nazarov~\cite{Naz} (sometimes called affine Sergeev algebra).  The Morita
super-equivalence \eqref{map:isorm.HC} between $\HC_n$ and  $\C
{\mf S}_n^-$ has been extended to one between $\Haff$ and the affine spin
Hecke algebras \cite[Proposition~3.4]{Wa1} and for other classical
type Weyl groups \cite[Theorem~4.4]{KW}.
\end{remark}

Denote by $\mc{P}_n^{\mathfrak{c}}$ the superalgebra generated by
even generators $x_1,\ldots,x_n$ and odd generators $c_1,\ldots,c_n$
subject to the relations (\ref{clifford}), (\ref{poly}) and
(\ref{xc}). For $\alpha=(\alpha_1,\ldots,\alpha_n)\in\mathbb{Z}_+^n$
and $\beta\in\mathbb{Z}_2^n$, set $x^{\alpha}=x_1^{\alpha_1}\cdots
x_n^{\alpha}$ and $c^{\beta}=c_1^{\beta_1}\cdots c_n^{\beta_n}$.
Then we have the following.

\begin{lemma}\cite{Naz}  \cite[Theorem 2.2]{BK}\label{lem:PBW}
The set $\{x^{\alpha}c^{\beta}w\mid \alpha\in\mathbb{Z}_+^n,
\beta\in\mathbb{Z}_2^n, w\in \mf S_n\}$ forms a basis of $\Haff$.
\end{lemma}

\begin{proof}[Sketch of a proof]
One can construct a representation $\pi$ of $\Haff$ on the
polynomial-Clifford algebra $\mc{P}_n^{\mathfrak{c}}$, where the
$x_i$ and $c_i$ for all $i$ act by left multiplication (and the
action of $s_i$'s is then determined uniquely). Then one checks that
the linear operators $\pi(x^{\alpha}c^{\beta}w)$ are linearly
independent. We refer to the proof of \cite[Theorem~3.4]{KW} for
detail.
\end{proof}

By \cite{Naz}, there exists a surjective homomorphism
\begin{align}
\digamma: \Haff&\longrightarrow \HC_n\label{surjmap}\\
c_k\mapsto c_k, s_l&\mapsto s_l, x_k\mapsto J_k, \quad (1\leq
k\leq n, 1\leq l\leq n-1),  \notag
\end{align}
and the kernel of $\digamma$ coincides with the ideal $\langle
x_1\rangle$ of $\Haff$ generated by $x_1$. Hence the category of
finite-dimensional $\HC_n$-modules can be identified as the category
of finite-dimensional $\Haff$-modules which are annihilated by
$x_1$. For the study of $\Haff$-modules, we shall mainly focus on
the so-called finite-dimensional {\it integral} modules, on which
$x^2_1,\ldots, x^2_n$ have eigenvalues of the form
$$
q(i)=i(i+1), \qquad i\in\Z_+.
$$
It is easy to see that a finite-dimensional $\Haff$-module $M$ is
integral if all of eigenvalues of $x^2_j$ {\em for a fixed $j$} on $M$ are
of the form $q(i)$  (cf. \cite[Lemma~4.4]{BK} or \cite[Lemma~
15.1.2]{Kle}). Hence the category of finite-dimensional
$\HC_n$-modules can be identified with the subcategory of integral
$\Haff$-modules on which $x_1=0$.

By Lemma~\ref{lem:PBW}, $\mc{P}_n^{\mathfrak{c}}$ can be identified
with the subalgebra of $\Haff$ generated by $x_1,\ldots,x_n$ and
$c_1,\ldots,c_n$. For $i\in\Z_+$, denote by $L(i)$ the
$2$-dimensional $\mathcal{P}_1^{\mathfrak{c}}$-module with
$L(i)_{\bar{0}}=\C v_0$ and $L(i)_{\bar{1}}=\C v_1$ and
$$
x_1v_0=\sqrt{q(i)}v_0,\quad x_1v_1=-\sqrt{q(i)}v_1, \quad
c_1v_0=v_1,\quad c_1v_1=v_0.
$$
Note that $L(i)$ is irreducible of type $\texttt{M}$ if $i\neq 0$,
and irreducible of type $\texttt{Q}$ if $i=0$. Moreover $L(i),
i\in\Z_+$ form a complete set of pairwise non-isomorphic integral
irreducible $\mathcal{P}_1^{\mathfrak{c}}$-module. Since
$\mc{P}_n^{\mathfrak{c}}\cong
\mathcal{P}_1^{\mathfrak{c}}\otimes\cdots\otimes
\mathcal{P}_1^{\mathfrak{c}}$,   Lemma~\ref{tensorsmod} implies the
following.

\begin{lemma} 
$\{L(\uli)=L(i_1)\circledast L(i_2)\circledast\cdots\circledast
L(i_n)|~\uli=(i_1,\ldots,i_n)\in\Z_+^n\}$ form a complete set of
pairwise non-isomorphic integral irreducible
$\mc{P}_n^{\mathfrak{c}}$-modules.
Furthermore,
$\text{dim}~L(\uli)=2^{n-\lfloor\frac{\gamma_0}{2}\rfloor}$, where
$\gamma_0$ denotes the number of $1\leq j\leq n$ with $i_j=0$, and
$\lfloor\frac{\gamma_0}{2}\rfloor$ denotes the greatest integer less
than or equal to $\frac{\gamma_0}{2}$ .
\end{lemma}

The following definition of  \cite{HKS, Wan} is motivated by similar
studies for the affine Hecke algebras in \cite{Ch, Ram, Ru}.
\begin{definition}\label{defn:CS}
A representation of $\Haff$ is called {\em completely splittable} if
$x_1,\ldots,x_n$ act semisimply.
\end{definition}

Since the polynomial generators $x_1,\ldots,x_n$ commute, a
finite-dimensional integral completely splittable $\Haff$-module $M$
can be decomposed as
\begin{align*}
M=\bigoplus_{\uli\in\Z_+^n}M_{\uli},
\end{align*}
where
$$
M_{\uli}=\{z\in M \mid x_k^2z=q(i_k)z, 1\leq k\leq n\}.
$$
If $M_{\underline{i}}\neq 0$, then $\underline{i}$ is called a {\em
weight} of $M$ and $M_{\underline{i}}$ is called a {\em weight
space}. Since $x_k^2, 1\leq k\leq n$ commute with $c_1,\ldots,c_n$,
each $M_{\uli}$ is actually $\mc{P}_n^{\mathfrak{c}}$-submodule of
$M$.

Following Nazarov, we define the intertwining elements as
\begin{align}
\phi_k:=s_k(x_k^2-x^2_{k+1})+(x_k+x_{k+1})+c_kc_{k+1}(x_k-x_{k+1}),\label{eq:intertw}
1\leq k<n.
\end{align}
It is known \cite{Naz} and easy to check directly that
\begin{align}
\phi_k^2=2(x_k^2+x^2_{k+1})-(x_k^2-x^2_{k+1})^2\label{eq:sqinter},\\
\phi_kx_k=x_{k+1}\phi_k, \phi_kx_{k+1}=x_k\phi_k,
\phi_kx_l=x_l\phi_k\label{eq:xinter},\\
\phi_kc_k=c_{k+1}\phi_k, \phi_kc_{k+1}=c_k\phi_k,
\phi_kc_l=c_l\phi_k\label{eq:cinter},\\
\phi_j\phi_k=\phi_k\phi_j,
\phi_k\phi_{k+1}\phi_k=\phi_{k+1}\phi_k\phi_{k+1},\label{braidinter}
\end{align}
for all admissible $j,k,l$ with $l\neq k, k+1$ and $|j-k|>1$.

\subsection{Weights and standard skew shifted tableaux}

This subsection is technical though elementary in nature, and we
recommend the reader to skip most of the proofs in a first reading.
The upshot of this subsection is Proposition~\ref{prop:bijection}
which identifies the weights as content vectors associated to
standard skew shifted tableaux.

\begin{lemma}\label{lem:separate weig.}
Suppose that $M$ is an integral completely splittable $\Haff$-module
and that $\uli=(i_1,\ldots,i_n)\in\Z_+^n$ is a weight of $M$. Then
$i_k\neq i_{k+1}$ for all $1\leq k\leq n-1$.
\end{lemma}
\begin{proof}
Suppose $i_k=i_{k+1}$ for some $1\leq k\leq n-1$. Let $0\neq z\in
M_{\underline{i}}$. One can show using~(\ref{px1}) that
\begin{align}
x_k^2s_k&=s_kx_{k+1}^2-\big(x_k(1-c_kc_{k+1})+(1-c_kc_{k+1})x_{k+1}\big)
 \label{reln:x2s1}\\
x_{k+1}^2s_k&=s_kx_k^2+\big(x_{k+1}(1+c_kc_{k+1})+(1+c_kc_{k+1})x_k\big).
\label{reln:x2s2}
\end{align}
Since $M$ is completely
splittable, $(x_k^2-q(i_k))z=0=(x^2_{k+1}-q(i_{k+1}))z$. This
together with~(\ref{reln:x2s1}) shows that
\begin{equation}  \label{eq:aux1}
(x_k^2-q(i_k))s_kz=(x_k^2-q(i_{k+1}))s_kz=-\big(x_k(1-c_kc_{k+1})+(1-c_kc_{k+1})x_{k+1}\big)z,
\end{equation}
and hence
$$
(x_k^2-q(i_k))^2s_kz=
-\big(x_k(1-c_kc_{k+1})+(1-c_kc_{k+1})x_{k+1}\big)(x_k^2-q(i_k))z=0.
$$
Similarly, we see that
$$
(x_{k+1}^2-q(i_{k+1}))^2s_kz=0.
$$
Hence $s_kz\in M_{\underline{i}}$, i.e., $(x_k^2-q(i_k))s_kz=0$, and
therefore \eqref{eq:aux1} implies that
\begin{align*}
&\big(x_k(1-c_kc_{k+1})+(1-c_kc_{k+1})x_{k+1}\big)z=0,
 \\
&2(x_k^2+x^2_{k+1})z=\big(x_k(1-c_kc_{k+1})+(1-c_kc_{k+1})x_{k+1}\big)^2z=0.
\end{align*}
This means that
 $q(i_{k+1})=-q(i_k)$ and hence $q(i_k)=q(i_{k+1})=0$ since
$i_k=i_{k+1}$.
We conclude that $x_k=0=x_{k+1}$ on $ M_{\underline{i}}$. This
implies that $x_{k+1}s_kz=0$ since $s_kz\in M_{\underline{i}}$ as
shown above. Then
$$(1+c_kc_{k+1})z=x_{k+1}s_kz-s_kx_kz=0,
$$
and hence $z=\hf(1-c_kc_{k+1})(1+c_kc_{k+1})z=0$, which is a
contradiction.
\end{proof}
\begin{lemma}\label{lem:interact}
Assume that $\uli=(i_1,\ldots,i_n)\in\Z_+^n$ is a weight of an
irreducible integral completely splittable $\Haff$-module $M$. Fix
$1\leq k\leq n-1$.
\begin{enumerate}
\item If $i_k\neq i_{k+1}\pm1$, then
$\phi_kz$ is a nonzero weight vector of weight $s_k\cdot \uli$ for
any $0\neq z\in M_{\uli}$. Hence $s_k\cdot\uli$ is a weight of $M$.

\item If $i_k=i_{k+1}\pm 1$, then $\phi_k=0$ on $M_{\uli}$.
\end{enumerate}
\end{lemma}

\begin{proof}
It follows from~\eqref{eq:xinter}~that $\phi_k M_{\uli}\subseteq
M_{s_k\cdot\uli}$. By \eqref{eq:sqinter}, we have
$$
\phi_k^2z
=\big(2(x_k^2+x^2_{k+1})-(x_k^2-x^2_{k+1})^2\big)z
=\big(2(q(i_k)+q(i_{k+1}))-(q(i_k)-q(i_{k+1}))^2\big)z
$$
for any $z\in M_{\uli}$. A calculation shows that
$2(q(i_k)+q(i_{k+1}))-(q(i_k)-q(i_{k+1}))^2 \neq 0$ when $i_k\neq
i_{k+1}\pm1$ and hence $\phi_k^2z\neq 0$. This proves (1).

Assume now that $i_k=i_{k+1}\pm 1$. Suppose $\phi_kz\neq0$ for some
$z\in M_{\uli}$. Since $M$ is irreducible, there exists a sequence
$1\leq a_1, a_2,\ldots, a_{m}\leq n-1$ such that
\begin{equation}\label{eq:intertact1}
\phi_{a_m}\cdots\phi_{a_2}\phi_{a_1}\phi_kz=\alpha z
\end{equation}
for some $0\neq\alpha\in\C$. Assume that $m$ is minimal such that
\eqref{eq:intertact1} holds. Let $\sigma=s_{a_m}\cdots
s_{a_1}s_k\in\mf S_n$. Then $\sigma\cdot\uli=\uli$. If $\sigma\neq
1$, then there exists $1\le b_1\le b_2\le n$ such that
$i_{b_1}=i_{b_2}$, and $\sigma=(i_1,i_2)$ by the minimality of $m$.
Hence, $i_{b_1}$ and $i_{b_2}$ can be brought to be adjacent by the
permutation $s_{a_j}\cdots s_{a_1}s_k\cdot\uli$ for some $1\leq
j\leq m$. That is, $s_{a_j}\cdots s_{a_1}s_k\cdot\uli$ is a weight
of $M$ of the form $(\cdots,\beta,\beta,\cdots)$,
%
which contradicts Lemma \ref{lem:separate weig.}. Hence $\sigma=1$
and $s_{a_m}\cdots s_{a_2}s_{a_1}=s_k$. We further claim that $m=1$.
Suppose that $m>1$. Then, by the exchange condition for Coxeter
groups, there exists $1\le p<q\le m$ such that $s_{a_m}\cdots
s_{a_q} \cdots s_{a_p}\cdots s_{a_1}=s_{a_m}\cdots \check{s}_{a_q}
\cdots \check{s}_{a_p}\cdots s_{a_1}$, where $\check{s}$ means the
very term is removed. This leads to an identity similar to
\eqref{eq:intertact1} for a product of ($m-1)$ $\phi$'s,
contradicting the minimality of $m$.
%
Therefore $m=1$ and then $a_1=k$, which together with
\eqref{eq:intertact1} leads to $\phi_k^2z=\alpha z\neq 0.$ This is
impossible by a simple computation:
$$
\phi_k^2=2(x^2_k+x^2_{k+1})-(x^2_k-x^2_{k+1})^2
=2(q(i_k)+q(i_{k+1}))-(q(i_k)-q(i_{k+1}))^2=0
$$
on $M_{\uli}$ since $i_k=i_{k+1}\pm1$. This proves (2).
\end{proof}

\begin{corollary}\label{cor:wtbraid}
Assume that $\uli=(i_1,\ldots,i_n)\in\Z_+^n$ is a weight of an irreducible
integral completely splittable $\Haff$-module $M$. If
$i_k=i_{k+2}$ for some $1\leq k\leq n-2$, then $i_k=i_{k+2}=0$ and
$i_{k+1}=1$.
\end{corollary}
\begin{proof}
 If $i_k\neq i_{k+1}\pm1 $, then $s_k\cdot\underline{i}$
is a weight of $M$ of the form $(\cdots, u,u,\cdots)$  by
Lemma~\ref{lem:interact}(1), which contradicts
Lemma~\ref{lem:separate weig.}. Hence $i_k=i_{k+1}\pm 1$. By
Lemma~\ref{lem:interact}(2), we have
\begin{align*}
(a-b)s_kz&=
-\big((x_k+x_{k+1})+c_kc_{k+1}(x_k-x_{k+1})\big)z,\\
(a-b)s_{k+1}z&=
-\big((x_{k+1}+x_{k+2})+c_{k+1}c_{k+2}(x_{k+1}-x_{k+2})\big)z,
\end{align*}
for $z\in M_{\underline{i}}$, where $a=q(i_k)=q(i_{k+2})$, $ b=q(i_{k+1})$.
A direct calculation shows that
\begin{align}
&{(a-b)(b-a)(a-b)} (s_k s_{k+1} s_k-s_{k+1}s_ks_{k+1})z\notag\\
&=\big((x_k+x_{k+2})(6x_{k+1}^2+2x_kx_{k+2}) +
c_kc_{k+2}(x_k-x_{k+2})(6x^2_{k+1}-2x_kx_{k+2})\big) z\notag\\
&=0. \label{eq:braidaction}
\end{align}
for $z\in M_{\underline{i}}$ since $s_k s_{k+1} s_k=s_{k+1}s_ks_{k+1}$.
Decompose
$M_{\underline{i}}$ as $M_{\underline{i}}=N_1\oplus N_2$, where
$N_1=\{z\in M_{\underline{i}}~|~x_kz=x_{k+2}z=\pm\sqrt{a}z\}$ and
$N_2=\{z\in M_{\underline{i}}~|~x_kz=-x_{k+2}z=\pm\sqrt{a}z\}$. Now
applying the equality \eqref{eq:braidaction} to $z$ in $N_1$ and $N_2$, we obtain that
$$
2\sqrt{q(i_k)}\big(6q(i_{k+1})+2q(i_k)\big)=0,\label{eq.ik}
$$
which, thanks to $i_{k+1}=i_{k}\pm 1$, is equivalent to one of the
following two identities:
\begin{align}
\text{if } i_{k+1}=i_k-1, \text{ then } &\sqrt{i_k(i_k+1)}(4i_k-2)i_k=0;  \label{eq.ik1}\\
\text{if } i_{k+1}=i_k+1, \text{ then }
&\sqrt{i_k(i_k+1)}(4i_k+6)(i_k+1)=0.\label{eq.ik2}
\end{align}
There is no solution for (\ref{eq.ik1}), and the solution of
(\ref{eq.ik2}) is $i_k=0, i_{k+1}=1$.
\end{proof}

Denote by $\mc{W}(n)$ the set of weights of all integral
irreducible completely splittable $\Haff$-modules.
\begin{proposition}\label{prop:wtprop}
Assume $\uli\in \mc{W}(n)$ and $i_k=i_{\ell}=a$ for some $1\leq k<\ell\leq n$.
\begin{enumerate}
\item If $a=0$, then $1\in\{i_{k+1},\ldots, i_{\ell-1}\}$.
\item If $a\geq1$, then $\{a-1, a+1\}\subseteq\{i_{k+1},\ldots, i_{\ell-1}\}$.
\end{enumerate}
\end{proposition}

\begin{proof}
Without loss of generality, we can assume that
$a\notin\{i_{k+1},\ldots,i_{\ell-1}\}$.

If $a=0$ but $1\not\in\{i_{k+1},\ldots,i_{\ell-1}\}$, we can
repeatedly swap $i_{\ell}$ with $i_{\ell-1}$ then with $i_{\ell-2}$,
etc., all the way to obtain a weight of $M$ of the form
$(\cdots,0,0,\cdots)$ by Lemma~\ref{lem:interact}. This contradicts
Lemma~\ref{lem:separate weig.}. This proves (1).

Now assume $a\geq 1$ and $a+1\notin\{i_{k+1},\ldots,i_{\ell-1}\}$.
If $a-1$ does not appear between $i_{k+1}$ and $i_{\ell-1}$ in
$\underline{i}$, then we can swap $i_{\ell}$ with $i_{\ell-1}$ then
with $i_{\ell-2}$, etc., and by Lemma~\ref{lem:interact} this gives
rise to a weight of $M$ having the form $(\cdots,a,a,\cdots)$, which
contradicts Lemma~\ref{lem:separate weig.}. If $a-1$ appears only
once between $i_{k+1}$ and $i_{\ell-1}$ in $\underline{i}$, then
again by swapping $i_{\ell}$ with $i_{\ell-1}$  then with
$i_{\ell-2}$, etc. we obtain a weight of $M$ of the form
$(\cdots,a,a-1,a,\cdots)$, which contradicts
Corollary~\ref{cor:wtbraid}. Hence $a-1$ appears at least twice
between $i_{k+1}$ and $i_{l-1}$ in $\underline{i}$. This implies
that there exist $k<k_1<\ell_1<\ell$ such that
$$
i_{k_1}=i_{\ell_1}=a-1,
\{a,a-1\}\cap\{i_{k_1+1},\ldots,i_{\ell_1-1}\}=\emptyset.
$$
An identical argument shows that there exist $k_1<k_2<\ell_2<\ell_1$
such that
$$
i_{k_2}=i_{\ell_2}=a-2,
\{a,a-1,a-2\}\cap\{i_{k_2+1},\ldots,i_{\ell_2-1}\}=\emptyset.
$$
Continuing in this way, we obtain $k<s<t<l$ such that
$$
i_{s}=i_{t}=0,
\{a,a-1,\ldots,1,0\}\cap\{i_{s+1},\ldots,i_{t-1}\}=\emptyset,
$$
which contradicts (1).

Now assume that $a\geq 1$ and
$a-1\notin\{i_{k+1},\ldots,i_{\ell-1}\}$. Then $a+1$ must appear in
the subsequence $(i_{k+1},\ldots,i_{\ell-1})$ at least twice,
otherwise we can repeatedly swap $i_{\ell}$ with $i_{\ell-1}$ then
with $i_{\ell-2}$, etc., all the way to obtain a weight of $M$ of
the form $(\cdots,a,a+1,a\cdots)$ by Lemma~\ref{lem:interact}, which
contradicts Corollary~\ref{cor:wtbraid}.
%
Continuing this way we see that any integer greater than $a$ will
appear in the finite sequence $(i_{k+1},\ldots,i_{l-1})$ which is
impossible. This completes the proof of (2).
\end{proof}

For $\nu,\xi \in \mc{SP}$ such that $\nu\subseteq\xi$, the diagram
obtained by removing the subdiagram $\nu^*$ from the shifted diagram
$\xi^*$ is called a {\em skew shifted diagram} and denoted by
$\xi/\nu$. It is possible that a skew shifted diagram is realized by
two different pairs $\nu\subseteq\xi$ and $\tilde{\nu}
\subseteq\tilde{\xi}$.
\begin{example}
Assume $\xi=(5,3,2,1)$ and $\nu=(5,1)$. Then the corresponding skew
shifted Young diagram $\xi/\nu$ is
$$
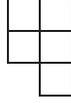
\young(:::\,\,,:::\,\,,::::\,)
$$
\end{example}

A filling by $1,2, \ldots, n$ in a skew shifted diagram $\xi/\nu$
with $|\xi/\nu|=n$ such that the entries strictly increase from left to
right along each row and down each column
is called a {\em standard skew shifted tableau}
of size $n$. Denote
\begin{align*}
\mc{W}'(n) &=\{\uli\in\Z_+^n \text{ satisfying the properties in
Proposition~\ref{prop:wtprop}}\},
  \\
\mc{F}(n)&=\{\text{standard skew shifted tableaux of size } n\}.
\end{align*}
\begin{proposition}\label{prop:bijection}
There exists a canonical bijection between $\mc{W}'(n)$ and
$\mc{F}(n).$
\end{proposition}
\begin{proof}
For $T\in\mc {F}(n)$, set
$$
c(T)=(c(T_1), c(T_2),\ldots, c(T_n))\in\Z_+^n,
$$
where $c(T_k)$ denotes the content of the cell occupied by $k$ in
$T$, for $1\leq k\leq n$. It is easy to show that  $c(T)\in
\mc{W}'(n)$. Then we define
\begin{align}
\Theta: \mc{F}(n)&\longrightarrow \mc{W}'(n),
 \qquad \Theta(T)= c(T).
  \label{map:bijection}
\end{align}
To show that $\Theta$ is a  bijection, we shall construct  by
induction on $n$ a unique tableau $T(\uli)\in\mc{F}(n)$ satisfying
$\Theta(T(\uli))=\uli$, for a given
$\uli=(i_1,\ldots,i_{n})\in\mc{W}'(n)$. If $n=1$, let
$T(\uli)\in\mc{F}(n)$ be a cell labeled by $1$ of content $i_1$.
Assume that $T(\uli')\in \mc F(n-1)$  is already defined, where
$\uli^{\prime}=(i_1,\ldots,i_{n-1})\in \mc{W}'(n-1)$. Set $u=i_n$.

\noindent\textit{Case 1}: $T(\uli')$ contains neither a cell of
content $u-1$ nor a cell of content $u+1$. Adding a new component
consisting of one cell labeled by $n$ of content $u$ to $T'$, we
obtain a new  standard tableau $T\in\mc F(n)$. Set $T(\uli)=T$.

\noindent\textit{Case 2}:  $T(\uli')$ contains cells of content
$u-1$ but no cell of content $u+1$. This implies
$u+1\notin\{i_1,\ldots,i_n\}$. Since $(i_1,\ldots,i_n)$ belongs to
$\mc{W}'(n)$, $u$ does not appear in $\underline{i}'$ and hence
$u-1$ appears only once in $\uli'$ by Proposition~\ref{prop:wtprop}.
Therefore there is no cell of content $u$ and only one cell denoted
by $A$ of content $u-1$ in $T(\uli')$. So we can add a new cell
labeled by $n$ with content $u$ to the right of $A$ to obtain a new
tableau $T$. Set $T(\uli)=T$.
 Observe that
there is no cell above $A$ in the column containing $A$ since there
is no cell of content $u$ in $T(\uli')$.  Hence
$T(\uli)\in\mc F(n)$.

\noindent\textit{Case 3}: $T(\uli')$ contains cells of content $u+1$
but no cell of content $u-1$. This implies
$u-1\notin\{i_1,\ldots,i_n\}$. Since $(i_1,\ldots,i_n)$ is in
$\mc{W}'(n)$,  $u$ does not appear in $\uli'$ and hence $u+1$
appears only once in $\uli'$ by Proposition~\ref{prop:wtprop}.
Therefore $T(\uli')$ contains only one cell denoted by $B$ of
content $u+1$ and no cell of content $u$. This means that there is
no cell below $B$ in $T(\uli')$. Adding a new cell labeled by $n$ of
content $u$ below $B$, we obtain a new tableau $T$. Set $T(\uli)=T$.
Clearly $T(\uli)\in\mc F(n)$.

\noindent\textit{Case 4}: $T(\uli')$ contains cells of contents
$u-1$ and $u+1$. Let $C$ and $D$ be the last cells on the diagonals
of content $u-1$ and $u+1$, respectively. Suppose that $C$ is
labeled by $s$ and $D$ is labeled by $t$. Then $i_s=u-1, i_t=u+1$,
and moreover $u-1\notin\{i_{t+1},\ldots,i_{n-1}\},
u+1\notin\{i_{s+1},\ldots,i_{n-1}\}$. Since $i_n=u$, by
Proposition~\ref{prop:wtprop} we see that
$u\notin\{i_{t+1},\ldots,i_{n-1}\}$ and
$u\notin\{i_{s+1},\ldots,i_{n-1}\}$.  This implies that there is no
cell below $C$ and no cell to the right of $D$ in $T(\uli')$.
Moreover $C$ and $D$ must be of the following shape
$$
\young(:C,D).
$$
Add a new cell labeled by $n$ to the right of $D$ and below $C$ to
obtain a new tableau $T$. Set $T(\uli)=T$.  Again it is clear that
$T(\uli)\in\mc F(n)$.
\end{proof}

\begin{example}\label{example1}
Suppose $n=5$. Then the standard skew shifted tableau corresponding
to
 $\uli=(1,2,0,1,0)\in\mc W'(5)$ is
$$
T(\uli)=\young(12,34,:5).
$$
\end{example}

\subsection{Classification of irreducible completely splittable $\Haff$-modules}

For a skew shifted diagram $\xi/\nu$ of size $n$, denote by
$\mc{F}(\xi/\nu)$ the set of standard skew shifted tableaux of shape
$\xi/\nu$, and form a vector space
$$
\widehat{U}^{\xi/\nu}=\bigoplus_{T\in\mc{F}(\xi/\nu)}\Cl_nv_{T}.
$$
Define
\begin{align}
x_iv_T&=\sqrt{q(c(T_i))}v_T, ~1\leq i\leq n,\label{xaction}\\
\begin{split}
 \label{saction}
s_kv_T&=\ds\Big(\frac{1}{\sqrt{q(c(T_{k+1}))}-\sqrt{q(c(T_k))}}
+\frac{1}{\sqrt{q(c(T_{k+1}))}+\sqrt{q(c(T_k))}}c_kc_{k+1}\Big)v_T \\
&+\sqrt{1-\frac{2(q(c(T_{k+1}))+q(c(T_k)))}{(q(c(T_{k+1}))-q(c(T_k)))^2}}v_{s_kT},~1\leq k\leq n-1,
\end{split}
\end{align}
where $s_kT$ denotes the tableau obtained by switching $k$ and $k+1$
in $T$ and $v_{s_kT}=0$ if $s_kT$ is not standard.

\begin{proposition}
Suppose $\xi/\nu$ is a skew shifted diagram of size $n$. Then
$\widehat{U}^{\xi/\nu}$ affords a completely splittable
$\Haff$-module under the action defined by~\eqref{xaction} and
\eqref{saction}.
\end{proposition}
\begin{proof}
We check the defining relations (\ref{braid}), (\ref{pc}),
(\ref{px1}), and (\ref{px2}).
It is routine to check (\ref{px1}), (\ref{px2}) and (\ref{pc}). It
remains to check the Coxeter relations (\ref{braid}).

It is clear by (\ref{px2}) that $s_ks_l=s_ls_k$ if $|l-k|>1$. We now
prove $s_k^2=1$. Let $T\in\mc{F}(\xi/\nu)$. A direct calculation
shows that if $s_kT$ is standard then
$$
s_k^2v_T=\Big(\frac{2(q(c(T_{k+1}))+q(c(T_k))}{(q(c(T_{k+1}))-q(c(T_k)))^2}\Big)v_T
 +\Big(1-\frac{2(q(c(T_{k+1}))+q(c(T_k))}{(q(c(T_{k+1}))-q(c(T_k)))^2}\Big)v_T=v_T.
$$
Otherwise, if $s_kT$ is not standard then $c(T_k)=c(T_{k+1})\pm1$,
and we have
$$
s_k^2v_T
=\Big(\frac{2(q(c(T_{k+1}))+q(c(T_k))}{(q(c(T_{k+1}))-q(c(T_k)))^2}\Big)v_T
 =v_T.
$$

So it remains to prove that $s_ks_{k+1}s_k=s_{k+1}s_ks_{k+1}$.
Fix $1\leq k\leq n-2$ and $T\in\mc{F}(\xi/\nu)$.
Let $a=q(c(T_k)), b=q(c(T_{k+1})), c=q(c(T_{k+2}))$.
If $c(T_{k})= c(T_{k+2})$, then by Corollary~\ref{cor:wtbraid} we
have $c(T_{k})=c(T_{k+2})=0,c(T_{k+1})=1$ and hence $a=c=0, b=2$.
Then $(a-b)^2=2(a+b)$.
By \eqref{saction}, we obtain that
\begin{align*}
s_kv_T=\frac{\sqrt{2}}{2}(1+c_kc_{k+1})v_T, \qquad
s_{k+1}v_T=\frac{\sqrt{2}}{2}(-1+c_{k+1}c_{k+2})v_T.
\end{align*}
Then one can check that $s_ks_{k+1}s_kv_T=s_{k+1}s_ks_{k+1}v_T$.

Now assume $c(T_k)\neq c(T_{k+2})$ and hence $a, b, c$ are distinct.
Then it suffices to show $\phi_k\phi_{k+1}\phi_kv_T=\phi_{k+1}\phi_k\phi_{k+1}v_T$ for the
intertwining elements $\phi_k, \phi_{k+1}$ defined via \eqref{eq:intertw}.
It is clear by~\eqref{saction}~that
$$\phi_rv_T=
\ds \sqrt{(q(c(T_{r+1}))-q(c(T_r)))^2-2(q(c(T_{r+1}))+q(c(T_r)))}v_{s_rT},
$$
if $s_rT$ is standard and $\phi_r v_T=0$ otherwise for $1\leq r\leq n-1$.
Now for our fixed $1\leq k\leq n-2$,
if one of $c(T_k)-c(T_{k+1})$,  $c(T_{k+1})-c(T_{k+2})$ and
$c(T_{k})-c(T_{k+2})$ is $\pm1$, then
$\phi_k\phi_{k+1}\phi_kv_T
=0=\phi_{k+1}\phi_k\phi_{k+1}v_T$.
Otherwise, one can check that
\begin{align*}
\phi_k\phi_{k+1}\phi_kv_T
&=\Big(\sqrt{(a-b)^2-2(a+b)}\sqrt{(b-c)^2-2(b+c)}
\sqrt{(a-c)^2-2(a+c)}\Big)v_T\\
&=\phi_{k+1}\phi_k\phi_{k+1}v_T.
\end{align*}
Therefore the proposition is proved.

\end{proof}

For  a skew shifted diagram $\xi/\nu$ of size $n$,  pick a  standard
skew shifted tableau $T^{\xi/\nu}$ of shape $\xi/\nu$. Observe that
the $\mc{P}_n^{\mathfrak{c}}$-module $\Cl_nv_{T^{\xi/\nu}}$ contains
an irreducible submodule $\mc{L}(\xi/\nu)$ which is isomorphic to
$L(c(T^{\xi/\nu}_1))\circledast
L(c(T^{\xi/\nu}_2))\circledast\cdots\circledast L(c(T^{\xi/\nu}_n))$
and moreover
\begin{equation}\label{eq:decomp7}
\Cl_nv_{T^{\xi/\nu}}\cong (\mc{L}(\xi/\nu))^{\oplus2^{\lfloor\frac{\ell(\xi)-\ell(\nu)}{2}\rfloor}}.
\end{equation}
Set
$$
U^{\xi/\nu}:=\sum_{\sigma\in \mf
S_n}\phi_{\sigma}\mc{L}(\xi/\nu)\subseteq \widehat{U}^{\xi/\nu},
$$
where $\phi_{\sigma}=\phi_{i_1}\phi_{i_2}\cdots\phi_{i_k}$ with a reduced expression
$\sigma=s_{i_1}s_{i_2}\cdots s_{i_k}$.

\begin{lemma}
Suppose $\xi/\nu$ is a skew shifted diagram of size $n$. Then
$U^{\xi/\nu}$ is a $\Haff$-submodule of $\widehat{U}^{\xi/\nu}$.
\end{lemma}
\begin{proof}
Clearly, $U^{\xi/\nu}$ is a $\mc{P}_n^{\mathfrak{c}}$-submodule of
$\widehat{U}^{\xi/\nu}$ by~\eqref{eq:xinter}~and~\eqref{eq:cinter}.
Let $\sigma\in \mf S_n$ and $z\in \mc{L}(\xi/\nu)$ be such that
$\phi_{\sigma}z\neq 0$.  Then
$$
\phi_k\phi_{\sigma}z=\big(s_k(x_k^2-x^2_{k+1})+(x_k+x_{k+1})+c_kc_{k+1}(x_k-x_{k+1})\big)\phi_{\sigma}z\in U^{\xi/\nu}.
$$
Meanwhile $(x_k^2-x_{k+1}^2)$ acts as a nonzero scalar on $\phi_{\sigma}z$ and hence $s_k\phi_{\sigma}z\in U^{\xi/\nu}$.
%
\end{proof}

The following theorem is due independently to \cite{HKS, Wan}. The
results of the paper of the first author \cite{Wan} were actually formulated and
established over any characteristic $p\neq 2$.
\begin{theorem}\label{thm:CSmodule}
Suppose $\xi/\nu$ and $\xi'/\nu'$ are skew shifted diagrams of size
$n$. Then
\begin{enumerate}
\item
$U^{\xi/\nu}$ is an irreducible $\Haff$-module.

\item
$U^{\xi/\nu}\cong U^{\xi'/\nu'}$ if and only if $\xi/\nu=\xi'/\nu'$.

\item
$\widehat{U}^{\xi/\nu}\cong (U^{\xi/\nu})^{\oplus
2^{\lfloor\frac{\ell(\xi)-\ell(\nu)}{2}\rfloor}}$.

\item
$\dim
U^{\xi/\nu}=2^{n-\lfloor\frac{\ell(\xi)-\ell(\nu)}{2}\rfloor}g^{\xi/\nu}$,
where $g^{\xi/\nu}$ denotes the number of standard skew shifted
tableaux of shape $\xi/\nu$.

\item
Every integral irreducible completely splittable $\Haff$-module is
isomorphic to $U^{\xi/\nu}$ for some skew shifted diagram $\xi/\nu$
of size $n$.
\end{enumerate}
\end{theorem}
\begin{proof}
Suppose $N$ is a nonzero submodule of $U^{\xi/\nu}$. Then
$N_{\uli}\neq 0$ for some $\uli=\sigma\cdot c(T^{\xi/\nu})$ and
$\sigma\in\mf S_n$, and hence $N_{c(T^{\xi/\nu})}\neq 0$. Observe
that $U^{\xi/\nu}_{c(T^{\xi/\nu})}\cong \mc L(\xi/\nu)$. This
implies that $N_{c(T^{\xi/\nu})}=U^{\xi/\nu}_{c(T^{\xi/\nu})}$ as
$L(\xi/\nu)$ is irreducible as $\mc{P}_n^{\mathfrak{c}}$-module.
Therefore $N=U^{\xi/\nu}$. This proves (1). If $U^{\xi/\nu}\cong
U^{\xi'/\nu'}$, then $T^{\xi/\nu}\in\mc F(\xi'/\nu')$. Hence,
$\xi/\nu=\xi'/\nu'$ and whence (2). Part~(3) follows by the
definition of $\widehat{U}^{\xi/\nu}$ and \eqref{eq:decomp7},  and
(4) follows from (3).

It remains to prove (5). Suppose $U$ is an integral irreducible
completely splittable $\Haff$-module  and let $u_{\uli}$ be a
non-zero weight vector of $U$. By Propositions~ \ref{prop:wtprop}
and \ref{prop:bijection}, there exists $T\in\mc{F}(n)$ such that
$\uli=c(T)$. Assume $T$ is of shape $\xi/\nu$. Observe that there
always exists a sequence of simple transpositions $s_{k_1},\ldots,
s_{k_r}$ such that $s_{k_j}\cdots s_{k_1}T$ is standard for $1\leq
j\leq r$ and $s_{k_r}\cdots s_{k_1}T=T^{\xi/\nu}$. Then it follows
by Lemma \ref{lem:interact} that $u_{\xi/\nu}:=\phi_{s_{k_r}}\cdots
\phi_{s_{k_1}}u_{\uli}$ is a non-zero weight vector of $U$ of weight
$c(T^{\xi/\nu})$. Hence $U_{c(T^{\xi/\nu})}\neq 0$ and it must
contain a $\mc{P}_n^{\mathfrak{c}}$-submodule $U'$ isomorphic to
$\mc L(\xi/\nu)$. Again by Lemma \ref{lem:interact},
$\sum_{\sigma\in\mf S_n}\phi_{\sigma}U'$ forms a $\Haff$-submodule
of $U$. Thus $U=\sum_{\sigma\in\mf S_n}\phi_{\sigma}U'.$ Let
$\tau:U'\rightarrow \mc L(\xi/\nu)$ be a
$\mc{P}_n^{\mathfrak{c}}$-module isomorphism. Then it is easy to
check that the map $ \overline{\tau}:\sum_{\sigma\in\mf
S_n}\phi_{\sigma}U'\rightarrow U^{\xi/\nu}$, which sends
$\phi_{\sigma}z$ to $\phi_{\sigma}\tau(z)$ for all $z\in U'$, is an
$\Haff$-module isomorphism.
\end{proof}

\subsection{The seminormal form construction for $\HC_n$}
When restricting Theorem \ref{thm:CSmodule} to the case of shifted
diagrams, we have the following.

\begin{theorem}  \label{th:seminormalHn}
$\{U^{\xi}|\xi\in\mc{SP}_n\}$ forms a complete set of non-isomorphic
irreducible $\HC_n$-modules. The Jucys-Murphy elements $J_1,
J_2,\ldots, J_n$ act semisimply on each $U^{\xi}$.
\end{theorem}

\begin{proof}
Consider the $\Haff$-modules $\widehat{U}^{\xi}$ and $U^{\xi}$, for
$\xi\in\mc{SP}_n$. For any standard shifted tableau $T$ of shape
$\xi$, we have $c(T_1)=0$ and hence $x_1v_T=0$. Hence the action of
$\Haff$ on $\widehat{U}^{\xi}$ and $U^{\xi}$ factors through to an
action of $\HC_n$ and $x_k$ acts as $J_k$ by \eqref{surjmap}, as
$\HC_n\cong \Haff/\langle x_1\rangle$. The theorem now follows from
Theorem~\ref{thm:CSmodule}.
\end{proof}

The construction of $\HC_n$-modules $U^\xi$ above can be regarded a
seminormal form for irreducible $\HC_n$-modules.
Theorem~\ref{th:seminormalHn} in different forms has been
established via different approaches in \cite{Naz, VS, HKS, Wan}.

\end{document}